\documentclass[11pt,a4paper,reqno]{amsart} 

\usepackage{latexsym,eqnarray}
\usepackage{subfig}
\usepackage{latexsym}
\usepackage{verbatim}
\usepackage{epsfig}
\usepackage{rotating}
\usepackage{amssymb}
\usepackage[T1]{fontenc}
\usepackage{afterpage}
\usepackage{color}
\usepackage{dsfont}
\usepackage{url}
\usepackage[utf8]{inputenc}
\usepackage{amssymb,amsfonts,amsmath,stmaryrd,bbm}
\usepackage[plainpages=false,pdfpagelabels,colorlinks=true,citecolor=blue,hypertexnames=false]{hyperref}
\graphicspath{{Figures/}}

\usepackage{pst-node}
\usepackage{tikz-cd} 

\newtheorem{Theorem}{Theorem}[section]
\newtheorem{Lemma}[Theorem]{Lemma}
\newtheorem{Proposition}[Theorem]{Proposition}
\newtheorem{Corollary}[Theorem]{Corollary}

\theoremstyle{definition}
\newtheorem{Definition}[Theorem]{Definition}
\newtheorem{Remark}[Theorem]{Remark}

\newtheorem{Example}[Theorem]{Example}

\newcommand{\N}{\mathbb{N}}

\newcommand{\Z}{\mathbb{Z}}

\newcommand{\GD}{\mathbb{D}}

\newcommand{\beq}{\begin{equation}}
\newcommand{\eeq}{\end{equation}}

\def\emm#1,{{\em #1}}

\newcommand{\tC}{\widetilde C}

\newcommand{\ts}{\tilde s}
\newcommand{\tw}{\tilde w}
\newcommand{\td}{\tilde d}
\newcommand{\tib}{\tilde b}
\newcommand{\tu}{\tilde u}

\newcommand{\TL}{\operatorname{TL}}
\newcommand{\FC}{\operatorname{FC}}

\newcommand{\Zd}{\mathbb{Z}[\delta]}
\newcommand{\tn}{\blacktriangle}
\newcommand{\tb}{\vartriangle}

\newcommand{\lott}{\mathcal{L}_{\tn}^{\tb}}

\newcommand{\Snakes}{\operatorname{Snakes}}

\newcommand{\PLR}{\mathcal{P}^{LR}_k(\Omega)}

\newcommand{\LRPZ}{\operatorname{LRP-Z}}
\newcommand{\ALT}{\operatorname{ALT}}
\newcommand{\LP}{\operatorname{LP}}
\newcommand{\RP}{\operatorname{RP}}
\newcommand{\LRP}{\operatorname{LRP}}
\newcommand{\ZZ}{\operatorname{ZZ}}
\newcommand{\PZZ}{\operatorname{PZZ}}

\newcommand{\cru}{{\mathfrak{R}}}
\newcommand{\del}{\operatorname{del}}

\newcommand{\Hc}{{\check{H}}}
\newcommand{\dc}{{d_{\check{e}}}}

\begin{document}
\title[Heaps, decorated diagrams and $\widetilde{C}$-Temperley-Lieb algebra]{Heaps reduction, decorated diagrams, and the affine Temperley-Lieb algebra of type $C$}

\author[R. Biagioli]{Riccardo Biagioli}
\address{R. Biagioli : Dipartimento di Matematica, Universit\`a di Bologna, Piazza di Porta San Donato 5, 40126 Bologna, Italy}
\email{riccardo.biagioli2@unibo.it}

\author[G. Calussi]{Gabriele Calussi}
\address{G. Calussi, G. Fatabbi : Dipartimento di Matematica, Universit\`a degli Studi di Perugia, Via Vanvitelli 1, 06123 Perugia, Italy}
\email{giuliana.fatabbi@unipg.it, gabriele.calussi@gmail.com}

\author[G. Fatabbi]{Giuliana Fatabbi}

\thanks{The first author thanks GNSAGA of INdAM and the Universit\`a degli Studi di Perugia for their partial support while visiting the third author.}

\begin{abstract}
In this paper we propose a combinatorial framework to study a diagrammatic representation of the affine Temperley-Lieb algebra of type $\tC_n$ introduced by Ernst. In doing this, we define two procedures, a decoration algorithm on diagrams and a reduction algorithm on heaps of independent interest. Using this approach, an explicit algorithmic description of Ernst representation map is provided from which its  faithfulness can be deduced. We also give a construction of the inverse map.  
\end{abstract}
\date{\today}


\maketitle


\section*{Introduction}\label{sec:intro}

The Temperley-Lieb algebra is a  finite dimensional associative algebra which was first introduced by Temperley and Lieb in 1971 in \cite{TemperleyLieb} and since then it has been  playing a central role in several domains of mathematical physics, mainly in the statistical physics description of lattice models and in conformal field theory.
It has been natural for mathematicians to consider it as an abstract algebra over the complex field and to study its representation theory.
In~\cite{HKAUFFMAN1987395,Penrose71angularmomentum:} Kauffman  and Penrose viewed the Temperley-Lieb algebra  as a diagram
algebra, i.e., an associative algebra with a basis made of certain diagrams and a multiplication given by the application of local combinatorial rules to the diagrams.
On the other hand, in \cite{Jones1983}, Jones independently found the Temperley-Lieb algebra as an algebra defined by generators and relations and in \cite{JonesAnnals} he showed that it occurs naturally as a quotient of the Hecke algebra of type $A$. The realization of
the Temperley-Lieb algebra as a Hecke algebra quotient was generalized by  Graham in~\cite{Graham}, where he defined the so-called generalized Temperley-Lieb algebra $\TL(\Gamma)$ for any Coxeter system of type $\Gamma$, and showed that it admits a monomial basis indexed by the fully commutative elements (FC) of the corresponding Coxeter group. 
This gave rise to the problem of finding analogous diagrammatic descriptions of $\TL(\Gamma)$ for an arbitrary Coxeter system. In a series of paper, \cite{Green_TLBD}, \cite{Green_TLE} and \cite{Green_H}, Green defined a diagram calculus in finite Coxeter types $B, D, E$ and $H$. In the affine case, Fan and Green, in~\cite{FanGreen_Affine}, gave a realization of $\TL(\widetilde{A}_{n+1})$ as a diagram algebra on a cylinder, while Ernst, in~\cite{ErnstDiagramI,ErnstDiagramII}, interpreted $\TL(\tC_n)$ as an algebra of decorated diagrams. More precisely, Ernst introduced an infinite associative algebra denoted by $\GD(\widetilde{C}_{n})$, whose elements are the classical diagrams with decorated edges. He provided an explicit basis for $\GD(\widetilde{C}_{n})$, consisting of admissible diagrams, and defined an algebra homomorphism $\tilde{\theta}\colon \TL(\tC_n)\longrightarrow \GD(\widetilde{C}_{n})$ mapping any monomial basis element to an admissible diagram. One of his main results is that the map $\tilde{\theta}$ is a faithful representation.  

In the finite case, the diagrammatic representations of $\TL(\Gamma)$ previously mentioned are the faithful and this is proved by a counting argument. In type $\tC_n$, Ernst proof of injectivity requires many preliminary technical results and it is based on a classification of the so-called non-cancellable elements. In the conclusion of his paper, Ernst himself wonders whether a shorter proof of the injectivity exists. 

In this paper, motivated by his question, we propose a combinatorial framework to study  and give more insight on Ernst representation. Our approach is more algorithmic and it is based on a classification of the fully commutative elements in type $\tC_n$ given in \cite{BJN_FC}. We define several constructions on certain posets called heaps that encode the elements of the monomial basis of the $\TL(\tC_n)$ algebra. 
Our main result is an explicit combinatorial description of Ernst map $\tilde{\theta}$ from which its injectivity follows quite easily. This new characterization is based on two procedures, a reduction algorithm $\mathfrak{R}$ on heaps and a decoration algorithm on diagrams. Our technique can be briefly illustrated by the following schema 

\begin{center}
\begin{tikzcd}
\FC(\tC_n) \arrow{r}{\tilde{\theta}} 
  \dar[color=black]{\mathfrak{R}}
        & \GD(\widetilde{C}_{n})\\ 
  \FC(A_{n+1}) \rar[color=black]{\theta}
& \GD({A}_{n+1})\arrow[dashed]{u}{}
\end{tikzcd}
\end{center}
where the image through $\tilde{\theta}$ of the monomial basis element of $\TL(\tC_n)$ indexed by a FC heap $H$, can be obtained by adding loops and decorations to the diagram corresponding to the monomial basis element of $\TL(A_{n+1})$ indexed by $\mathfrak{R}(H)$.

The reduction algorithm on heaps can be generalized to other Coxeter types. In a forthcoming work~\cite{BFS}, we plan to extend such diagram calculus to the remaining classical affine types $\widetilde{B}$ and $\widetilde{D}$, by exploiting the algorithms and the other techniques developed here. 

This paper is organized as follows. In Section~\ref{sub:fullycomm}, we recall definitions and basic results on Coxeter groups, heaps theory, fully commutative elements and the classification of FC elements in type $\tC_n$ in terms of heaps given in \cite{BJN_FC}. In Sections~\ref{sec:TL-algebra} and \ref{sec:diagrammi}, we recall the construction of the algebra of decorated diagrams given by Ernst in~\cite{ErnstDiagramI}. In Section~\ref{sec:algoCA}, we associate to any FC element in $\tC_n$ a unique element in $A_{n+1}$ by means of a reduction algorithm on heaps. In Sections~\ref{sec:snakes_paths} and \ref{sec:algdec}, we distinguish a set of paths, called snakes paths, that cover an alternating heap and that will be used to define a decoration algorithm.
In Section~\ref{sec:inj}, all the introduced material is combined to exhibit an alternative proof of the injectivity of the map $\tilde{\theta}$. Finally, in Section~\ref{surjectivity}, given an admissible diagram $d$ in $\GD(\widetilde{C}_{n})$, we present an algorithm that recovers the unique element in $\FC(\tC_n)$ that indexes the inverse image of $d$ through $\tilde{\theta}$. 

\section{Fully commutative elements in Coxeter groups}
\label{sub:fullycomm}

Let $M$ be a square symmetric matrix indexed by a finite set $S$, satisfying $m_{ss}=1$ and, for $s\neq t$, $m_{st}=m_{ts}\in\{2,3,\ldots\}\cup\{\infty\}$. The {\em Coxeter group} $W$ associated with the  {\em Coxeter matrix} $M$ is defined by generators $S$ and relations $(st)^{m_{st}}=1$ if $m_{st}<\infty$. These relations can be rewritten more explicitly as $s^2=1$ for all $s$, and \[\underbrace{sts\cdots}_{m_{st}}  = \underbrace{tst\cdots}_{m_{st}},\]  where $m_{st}<\infty$, the latter being called \emph{braid relations}. When $m_{st}=2$, they are simply \emph{commutation relations} $st=ts$. 

The {\em Coxeter graph} $\Gamma$ associated to the {\em Coxeter system} $(W,S)$ is the graph with vertex set $S$ and, for each pair $\{s,t\}$ with $m_{st}\geq 3$, an edge between $s$ and $t$ labeled by $m_{st}$. When $m_{st}=3$ the edge is usually left unlabeled since this case occurs frequently. Therefore non adjacent vertices correspond precisely to commuting generators.
 
\begin{figure}[h]
\centering
\includegraphics[width=0.6\linewidth]{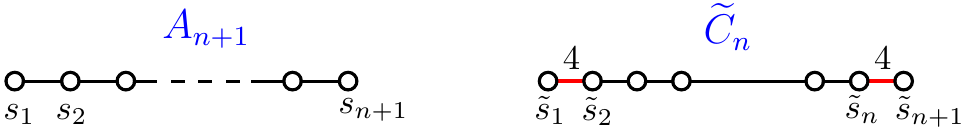}
\caption{Coxeter graphs of types $A_{n+1}$ and $\tC_{n}$.}\label{Ctilde-renamed}
\end{figure}

For $w\in W$, the {\em length} of $w$, denoted by $\ell(w)$, is the minimum length $l$ of any expression $w=s_1\cdots s_l$ with $s_i\in S$. These expressions of length $\ell(w)$ are called \emph{reduced}, and we denote by $\mathcal{R}(w)$ the set of all reduced expressions of $w$, which will be denoted with bold letters. A fundamental result in Coxeter group theory, sometimes called the {\em Matsumoto property}  states that any expression in $\mathcal{R}(w)$ can be obtained from any other one using only braid relations (see for instance~\cite{Humphreys}). The notion of full commutativity is a strengthening of this property.

\begin{Definition}
    \label{defi:FC}
 An element $w$ is \emph{fully commutative} (FC) if any reduced expression for $w$ can be obtained from any other one by using only commutation relations.
\end{Definition}

The following characterization of FC elements, originally due to Stembridge, is particularly useful in order to test whether a given element is FC or not.

\begin{Proposition} [Stembridge \cite{St1}, Prop. 2.1]
\label{prop:caracterisation_fullycom}
An element $w\in W$ is fully commutative if and only if for all $s,t$ such that $3\leq m_{st}<\infty$, there is no expression in $\mathcal{R}(w)$ that contains the factor $\underbrace{sts\cdots}_{m_{st}}$.
\end{Proposition}

Therefore an element $w$ is FC if all reduced expressions avoid all braid relations; since, by definition, $\mathcal{R}(w)$ forms a commutation class, the concept of heap helps to capture the notion of full commutativity. We briefly describe a way to define the above mentioned heap and its relations with full commutativity, for more details see for instance~\cite{BJN_FC} and the references cited there.

Let $(W,S)$ be a Coxeter system with Coxeter graph $\Gamma$, and fix an expression $\mathbf{w}=s_{a_1}\cdots s_{a_l}$ with $s_{a_j} \in S$. Define a partial ordering $\prec$ on the index set $\{1,\ldots, l\}$ as follows: set $i\prec j$ if $i<j$ and $s_{a_i}$, $s_{a_j}$ do not commute, and extend it by transitivity. We denote this poset together with the labeling map $\epsilon:i\mapsto s_{a_i}$ by $H({\mathbf{w}})$ and we call it a {\em labeled heap of type $\Gamma$} or simply a {\em heap}. Heaps are well-defined up to commutation classes~\cite{ViennotHeaps}, that is, if ${\bf w}$ and ${\bf w'}$ are two reduced expressions for $w \in W$, that are in the same commutation class, then the corresponding labeled heaps are isomorphic. This means that there exists a poset isomorphism between them which preserves the labels. Therefore, when $w$ is FC  we can define $H(w):=H(\mathbf{w})$, where ${\bf w}$ is any reduced expression for $w$. Heaps of this form will be called {\em FC heaps}. Another important feature for FC heaps is that the linear extensions of $H(w)$ are in bjiection with the reduced expressions of $w$, see \cite[Proposition 2.2]{St1}. 

Given a heap $H$ and a subset $I\subset S$, we denote by $H_{I}$ the {\em subheap} induced by all elements of $H$ with labels in $I$ (see~\cite[\S 2]{ViennotHeaps}).
\smallskip

In the Hasse diagram of $H({\bf w})$, elements with the same labels are drawn in the same column. Moreover, as in \cite{ErnstDiagramI}, we draw heaps from top to bottom, namely the entries on the top of $H({\bf w})$, correspond to generators occurring on the left of ${\bf w}$.

\begin{Example}
Consider $w=s_1s_3s_2s_4s_3 \in \FC(A_4)$. Its heap is represented in Figure~\ref{ExampleA}, left. Its set of reduced expressions, obtained by listing the labels of each linear extension of $H(w)$ is $\mathcal{R}(w)=\{s_1s_3s_2s_4s_3, s_1s_3s_4s_2s_1, s_3s_1s_2s_4s_3, s_3s_1s_2s_4s_3, s_3s_4s_1s_2s_3 \}$. The subheap $H_I$ corresponding to the subset $I=\{s_2,s_3\}$ is represented in Figure~\ref{ExampleA}, right. Notice that the corresponding group element $w_I=s_3s_2s_3$ is not FC.
\end{Example}

\begin{figure}[!ht]
\begin{center}
\includegraphics[scale=0.7]{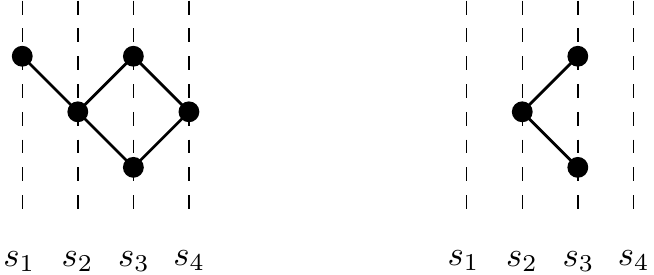}
\caption{A FC heap of type $A_4$ and its subheap $H_{\{s_2,s_3\}}$.}\label{ExampleA}  
\end{center}
\end{figure}

\begin{Definition}
\label{defi:alternating}
Let $(W,S)$ be a Coxeter system, $w \in \FC(W)$, and $H:=H(w)$. We say that $H$ is {\em alternating} if for each non commuting generators $s,t$ in $S$, the chain $H_{\{s,t\}}$ has alternating labels $s$ and $t$ from top to bottom.
\end{Definition}

Note that if $H(w)$ is alternating, then any reduced expression $\mathbf{w}$ of $w$ is \emph{alternating} in the sense that for each  non commuting generators  $s,t \in S$, the occurrences of $s$ alternate with those of $t$ in $\mathbf{w}$. In this case we also say that $w \in \FC(W)$ is alternating. 

For example, the heap on the left of Figure~\ref{ExampleA} is alternating: indeed, the subheaps corresponding to all pairs of noncommuting generators  $\{s_1,s_2\}$, $\{s_2,s_3\}$, $\{s_3,s_4\}$ are respectively the alternating chains $s_1s_2$, $s_3s_2s_3$, and $s_3s_4s_3$. Here we identify a subheap with the sequence of the labels of its vertices, obtained by reading such vertices from top to bottom. We will often use this identification in this paper.
\medskip

We now recall the descriptions of FC heaps corresponding to the Coxeter graphs of types $A_{n+1}$ and $\tC_n$ given for instance in~\cite{BJN_FC}. 

\begin{Theorem}[Classification of FC heaps in type $A_{n+1}$]\label{heaps-typeA}
A heap $H$ of type $A_{n+1}$ is fully commutative if and only if in $H$ 
\begin{itemize}
\item[(a)] There is at most one occurrence of $s_1$ ({\em resp.} $s_{n+1}$);
\item[(b)] For each $i \in \{1,\ldots, n+1\}$, the elements with labels $s_i, s_{i+1}$ form an alternating chain. 
\end{itemize}
Equivalently, each connected component of $H$ is alternating and starts and ends with a single vertex. 
\end{Theorem}

To classify FC heaps of type $\tC_n$ we need to introduce a couple of notations. A {\em peak} is a heap of the form:
$$P_{\rightarrow}(\ts_i):=H (\ts_i \ts_{i+1}\dots \ts_{n} \ts_{n+1} \ts_{n}\dots \ts_{i+1}\ts_i) \ \mbox{or} \ P_{\leftarrow}(\ts_i):=H(\ts_i \ts_{i-1}\dots \ts_{2} \ts_1 \ts_{2}\dots \ts_{i-1}\ts_i).$$ If $H$ is a heap  of type $\tC_n$ and $i\in\{2,\ldots,n\}$, then $H_{\{\leftarrow \ts_i\}}$ ({\em resp.} $H_{\{\rightarrow \ts_i\}}$) denotes the subheap of $H$ induced by the elements with labels $\{\ts_1,\ldots,\ts_{i-1},\ts_i\}$ ({\em resp.} $\{\ts_i,\ts_{i+1},\ldots,\ts_{n+1}\}$).

\begin{Definition}\label{def:famillesCtilde}  We define five families of heaps of type $\tC_n$.
\smallskip

\noindent \textbf{(ALT)} {\em Alternating}. $H\in (\ALT)$ if it is alternating in the sense of Definition~\ref{defi:alternating}.
\smallskip

\noindent \textbf{(ZZ)} {\em Zigzags}. $H\in $ (ZZ) if $H=H(\mathbf{\tw})$ where $\mathbf{\tw}$ is a finite factor of the infinite word $\left(\ts_1\ts_2\cdots \ts_{n} \ts_{n+1} \ts_{n}\cdots \ts_2\ts_1\right)^\infty$ such that $|H_{s_i}| \geq 3$ for at least one $i\in\{2,\dots,n\}$.
\smallskip

 \noindent \textbf{(LP)}  {\em Left-Peaks}. $H\in (\LP)$ if there exists $j_\ell \in\{2,\dots,n\}$ such that:
\begin{enumerate}
\item  $H_{\{\leftarrow \ts_{j_\ell}\}}=P_{\leftarrow}(\ts_{j_\ell})$;
\item There is no $\ts_{j{_\ell}+1}$-element between the two $\ts_{j_\ell}$-elements;
\item  $H_{\{{\ts}_{j_\ell}\rightarrow \}}$ is alternating when one $\ts_{j_\ell}$-element is deleted from it.
\end{enumerate}
\smallskip

\noindent \textbf{(RP)} {\em Right-Peaks}. $H\in $ (RP) if there exists ${j_r} \in\{2,\dots,n\}$ such that:
\begin{enumerate}
\item $H_{\{\ts_{j_r} \rightarrow\}}= P_{\rightarrow}(\ts_{j_r})$;
\item There is no $\ts_{{j_r}-1}$-element between the two $\ts_{j_r}$-elements; 
\item $H_{\{\leftarrow {\ts}_{{j_r}}\}}$ is alternating when one $\ts_{j_r}$-element is deleted.
\end{enumerate}
\smallskip

\noindent \textbf{(LRP)} {\em Left-Right-Peaks}. $H\in $ (LRP) if there exist $2\leq {j_\ell}<{j_r} \leq n$ such that:
\begin{enumerate}
\item  $\LP(1), \LP(2), \RP(1), \RP(2)$ hold;
\item  $H_{\{\ts_{{j_\ell}},\ldots,\ts_{{j_r}}\}}$ is  alternating when both a $\ts_{j_\ell}$- and a $\ts_{j_r}$-element are deleted.
\end{enumerate}
\end{Definition}

\begin{figure}[t]
\includegraphics[width=0.9\textwidth]{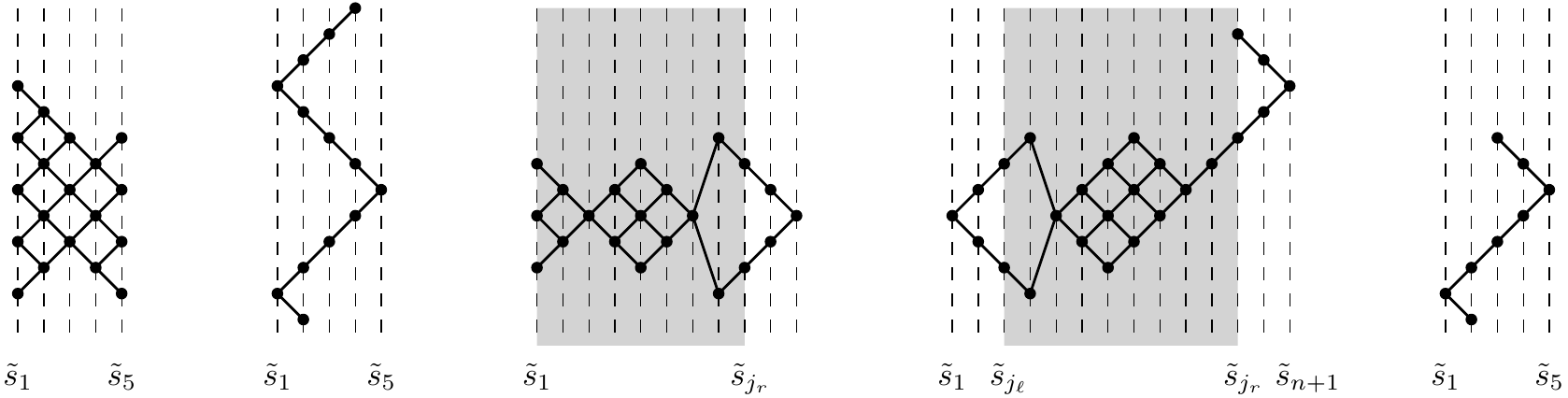}
\caption{Examples  of heaps of type  $\tC_n$ in families $(\ALT), (\ZZ), (\RP), (\LRP)$ and $(\PZZ)$.}
\end{figure}

\begin{Remark}
\label{rem:family}
The condition $|H_{s_i}|\geq 3$ in the definition of (ZZ) is only there to ensure that the families are disjoint. In families $(\LP), (\RP), (\LRP),$ the indices $j_\ell$ and $j_r$ are {\em uniquely determined}. 
\end{Remark}

We can now state the classification theorem in type $\tC_n$, see~\cite[Theorem 3.4]{BJN_FC}.

\begin{Theorem}[Classification of FC heaps of type $\tC_n$]\label{theo:affineCfamilles}
A heap of type $\tC_n$ is  fully commutative if and only if it belongs to one of the five families $(\ALT), (\ZZ), (\LP), (\RP), (\LRP)$.
\end{Theorem}

For our aims, it will be useful to give a different partition of the set of FC elements. This is why we consider an additional subfamily of $(\LRP)$. 

\begin{Definition} {\bf (PZZ)} {\em Pseudo Zigzags}. $H\in (\PZZ)$ if $H=H(\mathbf{\tw})$ where $\mathbf{\tw}$ is a finite factor of the infinite word $\left(\ts_1 \ts_2\cdots \ts_{n}\ts_{n+1}\ts_{n}\cdots \ts_2\ts_1\right)^\infty$ such that $|H_{\ts_i}| < 3$ for all $i\in\{2,\dots,n\}$ and contains the two factors $\ts_2\ts_1\ts_2$ and $\ts_{n}\ts_{n+1}\ts_{n}$.
\end{Definition}

In the rest of the paper, we will say that $\tw \in \FC(\tC_n)$ is in $(\ALT), (\ZZ), (\RP)$, $(\LRP)$ or $(\PZZ)$ if and only if the corresponding heap $H(\tw)$ is. 
\section{Temperley-Lieb algebras}\label{sec:TL-algebra}

In this section we recall the presentation of the classical Temperley-Lieb algebra and of the generalized Temperley-Lieb algebra of type $\tC_n$, specializing the result of Green \cite[Proposition 2.6]{Green-Star} that holds for any Coxeter type.  
\medskip

The {\em Temperley-Lieb algebra of type $A_{n+1}$}, denoted TL$(A_{n+1})$, is the unital $\Z[\delta]$-algebra generated by $\{b_1,\ldots, b_{n+1}\}$ with defining relations:
\begin{enumerate}
\item[(a1)] $b_i^2=\delta b_i$ for all $i$;
\item[(a2)] $b_ib_j=b_jb_i$ if $|i-j|>1$;
\item[(a3)] $b_ib_jb_i=b_i$ if $|i-j|=1$.
\end{enumerate}

The {\em Temperley-Lieb algebra of type $\tC_n$}, denoted TL$(\tC_n)$, is the unital $\Z[\delta]$-algebra generated by $\{\tib_1,\ldots, \tib_{n+1}\}$ with defining relations:
\begin{enumerate}
\item[(c1)] $\tib_i^2=\delta \tib_i$ for all $i \in \{1,2,\ldots, n+1\}$;
\item[(c2)] $\tib_i\tib_j=\tib_j\tib_i$ if $|i-j|>1$ and $i,j \in \{1,2,\ldots, n+1\}$;
\item[(c3)] $\tib_i\tib_j\tib_i=\tib_i$ if $|i-j|=1$ and $i,j \in \{2,\ldots, n\}$;
\item[(c4)] $\tib_i\tib_j\tib_i\tib_j=2\tib_i\tib_j$ if $\{i,j\}=\{1,2\}$ or $\{i,j\}=\{n,n+1\}$.
\end{enumerate}

For $w=s_{i_1}\cdots s_{i_k} \in \FC({A_{n+1}})$ and $\tilde{w} =\ts_{i_1}\cdots \ts_{i_k} \in \FC({\tC_n})$, we define 
$$b_w:=b_{i_1}\cdots b_{i_k}, \quad \mbox{and} \quad  \tib_{\tw}:=\tib_{i_1}\cdots \tib_{i_k}$$
the associated elements in TL$(A_{n+1})$ and $\TL(\tC_n)$. Since $w$ and $\tw$ are FC, $b_w$ and $\tib_{\tw}$ do not depend on the chosen reduced expressions. The sets $\{b_w \mid w \in \FC(A_{n+1})\}$ and $\{\tib_{\tw} \mid \tw \in \FC(\tC_{n})\}$ form $\Z[\delta]$-bases for TL$(A_{n+1})$ and $\TL(\tC_n)$  respectively, called the {\em monomial bases}. 

It is well-known that TL$(A_{n+1})$ has a faithful representation in terms of a diagram algebra. The main result of Ernst Ph.D. thesis is a generalization of such algebra of diagrams that turns out to be a faithful representation of the algebra TL$(\tC_{n+1})$.

\section{Undecorated, decorated and admissible diagrams}\label{sec:diagrammi}

In this section, following Sections 3--5 of \cite{ErnstDiagramI} and Section 4 of \cite{ErnstDiagramII}, we give a survey on diagram algebras, undecorated and decorated diagrams, and define the main object of this paper, the algebra $\GD(\tC_n)$. 

\subsection{Undecorated diagrams.}
The {\em standard $k$-box} is a rectangle with $2k$ marks points, called
    nodes (or vertices) labeled as in Figure~\ref{standard_box}. We will refer to the top of the rectangle as the north
face and to the bottom as the south face.  If we  think of the standard $k$-box as being embedded in the plane with the origin in the  lower left corner of the rectangle then we think  each node $i$ (respectively, $i'$) in  the point $(i, 1)$ (respectively, $(i, 0)$).

\begin{figure}[h]
\includegraphics[width=0.8\textwidth]{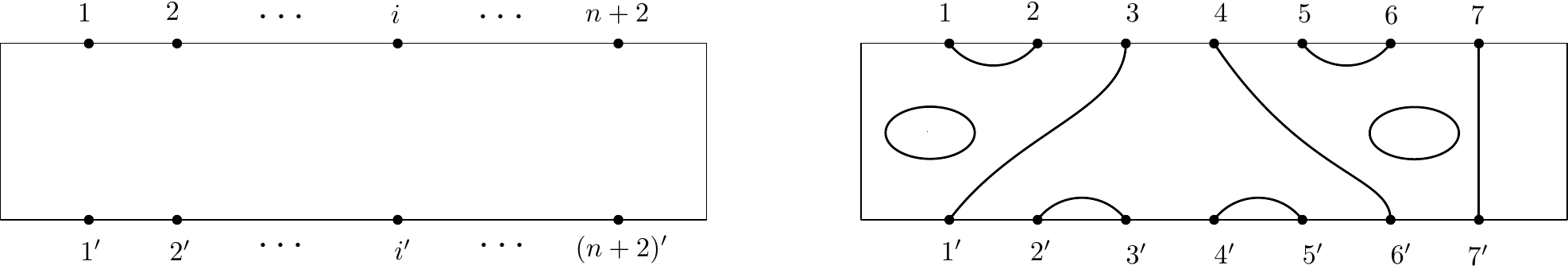}
\caption{A standard $(n+2)$-box and a pseudo 7-diagram.}\label{standard_box}
\end{figure}
A {\em concrete pseudo $k$-diagram} consists of a finite number of disjoint plane curves,  called edges,
embedded in the standard $k$-box. A plane curve  either meets  the
box transversely and its endpoints are the nodes of the box  or it is disjoint from the box.  An edge may be a closed  curve (isotopic to circles) in which case we refer to it as a {\em loop}. If an  edge is a curve  that joins node $i$ in the north face to node $j'$ in the
south face, then it  is called a {\em propagating edge} from $i$ to $j'.$  If a propagating edge joins $i$ to $i'$ then we will call it a
 {\em vertical propagating edge}. If an edge is not propagating, loop edge or otherwise, it will be called {\em non-propagating}. Given a diagram $d$, we denote by ${\bf a}(d)$ the number of non-propagating edges in the north face of $d$.
A {\em pseudo $k$-diagram} is defined to be an equivalence class of equivalent concrete pseudo $k$-diagrams modulo the isotopy equivalence and we denote the set of pseudo $k$-diagrams by $T_k(\emptyset).$ 

Let $R$ be a commutative ring with $1$. Following \cite{Jones}, we define the associative algebra $\mathcal{P} _k(\emptyset)$ as the free $R$-module with basis $T_k(\emptyset)$, and  product $d'd$ of $d', d \in T_k(\emptyset)$ as the diagram obtained by placing $d'$ on top of $d$, so that node $i' $ of $d'$ coincides with node $i$ of $d$, rescaling vertically by a factor of $1/2$ and then applying the appropriate translation to recover a standard $k$-box.  
 
 \begin{Definition}
 Let $\GD(A_{n+1})$ be the associative $\mathbb{Z}[\delta ]$-algebra equal to the quotient of $ \mathcal{P}_{n+2}(\emptyset)$ by the relation depicted in Figure~\ref{loop_A}.
\begin{figure}[h]
\includegraphics[width=0.12\textwidth]{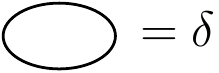}
\caption{The defining relation of $\GD(A_{n+1})$.}\label{loop_A}
\end{figure}
\end{Definition}

The following theorem collects three classical results. 
\begin{Theorem}\label{prop:thetaA} \
\begin{itemize}
\item[(a)] $\GD(A_{n+1})$ is equal to the  $\mathbb{Z}[\delta]$-module  having  the loop-free diagram   of $T_{n+2}(\emptyset)$ as a basis;
\smallskip

\item[(b)] $\GD(A_{n+1})$ is generated as $\mathbb{Z}[\delta]$-algebra by the set of {\em simple diagrams} $\{d_1,\ldots,d_{n+1}\}$, see Figure~\ref{diagrammi-semplici-1}, subjected to the  relations (a1)-(a2)-(a3) given in \S~\ref{sec:TL-algebra}; \smallskip
 
\item[(c)] the map $\theta: \TL(A_{n+1})\longrightarrow \GD(A_{n+1})$ determined by $b_i \longmapsto d_i$ is a well defined $\Z[\delta]$-algebra isomorphism that sends the monomial basis of $\TL(A_{n+1})$ to the loop-free diagrams basis of $\GD(A_{n+1})$. 
\end{itemize}
\end{Theorem}

\begin{figure}[h]
\includegraphics[width=0.4\textwidth]{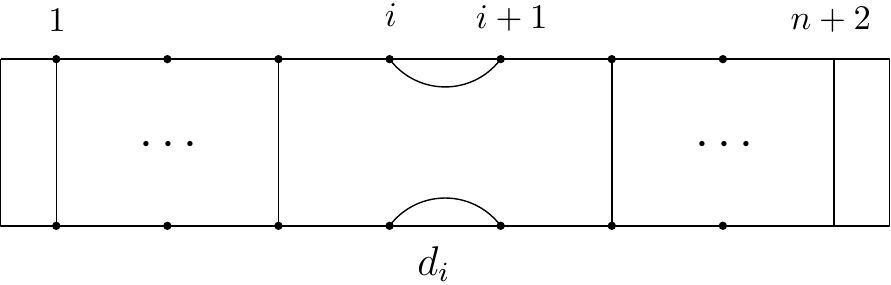}
\caption{The simple diagram $d_i$.}\label{diagrammi-semplici-1}
\end{figure}

\begin{Remark}\label{rem:TLA}
More precisely, point (c) above tells us that if $s_{i_1}\cdots s_{i_k}$ is a reduced expression of $w \in \FC({A_{n+1}})$, then the diagram 
\begin{equation}\label{def:dw}
d_w:=\theta(b_w)=d_{i_1}\cdots d_{i_k}
\end{equation}
is a loop-free diagram in $\GD(A_{n+1})$. Viceversa, for each loop-free diagram $d \in \GD(A_{n+1})$ there exists a unique $w \in \FC(A_{n+1})$ such that $d_w=d$.
\end{Remark}

\subsection{Decorated diagrams}
 Ernst, in~\cite{ErnstDiagramI}, introduced an algebra of diagrams which plays the role of $\GD(A_{n+1})$ in the case of the Coxeter system of type $\tC_{n}$. In what follows, we recall its definition and we highlight  particular aspects which will be important in our analysis. We leave the interested reader to consult~\cite[\S3.2]{ErnstDiagramI} for further studies.  
\smallskip

We fix the set $\Omega:=\{\bullet, \blacktriangle, \circ, \vartriangle\}$ and we will refer to each element of $\Omega$ as a {\em decoration}.  We will adorn the edges of a pseudo diagram $d \in T_{k}(\emptyset)$ with the elements of the free monoid $\Omega^{*}$. Sometimes we split such sequences into non empty {\em blocks} ${\bf b}=x_1x_2 \cdots x_r$, $x_i \in \Omega$. If $e$ is a non-loop edge, we adopt  the convention that we can read off the sequence of decorations as we traverse $e$, from $i$ to $j'$
if $e$ is propagating,  or from $i$ to $j$ (resp. $i'$ to $j'$) with $i < j$ (resp. $i'< j'$) if $e$ is non-propagating.
In particular, when we say the first (resp. last) decoration on a edge $e$, we mean the first (resp. last) one we read. If $e$ is a loop edge we consider two sequences of decorations equivalent if one can be changed into the other or its opposite by any cyclic permutation. 
Now we list the rules that we use to decorate the edges of $d$.

If ${\bf a}(d)=0$, then we do not adorn any of the edges of $d$. 
\smallskip

If ${\bf a}(d)\neq 0$, we might adorn the edges of $d$ with the elements of   $\Omega^{*}$ in such a way that all decorated edges can be deformed so as to take the decorations $\bullet, \blacktriangle$ to the left wall of the standard box and the decorations $\circ, \vartriangle$ to the right wall simultaneously without crossing any other edge. Furthermore, if $e$ is a non-propagating edge of $d$, the sequence of its decorations has to be considered as a unique block. 

\begin{figure}[h]
\includegraphics[width=0.8\textwidth]{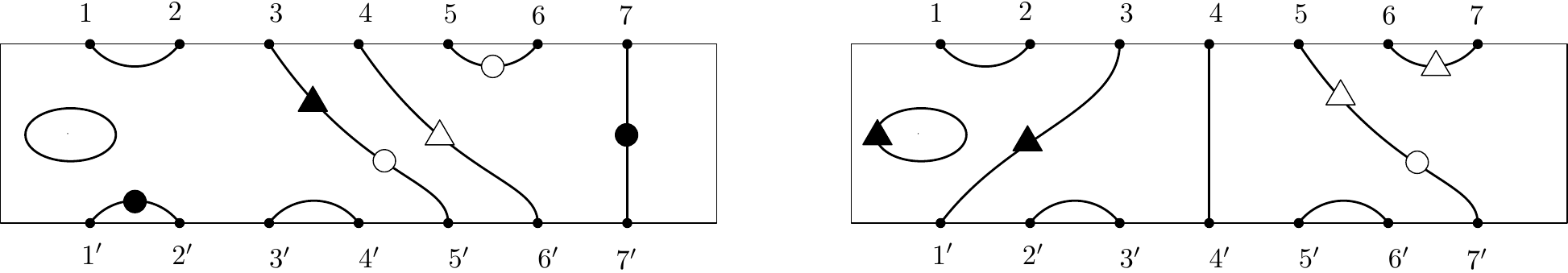}
\caption{The left diagram has not allowable decorations; the right one is in $T^{LR}_7(\Omega)$ but it  is not admissible.}\label{decorated-Ex1}
\end{figure}

If ${\bf a}(d)=1$, in addition, we need to introduce the notion of {\em vertical position} of a decoration which simply is its $y$-value  in the $xy$-plane. Even if the notions of block and vertical position make sense for any diagram with decorated edges, they are relevant only in this case. In fact, if  ${\bf a}(d)=1$ and $e$ is propagating, then we allow $e$ to be decorated subject to the following constraints:
\begin{itemize}
  \item[(a)] All decorations on propagating edges must have vertical position lower (resp. higher) than the vertical position of decorations occurring on the (unique) non-propagating edge in the north face (resp. south face) of the diagram.
  \item[(b)]  If $\mathbf{b}$ is a block of decorations occurring on $e$, then no other decorations occurring on any
other propagating edges may have vertical position in the range of vertical positions
that $\mathbf{b}$ occupies.
  \item[(c)]  If $\mathbf{b}_i$ and $\mathbf{b}_{i+1}$ are two adjacent blocks occurring on $e$, then they may be conjoined to
form a larger block only if the previous requirements are not violated.
\end{itemize}

\begin{figure}[h]
\includegraphics[width=0.8\textwidth]{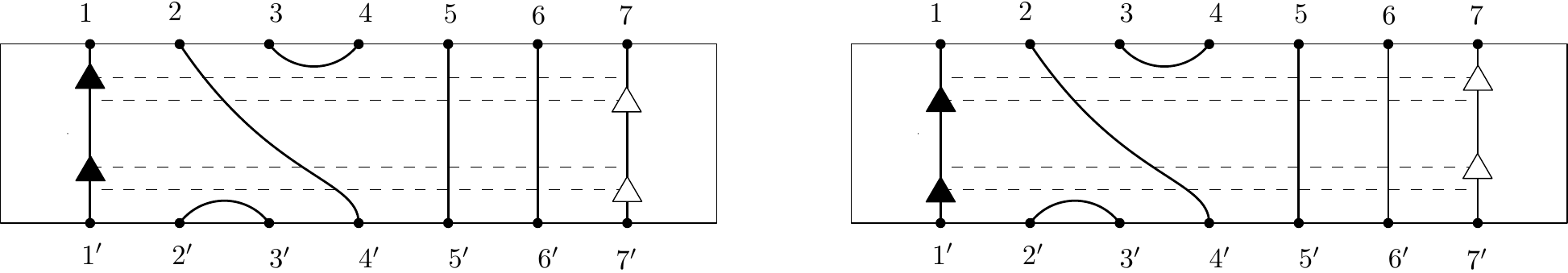}
\caption{Case ${\bf a}(d)=1$: dashed lines denote vertical positions. In these two different diagrams each decoration corresponds to a single block.}\label{decorated-Ex2}
\end{figure}

If ${\bf a}(d)> 1$, and $e$ is propagating, the sequence of its decorations has to be considered as a unique block.

A {\em concrete LR-decorated} pseudo $k$-diagram is any concrete $k$-diagram decorated using the above rules, (see conditions (D0)--(D4) in~\cite{ErnstDiagramI}). Moreover, we let a {\em LR-decorated} pseudo $k$-diagram an equivalence class 
of a concrete LR-decorated pseudo $k$-diagrams with respect to 
{\em $\Omega-$equivalence}, namely, if we can isotopically deform one diagram into the other in a way that any intermediate diagram is also a concrete LR-decorated pseudo $k$-diagram, and that the relative vertical positions of the blocks are preserved. We denote the set of LR-decorated pseudo diagrams by $T^{LR}_k (\Omega).$
Now define $\mathcal{P}^{LR}_k(\Omega)$ to be the free $\mathbb{Z}[\delta]$-module having $T_k^{LR}(\Omega)$ as a basis. We define a multiplication in $\PLR$ on the basis elements as follows and then we extend it bilinearly. Let $d', d \in T^{LR}_k(\Omega)$, then $d'd$ is obtained by concatenating $d'$ and $d$, by conjoining adjacent blocks and maintaining 
$\Omega-$equivalence. In~\cite[\S 3]{ErnstDiagramI}, Ernst proved that $\mathcal{P}^{LR}_k(\Omega)$ with this multiplication is a well-defined infinite dimensional $\mathbb{Z}[\delta]$-algebra.

We can notice that if  ${\bf a}(d')>1$, then for any $d \in T^{LR}_k(\Omega)$  ${\bf a}(d'd)>1 $, and so in any edge of $d'd$ we can consider the obtained sequence of decorations as a unique block. On the other hand if ${\bf a}(d')={\bf a}(d)=1$ and ${\bf a}(d'd)= 1$, we follow (b) and (c) given above to conjoin adjacent blocks in $d'd$.

Finally let $\widehat{\mathcal{P}}^{LR}_k(\Omega)$ be the associative $\Zd$-algebra equal to the quotient of $\PLR$ by the  relations depicted in Figure~\ref{C-Relations}, where the decorations on the edges represent adjacent decorations within the same block.
An example of product is given in Figure~\ref{decorated-Ex3}.

\begin{figure}[h]
\includegraphics[width=0.8\textwidth]{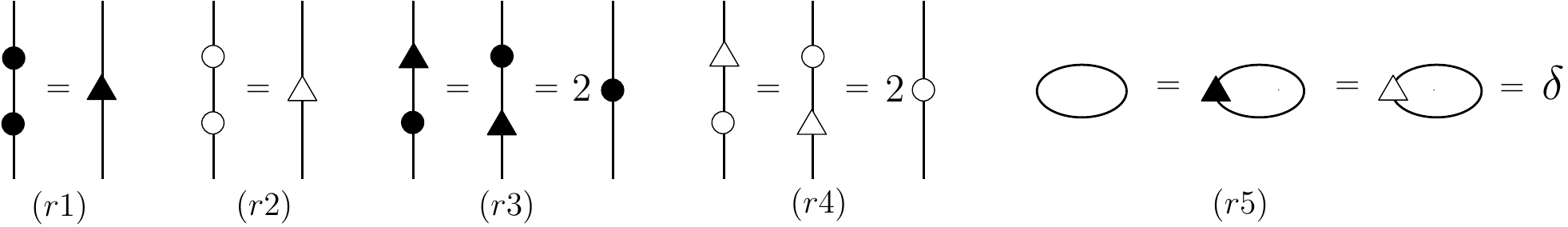}
\caption{The defining relations of $\widehat{\mathcal{P}}^{LR}_k(\Omega)$.}\label{C-Relations}
\end{figure}

\begin{figure}[h]
\includegraphics[width=0.8\textwidth]{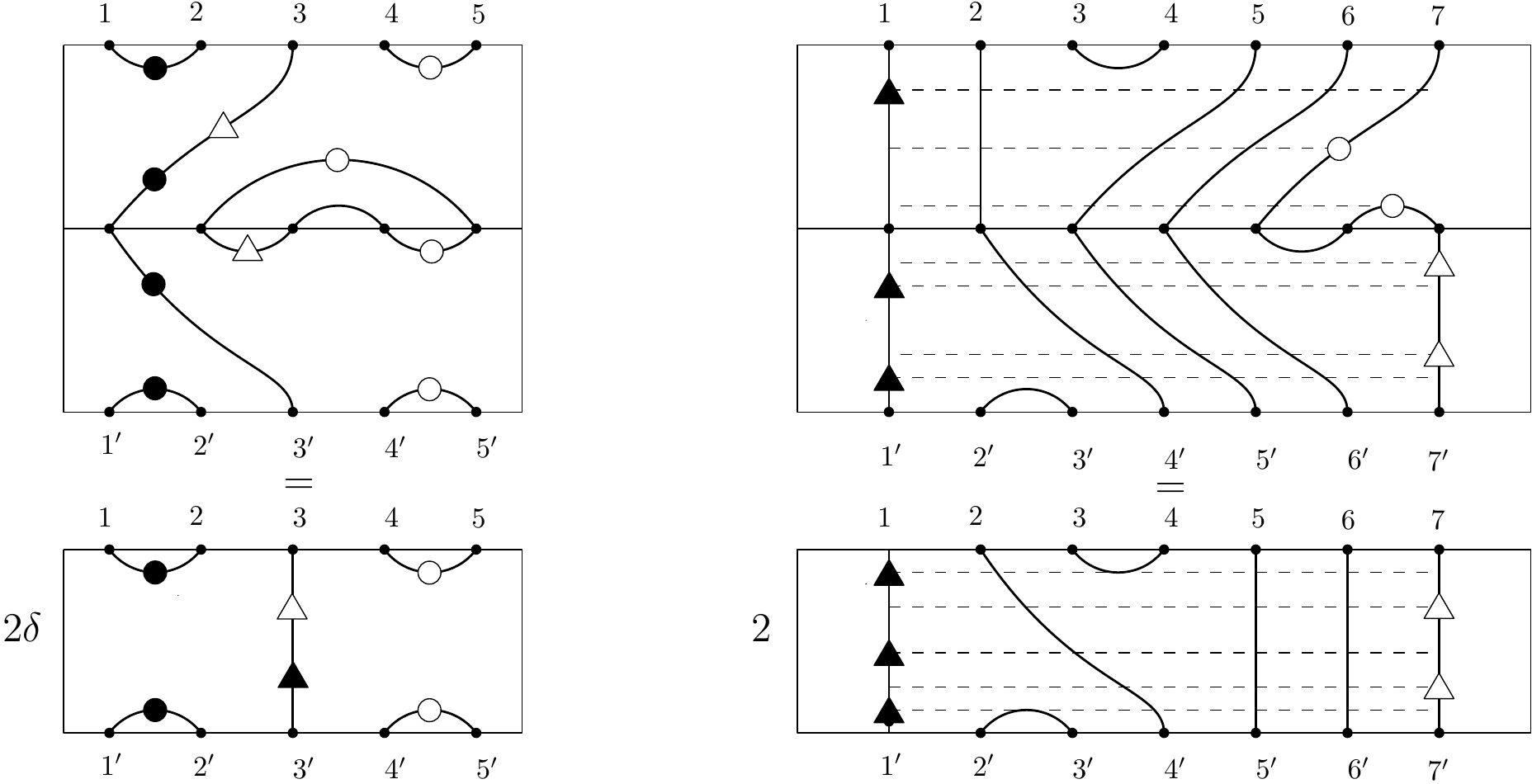}
\caption{Multiplications of decorated diagrams in $\widehat{\mathcal{P}}^{LR}_5(\Omega)$.}\label{decorated-Ex3}
\end{figure}

Ernst, in~\cite{ErnstDiagramI}, considers the following subalgebra of $\widehat{\mathcal{P}}^{LR}_{n+2}(\Omega)$.

\begin{Definition} We will denote by $\mathbb{D}(\widetilde{C}_n)$ the $\mathbb{Z}[\delta]$-subalgebra of $\widehat{\mathcal{P}}^{LR}_{n+2}(\Omega)$ generated by the {\em simple diagrams} $\td_1,\td_2,\dots,\td_n, \td_{n+1}$, where $\td_i=d_i$ for all $i\in \{2,\ldots,n\}$, and $\td_1$ and $\td_{n+1}$ are depicted in Figure~\ref{diagrammi-semplici-C}.
\end{Definition}

\begin{figure}[h]
\centering
\includegraphics[width=0.9\linewidth]{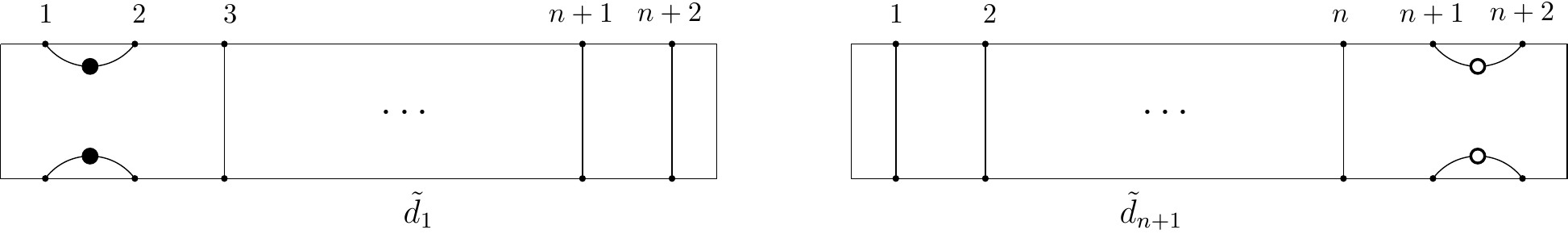}
\caption{Decorated simple diagrams.} \label{diagrammi-semplici-C}
\end{figure}

We notice that the decorated simple diagrams satisfy the relations (c1)--(c4) described in Section~\ref{sec:TL-algebra}. Ernst gave an explicit $\Z[\delta]$-basis of $\mathbb{D}(\widetilde{C}_n)$, consisting of the so-called admissible diagrams whose definition is stated below, and he proved the following results which is the analogous of Proposition~\ref{prop:thetaA} for $\tC_n$.

\begin{Theorem}[Ernst~\cite{ErnstDiagramI, ErnstDiagramII}]\label{prop:thetaC} \
\begin{itemize}
\item[(a)] $\GD(\tC_{n})$ is equal to the $\mathbb{Z}[\delta]$-module having the admissible diagrams as a basis;
\item[(b)] the map $\tilde{\theta}: \TL(\tC_{n})\longrightarrow \GD(\tC_{n})$ determined by $\tib_i \longmapsto \td_i$ is a well defined $\Z[\delta]$-algebra isomorphism that sends the monomial basis of $\TL(\tC_{n})$ to the basis of admissible diagrams of $\GD(\tC_{n})$.
\end{itemize}
\end{Theorem}

\begin{figure}[h]
	\centering
	\includegraphics[width=0.8\linewidth]{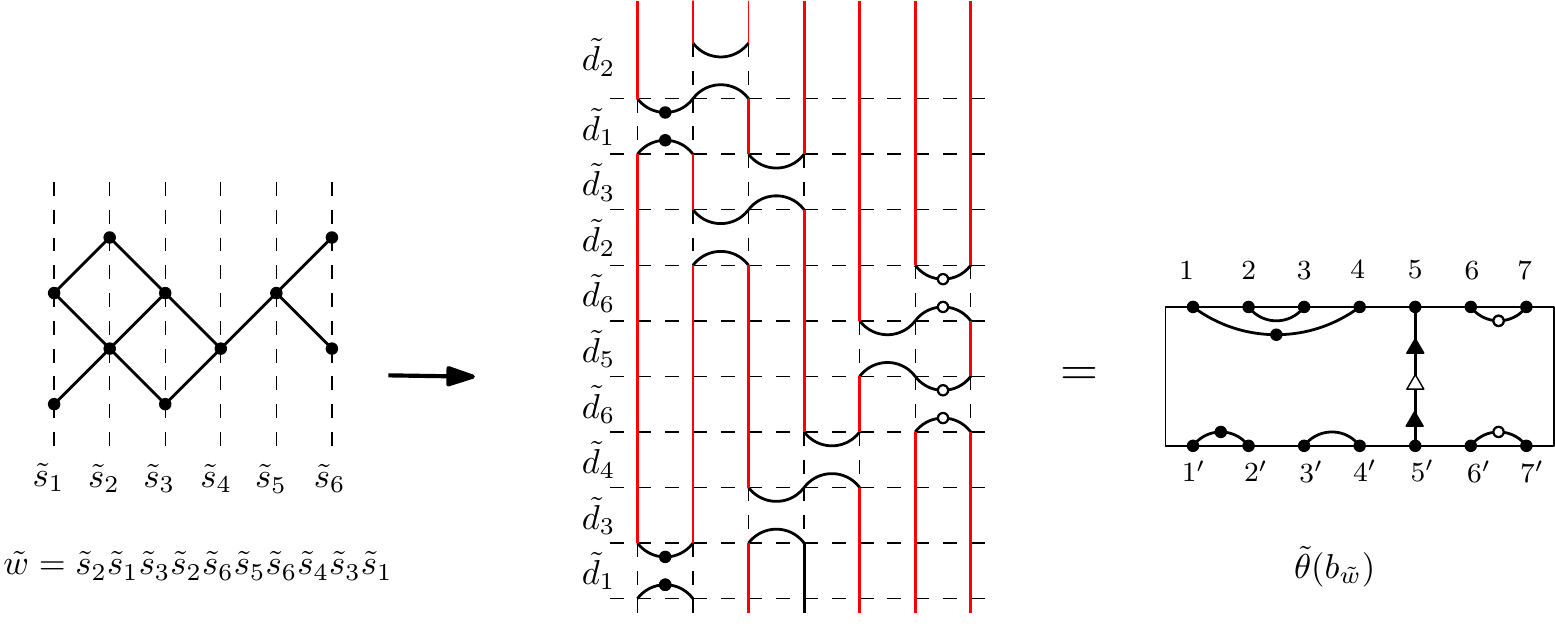}
	\caption{An application of the map $\tilde{\theta}$.}
	\label{Ernst-map}
\end{figure}

In Section \ref{sec:inj}, we will provide an algorithmic proof of the injectivity of the map $\tilde{\theta}$, giving a positive answer to a question raised by Ernst in~\cite[\S 6]{ErnstDiagramII}. More precisely, we will show that if $\tib_{\tw}$ and $\tib_{\tw'}$ are two distinct elements of the monomial basis of $\TL(\tC_{n})$, then $\tilde{\theta}(\tib_{\tw})$ and $\tilde{\theta}(\tib_{\tw'})$ are two distinct and independent diagrams in $\GD(\tC_{n})$. In~\cite{ErnstDiagramI}, Ernst proved that $\tilde{\theta}(\tib_{\tw})$ is an admissible diagram, but we will not use this feature to prove injectivity. Admissible diagrams will only be used in Section~\ref{surjectivity}, where we analyze the surjectivity of the map $\tilde{\theta}$.  

\subsection{Admissible diagrams}\label{sec:admissible}

\par We will now recall the definition of admissible diagrams of type $\tC_n$, or simply admissible diagrams. 

\begin{Definition}\label{def:ammissibili}
Let $d\in T_{n+2}^{LR}(\Omega)$, $d$  
is an {\em admissible diagram}  if it satisfies:
\begin{itemize}
  \item[(A1)] The only loops that may appear are equivalent to the one in Figure \ref{fig:loop}.
  
  \begin{figure}[h]
	\centering
	\includegraphics[width=0.055\linewidth]{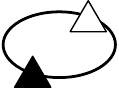}
	\caption{An admissible loop.}
	\label{fig:loop}
\end{figure}

\item[(A2)] Assume ${\bf a}(d)>1$ and let $e$ be the edge connected to node $1.$ If $e$ is not connected to node $1'$, then it is decorated and the first decoration is a $\bullet$. If $e$ is connected to $1'$, then exactly one of the following three conditions are met:
  \begin{enumerate}
      \item [(a)] $e$ is undecorated.
      \item [(b)] $e$ is decorated by a single $\tn$.
      \item [(c)] $e$ is decorated by a single block of decorations consisting of an alternating sequence of black  and white decorations such that the first decoration is  a  $\bullet$.
  \end{enumerate}  
   We have analogous restrictions for nodes $1',n+ 2,$ and $(n+ 2)'$, where we replace first with last for nodes  $1'$ and $(n+ 2)'$ and black decorations are replaced with white decorations for nodes $n+ 2$ and $(n+ 2)'$.
\item[(A3)]
   Assume ${\bf a}(d) = 1.$  Then the western end of $d$ is equal to one of the diagrams in  Figure \ref{fig:west}, where $u\in \{\emptyset,\tn \}$  and the other rectangles represent a sequence of blocks (possibly empty) such that each block is a single $\tn.$  Moreover, if $d$ is the diagram in Figure~\ref{fig:west} (b), then no more decorations occur on $d.$  Also, the occurrences of the $\bullet$ decorations  on the propagating edges in  Figure~\ref{fig:west} (c)-(d)-(e) have the highest (respectively, lowest) relative vertical position of all decorations occurring on any propagating edge. We have  analogous restrictions for the eastern end of $d,$  where the black decorations are replaced with white decorations.
\item[(A4)]No other $\bullet$ or $\circ$ decorations appear on $d$ other than those required in (A2) and (A3).
\end{itemize}
\end{Definition}

\begin{figure}[h]
	\centering
	\includegraphics[width=0.8\linewidth]{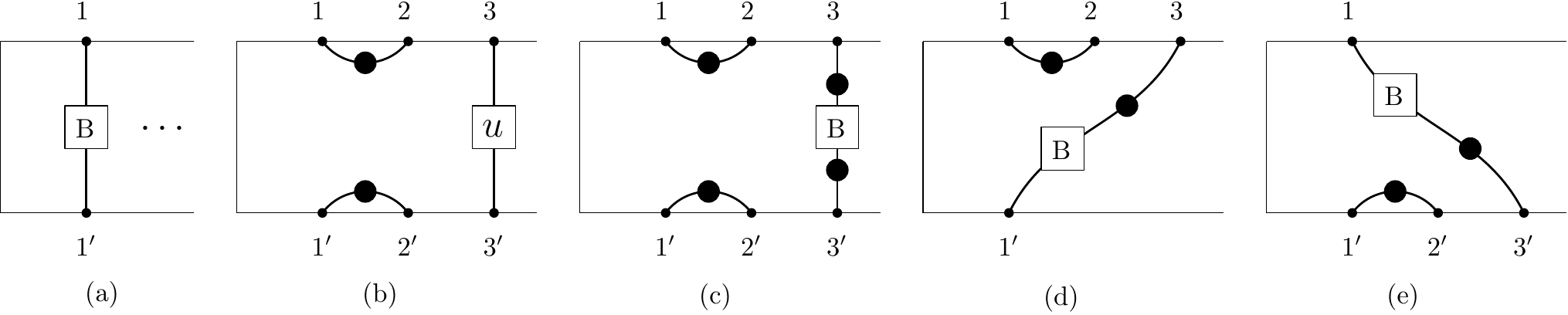}
	\caption{The western end of an admissible diagram.}
	\label{fig:west}
\end{figure}


\section{Reduction algorithm on heaps}\label{sec:algoCA}

In this section, starting from a FC element $\tw$ in $\tC_n$, we define a FC element $w$ in $A_{n+1}$ 
and we describe a reduction algorithm which transforms the  heap $H(\tw)$ into the heap $H(w)$.
 
\begin{Definition}\label{def:51}
Let $\tw \in \FC(\tC_n)$ and  $\ts_{i_1}\cdots \ts_{i_k}$ a  reduced expression for it.  We denote by $d(\tw)$ the following diagram of $T_{n+2}(\emptyset)$
$$d(\tw):=d_{i_1} \cdots d_{i_k}.$$
\end{Definition}

The element  $d(\tw)$ is well defined, in fact, any two reduced expressions of $\tw$ differ for a sequence of commutation relations and $\ts_i$ commutes with $\ts_j$ if and only if $d_{i}$ commutes with $d_{j}$. Moreover, $d(\tw)$ is a loop-free diagram in  $T_{n+2}(\emptyset)$ with a finite number of loops,  which allows us  to give  the following definition.

\begin{Definition}\label{w-primo} 
We denote by $\alpha:  \FC(\tC_n) \longrightarrow \N$ and  
$ \tau : \FC(\tC_n) \longrightarrow \FC(A_{n+1})$,   the maps that associate to any $\tw\in \FC(\tC_{n})$,  
 the  number of loops $\alpha(\tw)$ and the FC element $\tau(\tw)$ in $A_{n+1},$  defined by the decomposition      \begin{equation}\label{def:ad}
      d(\tw)=\delta^{\alpha(\tw)} d_{\tau(\tw)}
  \end{equation}
of the image of $d(\tw)$ in the quotient algebra $\GD(A_{n+1})$. Here and in the following we denote with the same symbol both the concrete diagram in $T_{n+2}(\emptyset)$ and its image in $\GD(A_{n+1})$. 
\end{Definition}

More precisely, if $\ts_{i_1}\cdots \ts_{i_k}$ is a reduced expression for $\tw$, $\tau(\tw)$ can be obtained as follows. Consider the product of the simple diagrams $d_{i_1} \cdots d_{i_k}$,  apply the relations (a1)-(a2)-(a3), and obtain a reduced expression for $d(\tw)$ of type $\delta^\alpha d_{j_1}\cdots d_{j_h}$. The diagram  $d_{j_1}\cdots d_{j_h}$ is loop-free and so by Remark~\ref{rem:TLA} there exists a unique $w \in \FC(A_{n+1})$ such that $d_{w} = d_{j_1}\cdots d_{j_h},$ hence  $\tau(\tw)=w=s_{j_1}\cdots s_{j_h}$.

\begin{Example}
The element $\tw=
\ts_2\ts_6 \ts_1\ts_3\ts_5\ts_7 \ts_2\ts_4\ts_6 \ts_1\ts_3\ts_5\ts_7 \ts_2\ts_4\ts_6 \ts_1\ts_3\ts_5\ts_7 \ts_4 \in \tC_6$ is alternating and its heap is depicted in Figure~\ref{Fig-delta-A} left.
A simple computation shows that $d(\tw)=\delta^2 d_2d_6d_1d_3d_5d_7d_4$, hence $\alpha(\tw)=2$ and $\tau(\tw)=w=s_2s_6s_1s_3s_5s_7s_4 \in A_7$ whose heap is represented in  Figure~\ref{Fig-delta-A} right.

\begin{figure}[h]
	\centering
	\includegraphics[width=0.8\linewidth]{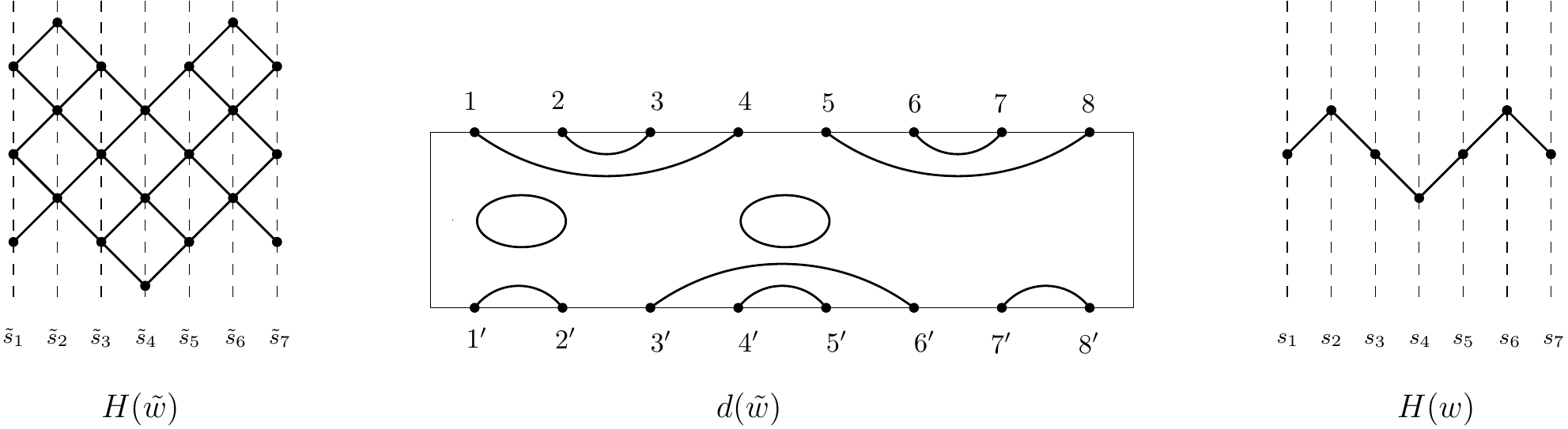}
	\caption{}
	\label{Fig-delta-A}
\end{figure}
\end{Example}

\begin{Remark}\label{step0}\
\begin{itemize}
\item[(1)] The diagram  $d_{\tau(\tw)}$ can also be obtained by computing the diagram $\tilde{\theta}(b_{\tw})=\td_{i_1} \cdots \td_{i_k}$ and then by removing all the loops and all the decorations on the other edges, where $\ts_{i_1} \cdots \ts_{i_k}$ is a reduced expression for $\tw$ (see also the map $r$ in \cite[\S 3.4]{ErnstDiagramI}).
\smallskip

\item[(2)] The map $\tau$ is surjective. In fact, if $s_{i_1} \cdots s_{i_k}$ is a reduced expression for $w \in \FC(A_{n+1})$, by Theorem~\ref{prop:caracterisation_fullycom}, $w$ is alternating and in $s_{i_1} \cdots s_{i_k}$ there is at most one occurrence of $s_1$ and $s_{n+1}$. This implies that the word $\ts_{i_1} \cdots \ts_{i_k}$ is a reduced expression of an alternating element $\tw\in \FC(\tC_n)$. Hence 
$$d(\tw)= d_{i_1} \cdots d_{i_k}=d_w$$ 
from which $\tau(\tw)=w$. On the other hand, if $\tw \in \FC(\tC_n)$ is alternating and its heap contains at most one occurrence of $\ts_1$ and $\ts_{n+1}$, then by the same argument the heaps $H(\tw)$ and $H(w)$ are isomorphic.  
\end{itemize}
\end{Remark}

In the following we give a general procedure that takes $H(\tw)$, erase some of its edges and vertices and produces $H(\tau(\tw))$. First, we need some definitions. 

If $H:=H(\tw)$ is the heap of $\tw \in \FC(\tC_n)$, we set $d(H):=d(\tw) \in  T_{n+2}(\emptyset)$. We also need a slightly more general definition. 
If $H$ is a heap of type $\tC_{n}$, not necessarily FC, we consider a reduced expression $\ts_{i_1} \cdots \ts_{i_k}$ in the commutation class of the element $\tw \in \tC_{n}$ represented by $H$. We define 
$$d(H):= d_{i_1} \cdots d_{i_k}$$ its associated diagram in $T_{n+2}(\emptyset)$. Note that if $H$ is not FC, the diagram $d(H)$ depends on $H$ and not on $\tw$. This is illustrated in Figure~\ref{Heap-diagram-H}, where $H_1$ and $H_2$ are heaps corresponding to the reduced expressions $\ts_3\ts_4\ts_3$ and $\ts_4\ts_3\ts_4$ of the same element $\tilde{w}\not\in \FC(\tC_5)$. These expressions belong to distinct commutation classes of $\tw$: the two corresponding diagrams are depicted in red, the heaps in black.   

\begin{figure}[h]
\centering
\includegraphics[width=0.7\linewidth]{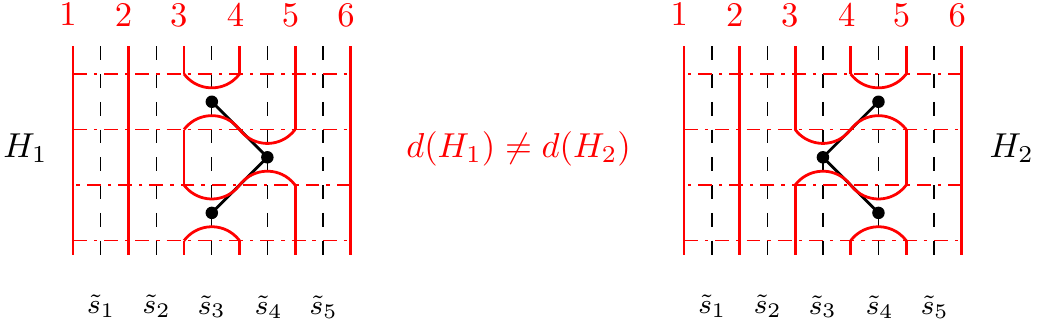}
\caption{Two non FC heaps corresponding to the same $\tilde{w} \in \tC_{5}$.} \label{Heap-diagram-H}
\end{figure}

\begin{Definition}
   Let $H$ be a heap of type $\tC_n$. We call {\em fork}, any convex subheap of $H$ of the form $H(\ts_i\ts_{i + 1}\ts_i)$ for $i \in \{1,\ldots, n\}$ or  $H(\ts_i\ts_{i -1}\ts_i)$ for $i\in \{2, \ldots, n+1\}$. 
\end{Definition}

In the sequel, if there is no ambiguity, we will denote forks simply writing the sequence of the labels $\ts_i\ts_{i \pm 1}\ts_i$. We are now ready to define a process that allows us to eliminate a fork $f$ from the Hasse diagram of a heap $H$ and to give rise to a new heap $H\setminus f$. 

\begin{Definition}[Fork elimination]\label{forkelimination}
 Let $H$ be a heap of type $\tC_n$ containing a fork $f=\ts_i\ts_{i \pm 1}\ts_i$. We denote by $H\setminus f$ the heap obtained by applying the following procedure to $H$, called {\em fork elimination} :
\begin{enumerate}
	\item[(F1)] Cancel the middle node $\ts_{i \pm 1}$ of the fork and all the edges in $H$ incident to it;
	\item[(F2)] Identify the two nodes $\ts_i$ of the fork;
	\item[(F3)] If $f=\ts_i\ts_{i + 1}\ts_i$ with $i <n$ (resp. $\ts_i\ts_{i - 1}\ts_i$ with $i>2$) and in column $i + 2$ (resp. $i-2$)  two consecutive unconnected nodes appear, then identify them.
\end{enumerate}
\end{Definition}

\begin{Definition}[Reduction algorithm]\label{reductionalgorithm}
Let $H$ be a heap of type $\tC_n$. We call {\em reduction algorithm} the following procedure:
\begin{enumerate}
	\item[(R1)]  Apply iteratively fork elimination to $H$ until no fork appears.
	\item[(R2)]  Denote the obtained heap by $\cru(H)$, where the nodes in column $i$ are labeled by $s_i$, for all $1 \leq i \leq n+1$.
	\item[(R3)]  Denote by $\del(H)$ the total number of identified nodes in the steps (F3).  	
\end{enumerate}
\end{Definition}

In Figure~\ref{Example-N2}, we start with the heap of the alternating element $\tw\in \FC(\widetilde{C}_5)$, whose a reduced expression is $\ts_{2}\ts_{1}\ts_{3}\ts_{2}\ts_{4}\ts_{6}\ts_{1}\ts_{3}\ts_{5}\ts_{2}\ts_{4}\ts_{6}\ts_{1}\ts_{3}\ts_{5}\ts_{2}\ts_{1}$. At the end of the algorithm we obtain the heap of $w\in \FC(A_6)$ whose a reduced expression is $s_2s_1s_3 s_6s_5$.

\begin{figure}[h]
\centering
\includegraphics[width=0.8\linewidth]{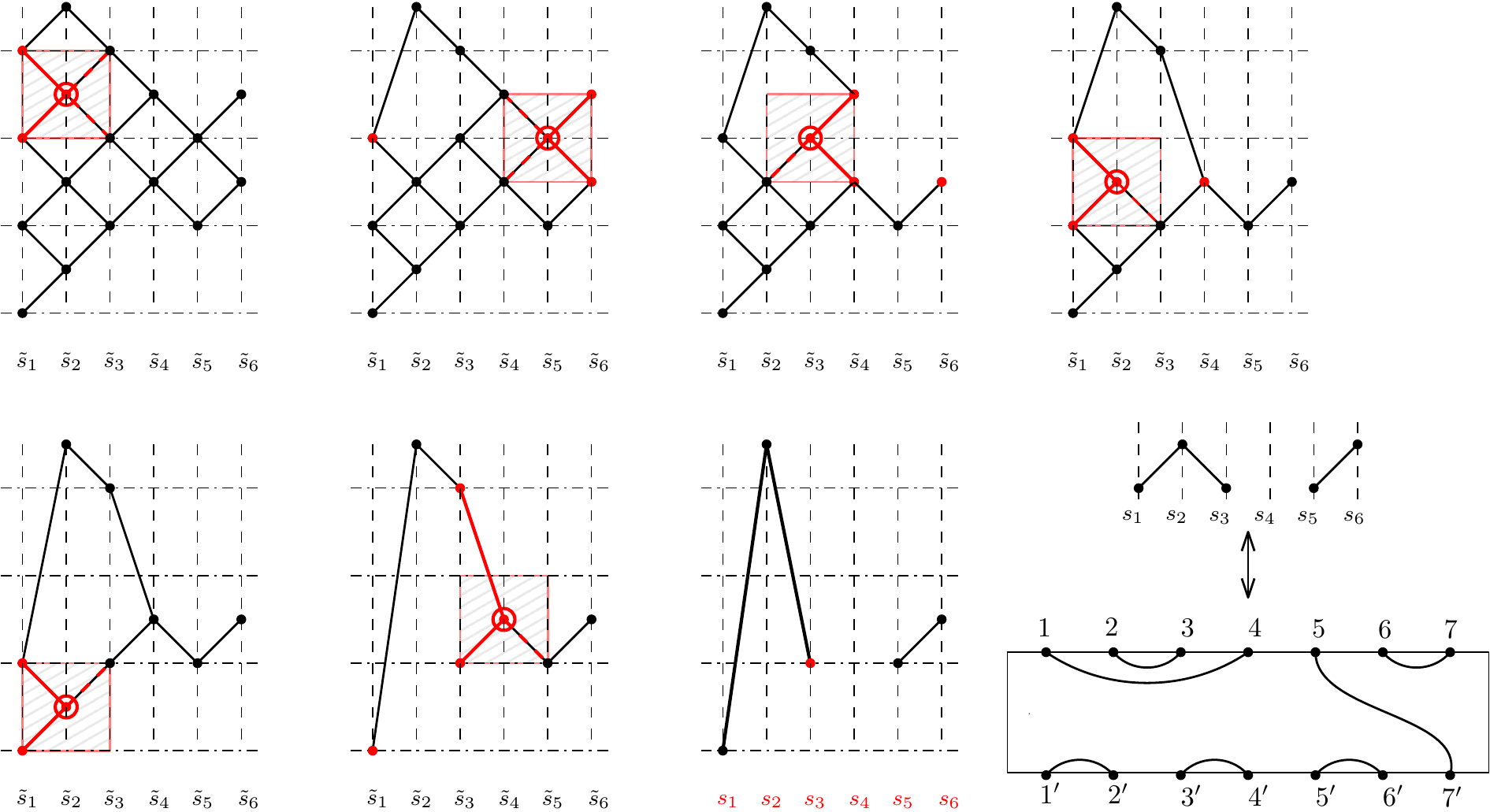}
\caption{How to pass from $H$ to $\cru(H)$: an alternating case.} \label{Example-N2}
\end{figure}

In Figure~\ref{Example-N3}, we consider two alternating elements in $\FC(\widetilde{C}_n)$ such that the associated heaps have the same images through the reduction algorithm. 
Step (F3) is used once in both cases. In the bottom example, the second fork elimination is performed from left to right and the nodes identified in (F3) are in column 5 (depicted in blue).
But if we would have eliminated the right fork $\ts_5 \ts_4\ts_5$, then the identified nodes would have been in column 3. In general, the columns where the identified nodes appear depend on the order of elimination of the forks while their total number does not, as it will be proved in Theorem~\ref{heap-reduction}. Moreover, observe that step (F3) can occur only if the heap contains at least a complete horizontal subheap as defined in Definition~\ref{def:complete}.
 
\begin{figure}[h]
\centering
\includegraphics[width=0.9\linewidth]{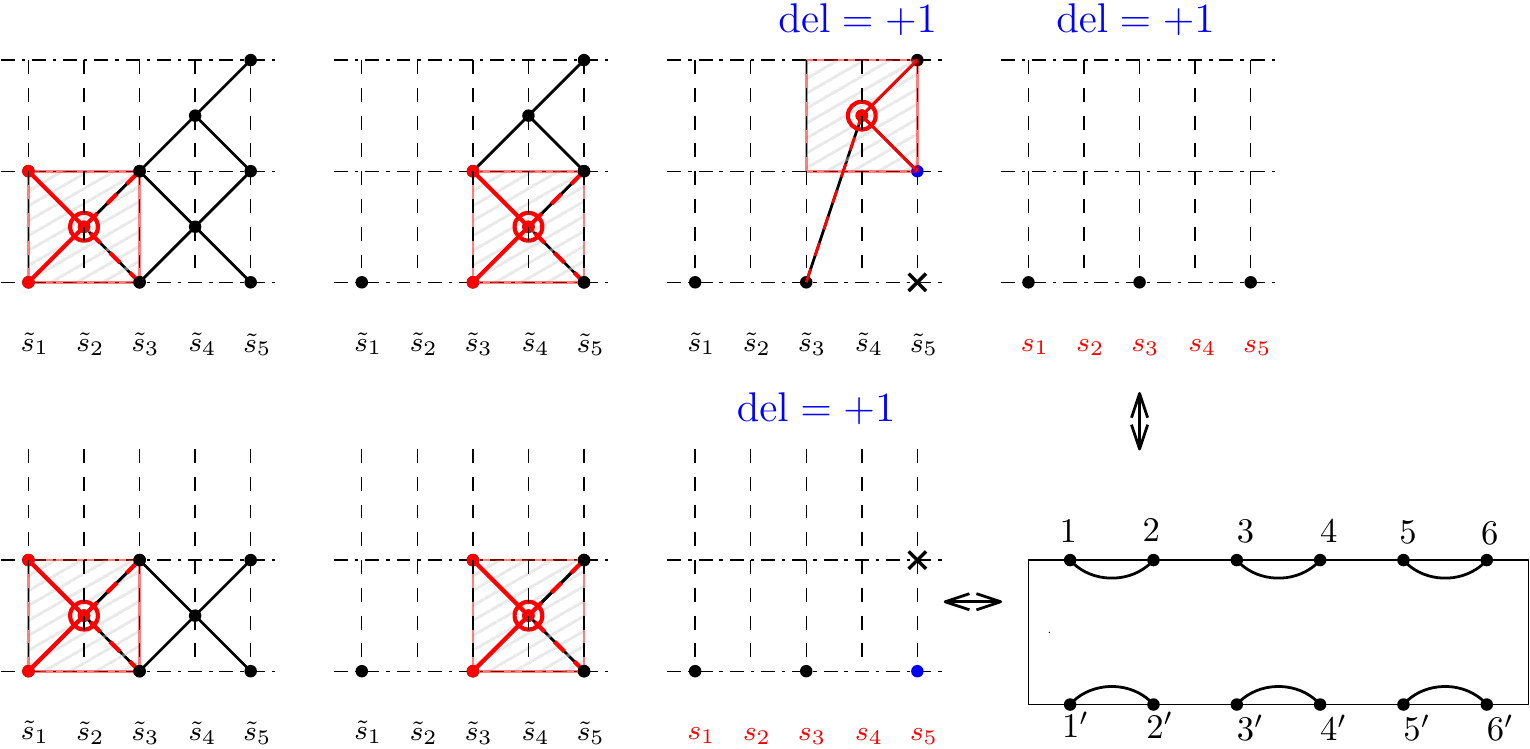}
\caption{How to pass from $H$ to $\cru(H)$: a case with a loop.} \label{Example-N3}
\end{figure}

In Figure~\ref{Example-N4}, we represent the case of a zigzag element. For each of these elements, the fork elimination reduces the zigzag to a segment joining the bottom node to the top one of the initial heap.
 
\begin{figure}[h]
\centering
\includegraphics[width=0.8\linewidth]{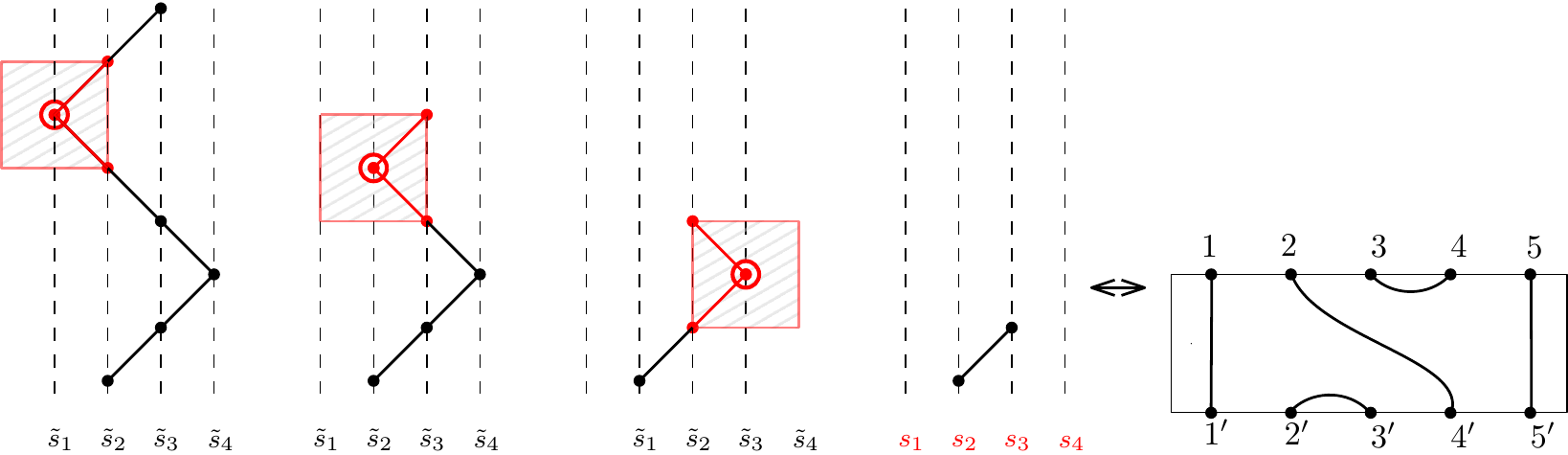}
\caption{How to pass from $H$ to $\cru(H)$: a zigzag case.} \label{Example-N4}
\end{figure}

In Figure~\ref{Example-N5} a left-peak is shown. The elements of this family with the same $j_\ell$, after reduction give rise to heaps with the first $j_\ell-1$ empty columns. The corresponding undecorated diagrams have $j_\ell-1$ initial vertical edges.

\begin{figure}[h]
\centering
\includegraphics[width=0.8\linewidth]{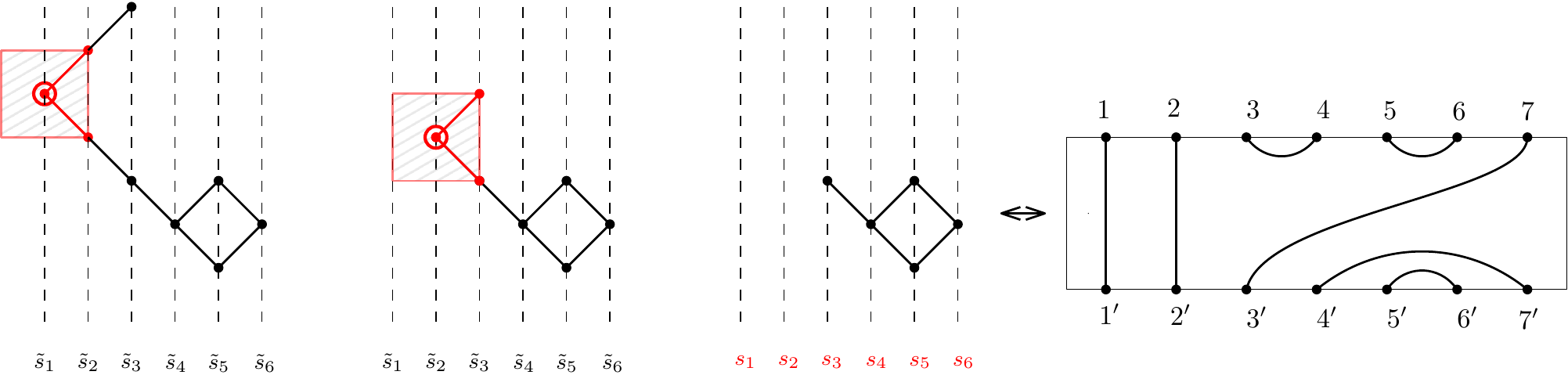}
\caption{How to pass from $H$ to $\cru(H)$: a left-peak case.} \label{Example-N5}
\end{figure}

All these examples show that the reduction algorithm always produces a heap without forks, corresponding to a particular $\tau(\tw) \in \FC(A_{n+1})$. In each case the corresponding undecorated diagram $d_{\tau(\tw)}$ is also represented. This happens in general as proved in the following result.

\begin{Theorem}\label{heap-reduction} Let $\tw \in \FC(\tC_n)$ and $H:=H(\tw)$. The heap $\cru(H)$ and the integer $\del(H)$ do not depend on the order of elimination of the forks. Moreover,  $\cru(H)=H(\tau(\tw))$, and $\del(H)=\alpha(\tw)$, where $\tau(\tw)$ and $\alpha(\tw)$ are defined in Definition~\ref{w-primo}. 
\end{Theorem}

\begin{proof}
Let us assume that $H$ contains a fork $f=\ts_i\ts_{i\pm 1}\ts_i$. By eliminating such a fork,  we obtain the heap $H\setminus f$ such that 
\begin{equation}\label{equality}
d(H)=\delta^\pm d(H\setminus f),
\end{equation}
where $\delta^\pm=\delta$ if (F3) occurs, and $\delta^\pm=1$ otherwise. This follows by the definition of $d(H)$, since $d_{i} d_{i\pm 1} d_{i}= d_{i}$, and $d_{i\pm 2}d_{i\pm 2}=\delta d_{i\pm 2}$. 

The  equality \eqref{equality} holds for any heap obtained running the algorithm. In particular, if $H_\ell$ is the heap obtained in the last step, we have $d(H)=\delta^{\del(H)}d(H_\ell)$. The heap $H_\ell$ does not contain any fork and no consecutive unconnected nodes in a single column, hence it is an alternating heap of type $A_{n+1}$, it corresponds to a unique $w\in \FC(A_{n+1})$, and by Theorem~\ref{prop:thetaA}(c) (see also Remark~\ref{rem:TLA}) we have  $d(H_\ell)=d_w$. The definition of $d(H)$ and this last equality imply that  $\delta^{\alpha(\tw)} d_{\tau(\tw)}=d(H)=\delta^{\del(H)}d_ {w}$, therefore  $\del(H)=\alpha(\tw)$ and by Theorem~\ref{prop:thetaA}(c) $w=\tau(\tw)$ hence $\cru(H)=H_\ell=H(\tau(\tw))$. 
\end{proof}

\section{Alternating heaps and snake paths}\label{sec:snakes_paths}

In this section we will introduce a set of paths on the edges of alternating heaps. They will play a crucial role in the definition of the decoration algorithm given in Section~\ref{sec:algdec}. 
\smallskip

In our graphical representation, any edge from a node $\ts_{i}$ to a node $\ts_{i+1}$ and any fork $\ts_{i}\ts_{i\pm 1} \ts_{i}$ in a heap $H$ identify  the {\em half-horizontal region} and the {\em horizontal region} of the plane delimited by the two horizontal lines passing through the nodes $\ts_{i}$ and $\ts_{i+1}$ of the edge, or through the two nodes $\ts_i$ of the fork, respectively. If we intersect such regions with the heap, we obtain a subheap consisting of all the nodes in or between the two horizontal lines delimiting the regions. 

\begin{Definition}\label{def:complete}
    If $H$ is of the form $H(\ts_1 \ts_3 \cdots \ts_{n+1} \cdot \ts_2  \ts_4 \cdots \ts_n \cdot \ts_1 \ts_3 \cdots \ts_{n+1})$, we will denote it by $H_{\infty}$ and call it a {\em complete horizontal heap}. In particular, it contains a fork $f_l
=\ts_1\ts_2\ts_1$ called a {\em left fork} and a fork $f_r
=\ts_{n+1}\ts_n\ts_{n+1}$ called a {\em right} fork.
\end{Definition}

\begin{figure}[h]
\centering
\includegraphics[width=0.85\linewidth]{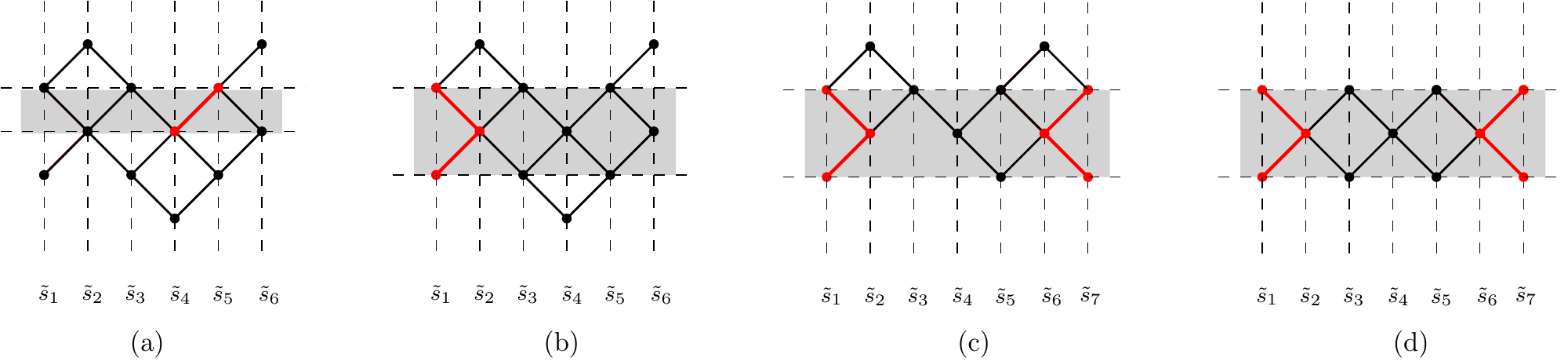}
\caption{The grey area represents a half-horizontal region in (a),  horizontal regions in (b), (c) and (d), but only in (d) identifies a complete horizontal heap.} \label{Fig_Regions}
\end{figure}

Left and right forks will play a crucial role in the sequel of this paper, they are depicted in red in Figure~\ref{Fig_Regions} (b),(c) and (d).

Let $H \in (\ALT)$ be a connected heap.  Starting in a node in the leftmost or in the  rightmost column of $H$,  we can build at most two paths, named {\em up-down paths}, made of alternating up and down steps, or down and up steps, joining adjacent nodes inside $H$. If both paths exist then one of them starts with a down step and the other one with an up step, they traverse the heap horizontally and can reach or not the opposite side, see Figure~\ref{Fig-Snakes}.
\smallskip

If $H$ is an alternating heap containing a left or a right fork $f$, then we can associate to $f$ a single path by joining up-down paths as described in the following algorithm.

 \begin{Definition}[Snake algorithm]\label{def:SA}
Let $H \in (\ALT)$ containing a left (resp. right) fork $f$. We construct a path $\gamma(H,f)$, called {\em snake path}, in $H$ as follows:
\begin{enumerate}

	\item[(UD1)]\begin{itemize} 
	\item[1.] Consider the up-down path starting in the top node of $f$, labeled $\ts_1$ (resp. $\ts_{n+1}$) having  a first right (resp. left)-down step, going right (resp. left) inside a horizontal half region of $H$, and reaching a node $p$ (resp. $q$) labeled $\ts_x$.
	
		If  $x < n+1$ (resp. $> 1$) then the path stops in $p$ (resp. $q$).
		
	    If  $x = n+1$  (resp. $= 1$) then the path stops in $p$ (resp. $q$) unless one of this two cases occurs:
		\begin{itemize}
		\item[(a)] there exists a node $p'$ (resp. $q'$) labeled $\ts_{n+1}$ (resp. $\ts_1$) just above $p$ (resp. $q$) and the path reaches $p$ with a right-down step (left-down);
		\item[(b)] there exists a node $p'$ (resp. $q'$) labeled $\ts_{n+1}$ (resp. $\ts_1$) just below $p$ (resp. $q$) and the path reaches $p$ with a right-up step (left-up). 
		\end{itemize}
    \item[2.] In cases (a) and (b) extend the path by composing it with the up-down path obtained applying step (UD1.1) to the fork $f'$, with top and bottom nodes $p$ and $p'$ (resp. $q$ and $q'$), starting in $p'$ (resp. $q'$) and going in the opposite direction inside the horizontal region determined by $f'$.
	\end{itemize}
	\smallskip
		
		\item[(UD2)] 
		
		\begin{itemize}
		\item[1.]  If the path does not stop, since $H$ is finite, then it forms a cycle. We set $\gamma(H,f):=\gamma_{\infty}$ such a cycle and the algorithm stops. 
		
	\item[2.] If the path ends in a node labeled by $\ts_x$, then we denote such a node by $B_x$, and the path by $\gamma^{top}$.

	\end{itemize}
	\smallskip
	
	\item[(UD3)] If $\gamma(H,f)\neq \gamma_{\infty}$, denote by $\gamma^{bot}$ the up-down path starting in the bottom node of $f$, labeled $\ts_1$ (resp. $\ts_{n+1}$), having  a first right (resp. left)-up step, going right (resp. left) inside a horizontal region of $H$, and obtained repeating the same steps as in (UD1), and denote by $B_y$ the node labeled $\ts_y$ where the path ends.
	\smallskip
	
	\item[(UD4)] Set $\gamma(H,f):=\gamma_{x,y}$ the path going from $B_x$ to $B_y$ obtained by composing $\gamma^{top}$ and $\gamma^{bot}$ going through the first backward and the second forward, and where the top and the bottom nodes of any encountered fork has been connected.
\end{enumerate}
\end{Definition}

	\begin{figure}[h]
	\centering
	\includegraphics[width=0.9\linewidth]{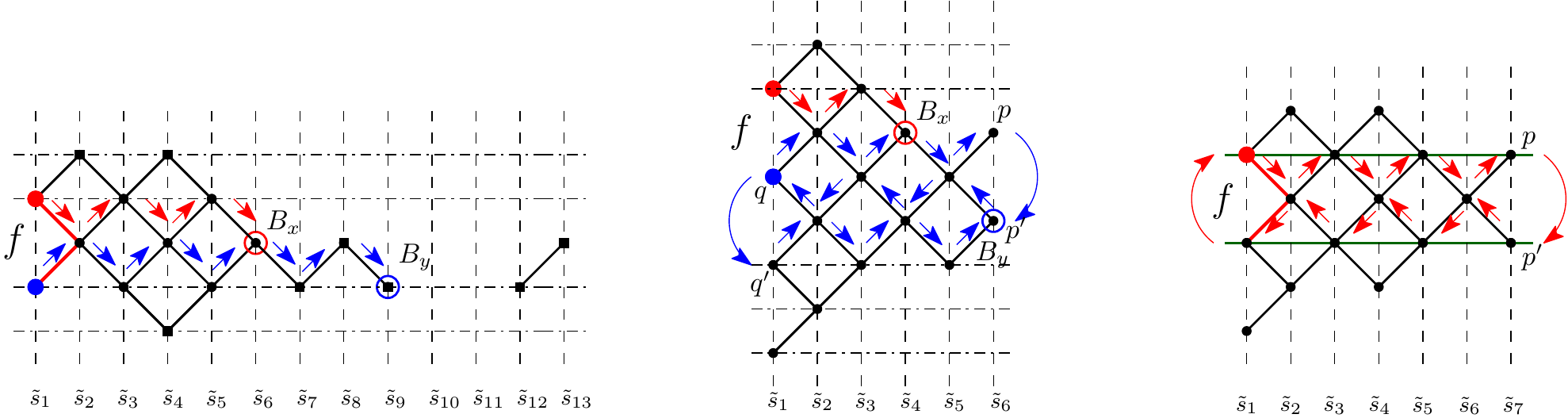}
	\caption{The red and blue up-down paths are examples of $\gamma^{top}$ and $\gamma^{bot}$, respectively.}\label{Fig-Snakes}
	\end{figure}

\begin{Remark}\label{rem:cycle}  Let $H \in (\ALT)$ with a fork $f$.
\begin{itemize}
\item[(a)] If $\gamma(H,f)$ crosses two different left (or right) forks $f$ and $\bar{f}$ in $H$, then $\gamma(H,f)=\gamma(H,\bar{f})$. Instead if $\gamma(H,f)=\gamma_{x,y}$ crosses a left fork $f$ and a right fork $\bar{f}$, then $\gamma(H,\bar{f})=\gamma_{x,y}$ or $\gamma(H,\bar{f})=\gamma_{y,x}$, see Figure~\ref{Fig_Regions}(c).

\item[(b)] If $\gamma(H,f)$ is a cycle then $n$ must be even. Indeed, if by contradiction $n$ is odd, the nodes $p$ and $p'$  (resp. $q$ and $q'$) in (UD1) can not be in the same horizontal region, so the path $\gamma(H,f)$ can not go back to the initial node and it would not be a cycle. An example is given in Figure~\ref{Fig-Snakes}, right.
\end{itemize}
\end{Remark}

\begin{Example}\label{rem:singlepath0}
 In Figure~\ref{Fig-Snakes}, left, $\gamma(H,f)=\gamma_{6,9}$ can be explicitly described as the path joining the node $B_6$ with the node $B_9$ and crossing the nodes labeled by 
 $$\textcolor{red}{B_6=\ts_6 - \ts_5 - \ts_4-\ts_3-\ts_2-\ts_1}- \textcolor{blue}{\ts_1-\ts_2-\ts_3-\ts_4-\ts_5-\ts_6-\ts_7-\ts_8-\ts_9=B_9}$$ and it is also represented by the blue path in Figure~\ref{Fig-UD1}, left.
\end{Example}

We now introduce some notation and definitions that will be useful in the sequel.

Let fix $H \in (\ALT)$. Any snake path $\gamma_{x,y}$ obtained by applying the algorithm to a fork identifies the subheap of $H$ of the nodes crossed by it.    
If there exists an edge of the Hasse diagram of $H$ not crossed by any $\gamma_{x,y}$, then in its associated half-horizontal region, there exists an up-down path containing it. Furthermore, we can always assume that such up-down path starts in a vertex $B_x$ labeled $\ts_x$ and ends in a vertex $B_y$ labeled $\ts_y$, with $x<y$, we denote it by $\gamma_{\vec{x,y}}$ and call it snake path as well. Examples are given in Figure~\ref{Fig-UD1}.
All the snake paths, coming from forks or not, give a partition of the edges of the Hasse diagram of $H$, where two snake paths are considered disjoint if they do not share any edge. Moreover, to avoid the ambiguity in Remark~\ref{rem:cycle} (a), the path $\gamma_{x,y}$ denotes the one obtained starting from a left fork if it crosses also a right fork. We also need to introduce also three types of {\em degenerate} snake paths: 
\begin{itemize}
\item $\gamma_{x}$ and $\gamma_{x'}$ denote the snake path containing only a maximal or a minimal vertical node labeled $\ts_x$, respectively.
\item $\gamma_{\check{x}}$ denotes the empty snake path at level $x$ meaning that in $H$ there are no nodes labeled by $\ts_{x-1}$ and $\ts_{x}$. 
\end{itemize}
We denote the set of all such paths by $\Snakes(H)$.

\begin{Example}\label{rem:singlepath}
The snake paths of the heaps in Figure~\ref{Fig-Snakes} are depicted in different colors in Figure~\ref{Fig-UD1}. In particular, the heap $H$ on the left has three snake paths not coming from forks depicted in green ${\gamma}_{\vec{1,6}}$, orange ${\gamma}_{\vec{3,6}}$, and black ${\gamma}_{\vec{12,13}}$, and several degenerate ones, so $\Snakes(H)=\{\gamma_{6,9}, {\gamma}_{\vec{1,6}}, {\gamma}_{\vec{3,6}},{\gamma}_{\vec{12,13}}, \gamma_{2},\gamma_{4},\gamma_{8},\gamma_{13},\gamma_{1'},\gamma_{4'},\gamma_{7'},\gamma_{9'}, \gamma_{12'},\gamma_{\check{11}}\}$.

\begin{figure}[h]
	\centering
	\includegraphics[width=0.9\linewidth]{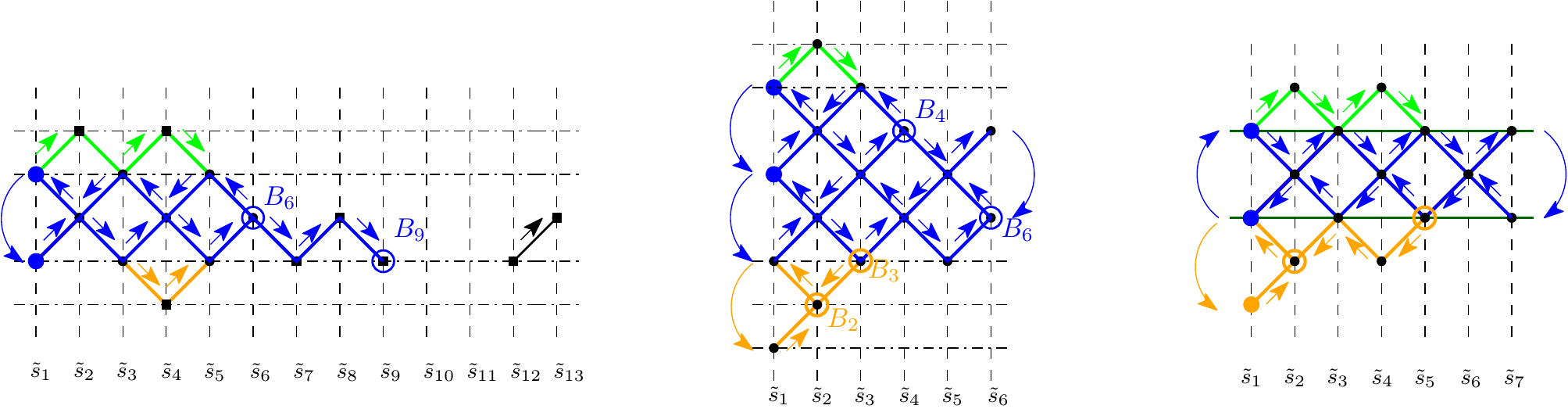}
	\caption{Heap partition in snakes paths.}\label{Fig-UD1}
	\end{figure}
\end{Example}

\begin{Definition}\label{def-hk} Let $H \in (\ALT)$ and $\gamma \in \Snakes(H)$. 
We define two nodes $h$ and $k$  of the  standard $(n+2)$-box as follows.
\begin{itemize}
\item[(a)] If $\gamma=\gamma_{x,y}$, then 
\begin{enumerate}
		\item  $h:=x+1$ (resp. $k:=y+1$) on the north face of the box, if the last step of $\gamma^{top}$  (resp. $\gamma^{bot}$) reaching $\ts_x$ (resp. $\ts_y$) is right-down;
		\item  $h:=x$ (resp. $k:=y$) on the north face of the box, if the last step of $\gamma^{top}$  (resp. $\gamma^{bot}$) reaching $\ts_x$ (resp. $\ts_y$) is left-down;
		\item  $h:=(x+1)'$ (resp. $k:=(y+1)'$) on the south face of the box, if the last step of $\gamma^{top}$  (resp. $\gamma^{bot}$) reaching $\ts_x$ (resp. $\ts_y$) is right-up;
		\item  $h:=x'$ (resp. $k:=y'$) on the south face of the box, if the last step of $\gamma^{top}$  (resp. $\gamma^{bot}$) reaching $\ts_x$ (resp. $\ts_y$) is left-up.
	\end{enumerate}  
\item[(b)] If $\gamma=\gamma_{\vec{xy}}$, then 
\begin{enumerate}
		\item[$(1)'$]  $h:=x$ (resp. $h:=x'$) on the north (resp. south) face of the box, if the first step of $\gamma$ is right-up (resp. right-down);
		\item[$(2)'$]  $k:=y+1$ (resp. $k:=(y+1)'$) on the north (resp. south) face of the box, if the last step of $\gamma$ is right-down (resp. right-up).
	\end{enumerate}  
\item[(c)] If $\gamma={\gamma}_{x}$, then $h=x$, $k=x+1$ on the north face of the box.
\item[(d)] If $\gamma={\gamma}_{x'}$, then $h=x'$, $k=(x+1)'$ on the south face of the box.
\item[(e)] If $\gamma={\gamma}_{\check{x}}$, then $h=x$ on the north face and $k=x'$ on the south face of the box.
\end{itemize}

We denote by $(h,k)_\gamma$ the corresponding edge in the standard $(n+2)$-box.  
\end{Definition}

\begin{figure}[h]
\centering
\includegraphics[width=0.85\linewidth]{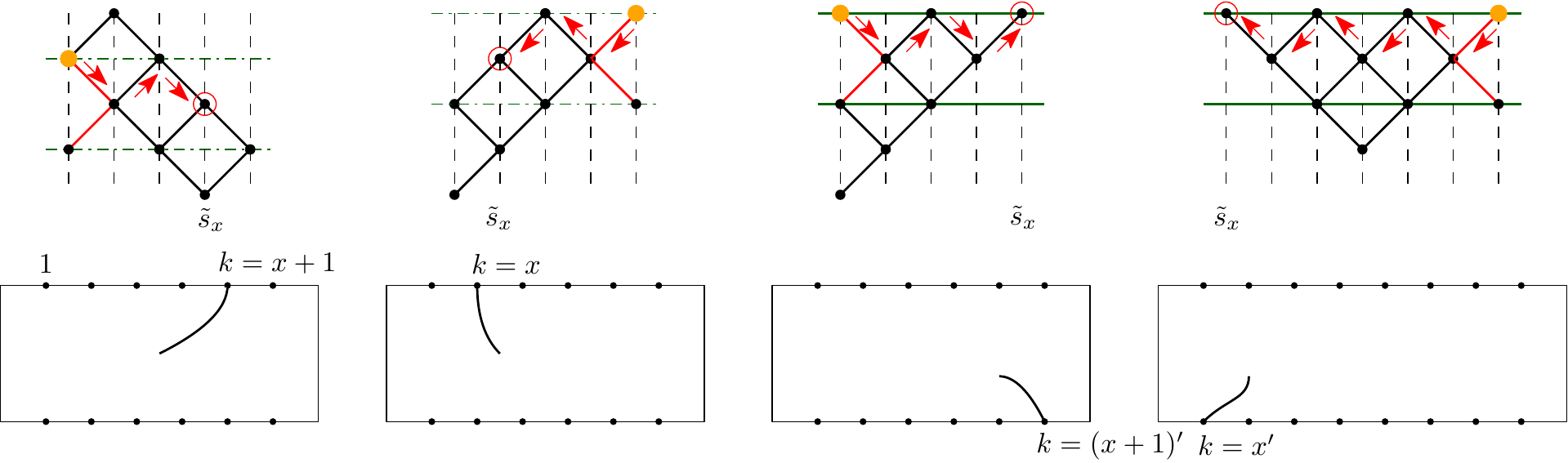}
\caption{} \label{Fig8}
\end{figure}

\begin{Example}
The edges corresponding to the snake paths in $\Snakes(H)$ computed in Example~\ref{rem:singlepath} are respectively $(7,10)$, $(1,6)$, $(3',6')$, $(12,14')$, $(2,3)$, $(4,5)$, $(8,9)$, $(13,14)$, $(1',2')$, $(4',5')$, $(7',8')$, $(9',10')$, $(12',13')$, and $(11,11')$.
\end{Example}

\begin{Theorem}\label{prop-hk}
Let $\tw\in (\ALT)$ and $\gamma\in \Snakes(H(\tw))$, different from $\gamma_\infty$ and non degenerate. Then the edge $(h,k)_\gamma$ belongs to $d_{\tau(\tw)}$.
\end{Theorem}

\begin{proof}
 Assume first  that  there exists a fork $f$ in $H(\tw)$ such that  $\gamma=\gamma_{x,y}=\gamma(H,f)$ and for the sake of simplicity suppose   $f$ is a left fork. 
\smallskip 

We first prove the statement for a path $\gamma_{x,y}$ obtained using only part (UD1.1) of the algorithm, i.e. $B_x$ and $B_y$ are reached by the paths $\gamma^{top}$ and $\gamma^{bot}$, without any change of direction, as in Figure~\ref{Fig-Snakes}, left.

Reading the labels of the nodes of $H(\tw)$ horizontally from left to right and from top to bottom, we obtain a reduced expression for $\tw$ of type ${\bf \tw}_1 \cdot {\bf \tw}_\gamma\cdot {\bf \tw}_2,$ where the factor
\begin{equation}\label{reduc-expression-H-epsilon}
 {\bf \tw}_\gamma:= \ts_1\ts_3\cdots \ts_{2i+1} \cdot \ts_2\ts_4\cdots \ts_{2j} \cdot \ts_1\ts_3\cdots \ts_{2l+1} 
\end{equation}
corresponds to the nodes crossed by the path $\gamma_{x,y}$, while ${\bf \tw}_1$ and ${\bf \tw}_2$ arise from the nodes in the regions above and below the one containing $\gamma_{x,y}$, respectively.

Since $f$ is a left fork, $\gamma^{top}$ does not change direction and $H(\tw)$ is alternating we have $h=x+1$ or $h=(x+1)'$ (see Definition~\ref{def-hk} (1)(3)):

\begin{itemize}
\item[(h1)] if $h=x+1$,  we do not have nodes labeled by $\ts_x$ and $\ts_{x+1}$ in the part of the heap above $B_x$, and so $\ts_x$ and $\ts_{x+1}$ are not in ${\bf \tw}_1$, see Figure~\ref{Fig-UD2}. In this case either  $x=2i+2$ or $x=2j$ where $i$ and $j$ are defined in (\ref{reduc-expression-H-epsilon}).
\item[(h2)] if $h=(x+1)'$, we do not have nodes labeled by $\ts_x$ and $\ts_{x+1}$ in the part of the heap below $B_x$, and so $\ts_x$ and $\ts_{x+1}$ are not in ${\bf \tw}_2$.  In this case $x=2i+1$.
\end{itemize}

Similarly, we can only have $k=y+1$ or $k=(y+1)'$ and :
\begin{itemize}
\item[(k1)] if $k=y+1$, we do not have nodes labeled by $\ts_y$ and $\ts_{y+1}$ in the part of the heap above $B_y$, and so $\ts_y$ and $\ts_{y+1}$ are not in ${\bf \tw}_1$, see Figure~\ref{Fig-UD2},  left.  In this case $y=2l+1$.
\item[(k2)] if $k=(y+1)'$, we do not have nodes labeled by $\ts_y$ and $\ts_{y+1}$ in the part of the heap below $B_y$, and so $\ts_y$ and $\ts_{y+1}$ are not in ${\bf \tw}_2$, see Figure~\ref{Fig-UD2}, right.  In this case  either $y=2l+2$ or $y=2j$.
\end{itemize}

\begin{figure}[h]
	\centering
	\includegraphics[width=0.63\linewidth]{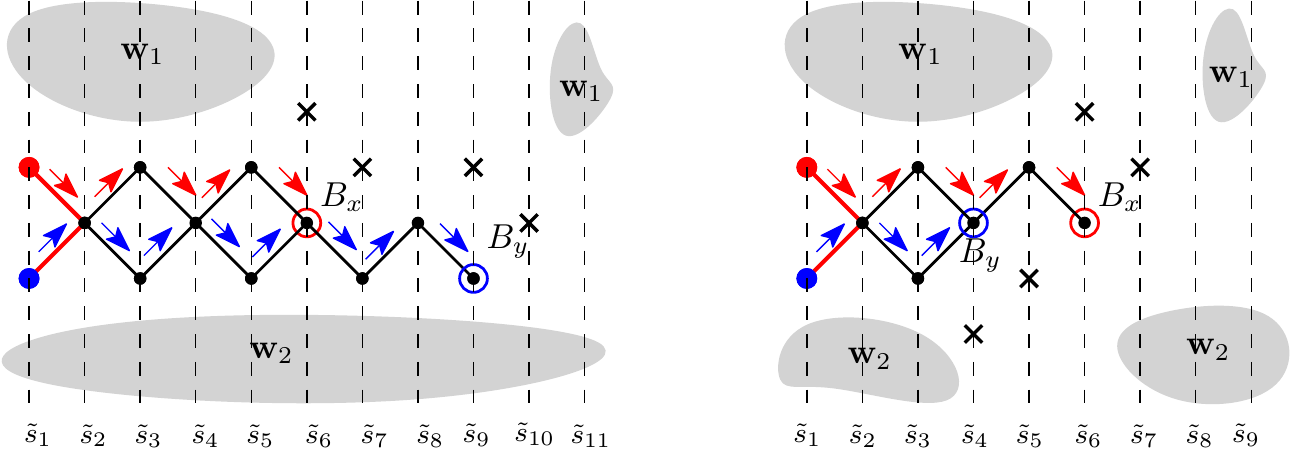}
	\caption{}
	\label{Fig-UD2}
\end{figure}

Now consider the concrete diagram obtained by the product $d({\bf \tw}_1) d({\bf \tw}_\gamma) d({\bf \tw}_2)$. We have
\begin{equation}\label{espressionedwgamma}
d({\bf \tw}_\gamma) = d_1d_3\cdots d_{2i+1}d_2d_4\cdots d_{2j}d_1d_3\cdots d_{2l+1}, 
\end{equation}
and similar factorizations for $d({\bf \tw}_1)$ and $d({\bf \tw}_2)$ arising from the heap. As in Definition~\ref{w-primo}, we know that
\begin{equation}\label{fact_delta}
d({\bf \tw}_1) d({\bf \tw}_\gamma) d({\bf \tw}_2)=\delta^{\alpha(\tw)} d_{\tau(\tw)}
\end{equation}
and 
a direct computation shows that $d({\bf \tw}_\gamma)$ has an edge going from node $h$ to node $k$. 
\smallskip

Now suppose that case (h1) holds. The diagram $d({\bf \tw}_1)$ has a vertical edge from $h$ to $h'$ since in its considered factorization $d_x$ and $d_{x+1}$ do not appear. Hence,  the edge starting from $h$ in the north face of $d({\bf \tw}_1)$ goes straight until it is composed with the edge in $d({\bf \tw}_\gamma)$ starting in $h$, and it sorts in node $k$. There are two possibilities: 
\begin{itemize}
\item If (k1) holds, then the path into analysis passes through a node $k$ of $d({\bf \tw}_\gamma)$ and goes back up to node $k$ in the north face of $d({\bf \tw}_1)$, since 
$d({\bf \tw}_1)$ has a vertical edge from $k$ to $k'$ (there are neither $d_y$ nor $d_{y+1}$ in the considered factorization). In this case $d({\bf \tw}_2)$ does not play any role. From \eqref{fact_delta}, $d_{\tau(\tw)}$ contains a non-propagating edge from $h$ to $k$ in the north face.
 
\item if (k2) holds, then the path into analysis passes through a node $k'$ in the south face of $d({\bf \tw}_\gamma)$ and goes down to node $(y+1)'=k$ in $d({\bf \tw}_2)$ since there are neither $d_y$ nor $d_{y+1}$ in the considered factorization of $d({\bf \tw}_2)$. Again, from \eqref{fact_delta}, $d_{\tau(\tw)}$ contains a propagating edge from $h$ in the north face to $k$ in the south face.
\end{itemize}
If (h2) holds a similar analysis shows that $d_{\tau(\tw)}$ contains the edge $(h,k)$, starting from the south face, propagating or not.
\medskip

In the general case $\gamma_{x,y}$ can reach the extreme columns $1$ and  $n+1$ of $H(\tw)$ and it can go back and forth several times between them, possibly crossing adjacent horizontal regions. This makes ${\bf \tw}_\gamma$ more involved then in the previous case as we will see in what follows. So assume that the path $\gamma^{top}$ reaches the last column and does not stop there (see (UD1.2)). 
\smallskip

If $n$ is even, $\gamma^{top}$ passes through the top node of a right fork and changes direction inside the same horizontal region and stops in a node $B_x$ with $x\geq 3$. The path $\gamma^{bot}$ can not reach the node $B_x$ otherwise $\gamma$ would be a cycle, so $\gamma^{bot}$ stops in a node $B_y$, with $y<x$. Hence the corresponding expression~\eqref{espressionedwgamma} has the form $d_1 d_3 \cdots d_{n+1} \cdot d_2 d_4 \cdots d_{n} \cdot d_1 d_3 \cdots d_{2l+1} \cdot d_{2q+1} d_{2q+3} \cdots d_{n+1}$, where $l < q$ depending on $y$ and $x$. A similar arguments as above can be used to show that the edge $(h,k)$ is in $d_{\tau(\tw)}$.
The symmetric situation, namely when  $n$ is  even,  $\gamma^{bot}$  reaches the column $n+1$  and $\gamma^{top}$ does not, is analogous.

\medskip If $n$ is odd, when the path $\gamma^{top}$ reaches column $n+1$, it passes through the top node of a right fork in the upper horizontal region and continues in the opposite direction. Hence it passes through all the nodes with labels $\ts_i$ with $1\leq i \leq n+1$. And so the expression  \eqref{espressionedwgamma} contains the two factors $d_\mathcal{O}:=d_{1}  d_3 \cdots d_{n}$ and $d_\mathcal{E}:=d_2 d_4 \cdots d_{n+1}$.  Now if $\gamma^{top}$ goes back till column $1$, it reaches the top node of another upper left fork and it passes through all nodes labeled $\ts_i$ with $i$ odd, and so the expression  \eqref{espressionedwgamma} will contain another factor $d_\mathcal{O}$. Depending on how many times the path $\gamma^{top}$ changes direction a finite numbers of such factors occurs in \eqref{espressionedwgamma}. The same argument works for $\gamma^{bot}$. Hence the general  expression  for  \eqref{espressionedwgamma} is
\begin{equation}\label{complicated}
d({\bf \tw}_\gamma)= d_{init}(d_\mathcal{E} \cdot d_\mathcal{O})^\ell  d_{end},
\end{equation}
where $\ell$ depends on the number of times the path goes back and forth from columns 1 and $n+1$, and $d_{init}$ and $d_{end}$ are made of initial or final subfactors of $d_\mathcal{O}$ or $d_\mathcal{E}$. 
Now the fact that $(h,k)$ belongs to  $d_{\tau(\tw)}$ follows from the same argument used at the beginning of this proof.

It remains to show the statement for $\gamma={\gamma}_{\vec{x,y}}$ not coming from a fork. In this case 
${\bf \tw}_\gamma:= \ts_x\ts_{x+2}\cdots \ts_{x+2i} \cdot \ts_{x+1}\ts_{x+3}\cdots \ts_{x+2i\pm 1}$
or the reciprocal ${\bf \tw}_\gamma:= \ts_{x+1}\ts_{x+3}\cdots \ts_{x+2i+1} \cdot \ts_x\ts_{x+2}\cdots \ts_{x+2i+1\pm 1}$. Once again the same argument above can be applied.
\end{proof}

\begin{Theorem}\label{prop:cycle} Let $\tw \in (\ALT)$, and let $f$ be a fork such that $\gamma(H(\tw),f)=\gamma_{\infty}$. Then a loop $\mathcal{L}$ appears in $d(\tw)$. Moreover, exactly two unconnected consecutive nodes arise in the fork elimination of $f$. 
\end{Theorem}

\begin{proof} For the sake of simplicity, we assume $f$ is a left fork. Since $\gamma(H(\tw),f)=\gamma_{\infty}$ is a cycle, then $n$ is even and $f$ identifies a complete horizontal subheap in $H(\tw)$, see Figures~\ref{Fig_Regions} (right). If  we apply Fork elimination~\ref{forkelimination} to the  left fork $\ts_1\ts_2\ts_1$ and we repeat it for  all the other forks $\ts_i\ts_{i+1}\ts_i$, $i=3,5,\ldots,n-1$ in the same horizontal region, then,  when the last fork $\ts_{n-1} \ts_n \ts_{n-1}$ has been eliminated, the two nodes labeled $\ts_{n+1}$ in the horizontal region are unconnected, and so step (F3) occurs. Hence by \eqref{equality} in Theorem~\ref{heap-reduction}, we obtain $\delta=d_{n+1}^2$ in the factorization of $d(\tw)$ in the quotient algebra or equivalently a loop appears in its realization as a concrete diagram.
\end{proof}

Let $\tw \in (\ALT)$ and $d=d(\tw)=\delta^{\alpha(\tw)} d_{\tau(\tw)}$. Let $\mathcal{E}(d)$ be the multiset of edges of $d$, namely the edges of  $d_{\tau(\tw)}$ together with $\alpha(\tw)$ undecorated loops $\mathcal{L}$.

\begin{Proposition}\label{prop:E}
Let $\tw \in (\ALT)$. The map $E:\Snakes(H(\tw)) \longrightarrow \mathcal{E}(d(\tw))$ defined by
$$\gamma \mapsto (h,k)_\gamma \quad \mbox{and} \quad \gamma_\infty \mapsto \mathcal{L}$$ is a one-to-one correspondence.
\end{Proposition}
\begin{proof}
Let us denote $H:=H(\tw)$ and $d:=d(\tw)$. The map $E$ is well-defined by Theorems~\ref{prop-hk}, ~\ref{prop:cycle}, and observing that if $\ts_x$ is maximal or a minimal element in $H$, then the non-propagating edge $(x,x+1)$ or $(x',(x+1)')$ is in $d(H)$, while if $\ts_x$ and $\ts_{x+1}$ are missing in $H$ then a vertical edge $(x,x')$ occurs in $d$.   
\smallskip

From Theorem~\ref{prop:cycle}, it follows that for each cycle in $\Snakes(H)$ we have an undecorated loops $\mathcal{L}$ in $d$. Viceversa, since we started from a reduced expression $\ts_{i_1}\cdots \ts_{i_k}$ of $\tw$, the  product $d_id_i$  never occurs in the expression $d_{i_1}\cdots d_{i_k}$ of $d$, hence if a loop appears in $d$, then it can only be obtained as the product $d_1\cdots d_{n+1} \cdot d_2\cdots d_n \cdot d_1\cdots d_{n+1}$ corresponding to a complete horizontal subheap $H_{\infty}$ in $H$, and so a snake path $\gamma_\infty$ is in $\Snakes(H)$.
\smallskip

Let $\gamma$ and $\gamma'$ two distinct snake paths that are not cycles. If one of them is degenerate, it is obvious that their images are different. So suppose both of them are non degenerate. Since they are disjoints they cannot reach the starting vertex $B_x$ with the same edge. Hence by Definition~\ref{def-hk} the corresponding node $h$ must be different. This settles the injectivity of $E$.
\smallskip

Let $e=(h,k)$ an edge of $d$ with $k$ different from $h\pm 1$ and  $h'$, for the sake of simplicity suppose $h$ on the north face. By the definition of $d$, there is a node $B_x$ labeled by $\ts_{h-1}$ or $\ts_h$ on the northern boarder of $H$ which is not a maximal one. More precisely, only one of these two situations occurs in the Hasse diagram of $H$: either $x=h$ and there is an up edge $\ts_h -\ts_{h+1}$ or $x=h-1$ and there is a down edge $\ts_{h-2} -\ts_{h-1}$. In both cases let us consider the unique snake path $\gamma$ starting in $B_x$ and containing such a step. By Theorem~\ref{prop-hk} the edge $(h,k)_\gamma$ belongs to $\mathcal{E}(d)$ and so necessarily $(h,k)_\gamma=e$. It is easy to see that the edges of type $(h,h+1)$ or $(h,h')$ are images of degenerate snake paths of $H$, and this settles the surjectivity.
\end{proof}

We end this section by showing two technical results that will be useful in the next section. 
Consider the order 
\begin{equation}\label{order}
1\succ_\ell 2\dots\succ_\ell n+2\succ_\ell(n+2)'\succ_\ell(n+1)'\succ_\ell\dots\succ_\ell 1'
\end{equation}
on the set of the nodes of the standard $n+2$-box.

\begin{Lemma}\label{successivo} Let $H(\tw) \in (\ALT)$, $\gamma_\infty \neq \gamma \in \Snakes(H)$ coming from a left fork $f$, and $(h,k)_\gamma$ the 
associated edge in  $d_{\tau(\tw)}$. Then $h \succ_\ell k$.
\end{Lemma}
\begin{proof}
	In order to see the inequality we need to distinguish some cases. If $h$ is in the north face and $k$ is in the south face we are done. If both $h$ and $k$ are in the north face, then $\gamma^{top}$ does not change direction and must stop with a right-down step in a node $B_{h-1}$ (see Definition~\ref{def-hk}(a)(1)). On the other hand, $\gamma^{bot}$ might stop  with  a right-down step in a node $B_{k-1}$ with $k>h$ if it does not change direction, or with a left-down step in a node $B_{k+1}$ with $k+1>h$ if it changes direction. In both cases $h\succ_\ell k$.
	
If $h$ is in the south face then either $\gamma^{top}$ does not change direction and ends with a right-up step in a node $B_{h-1}$ (see Definition~\ref{def-hk}(a)(2)) or ends with a left-up step in a node $B_{h+1}$ if it changes direction. In both cases, $\gamma^{bot}$ can not change direction and has to stop  with a right-up step in a node $B_{k-1}$ with $k-1<h-1$. Hence $k$ is in the south face and $h\succ_\ell k$.
\end{proof}

In particular, the previous lemma tells us that if $(h,k)_\gamma$ is a propagating decorated edge then $h$ is in the north face and $k$ in the south face of the standard box.

\begin{Lemma}\label{Lemmahk_nodeleft} Let $H(\tw) \in (\ALT)$, $f$ and $\bar{f}$ two left forks with associated snake paths $\gamma$ and $\bar{\gamma}$ that are not cycles. Let $f$ be  above $\bar{f}$ in our graphical representation of $H(\tw)$ and $(h,k)_\gamma$ and $(\bar{h},\bar{k})_{\bar{\gamma}}$ be the associated edges in $d_{\tau(\tw)}$. Then \begin{enumerate}
		\item If $h=\bar{h}$, then $k=\bar{k}$.
		\item If $h\not=\bar{h}$, then $h \succ_\ell k \succ_\ell \bar{h} \succ_\ell \bar{k}$.
	\end{enumerate}
	
\end{Lemma}
\begin{proof}
	Since $h$ and $k$ are the nodes of a unique edge in $d_{\tau(\tw)}$, then the first equality is obvious. To show point (2), thanks to Lemma~\ref{successivo} we only need to prove that $k\succ_\ell \bar{h}$.
	We first prove the inequality for two forks with a common node.
	We need to distinguish some cases. If $k$ is in the north face and $\bar{h}$ is in the south face we are done. If both $k$ and $\bar{h}$ are in the north face, $k$ is obtained from $\gamma^{bot}$ with $B_{y}$, $y=k-1$ (see Definition~\ref{def-hk}(a)(1)), while $\bar{h}$ is obtained from $\bar{\gamma}^{top}$ ending in $B_x$ with $x=\bar{h}-1$ or $\bar{h}$ (see Definition~\ref{def-hk}(a)(1)(2)). In both cases, there are nodes in $H(\tw)$ labeled $\ts_1,\ldots,\ts_{k-1}$ that are crossed by $\gamma^{bot}$ and $\bar{\gamma}^{top}$. Since $\bar{h}$ is on the north face, the second path can not stop before the first one and so $\bar{h}>k$. If $\bar{h}$ and $k$ are in the south face the same argument with the roles of  $\gamma^{bot}$ and $\bar{\gamma}^{top}$ reversed works.
	
	Now if the two forks $f$ and $\bar{f}$ does not share a node, then we consider the sequence of forks $f=f_1,\ldots,f_q=\bar{f}$  such that  $f_i$ is  above $f_{i+1}$,  and with a common node.   By applying iteratively  (1) or (2) in the case of a common node we get the inequality.
\end{proof}

Now if we define the order 
\begin{equation}\label{order1}
n+2\succ_r n+1\succ_r\ldots 1\succ_r1'\succ_r\ldots\succ_r(n+2)'
\end{equation}
on the set of the nodes then we have the two following analogous results for right forks.

\begin{Lemma}\label{successivo2}
Let $H(\tw) \in (\ALT)$, $\gamma_\infty \neq \gamma \in \Snakes(H)$ coming from a right fork $f$, and $(h,k)_\gamma$ the 
associated edge in  $d_{\tau(\tw)}$. Then $h \succ_r k$.
\end{Lemma}

\begin{Lemma}\label{Lemmahk_noderight}
Let $H(\tw) \in (\ALT)$, $f$ and $\bar{f}$ two right forks with associated snake paths $\gamma$ and $\bar{\gamma}$ that are not cycles. Let $f$ be  above $\bar{f}$ in our graphical representation of $H(\tw)$ and $(h,k)_\gamma$ and $(\bar{h},\bar{k})_{\bar{\gamma}}$ be the associated edges in $d_{\tau(\tw)}$. Then \begin{enumerate}
		\item If $h=\bar{h}$, then $k=\bar{k}$.
		\item If $h\not=\bar{h}$, then $h \succ_r k \succ_r \bar{h} \succ_r \bar{k}$.
	\end{enumerate}
	\end{Lemma}

\section{Decoration algorithm}\label{sec:algdec}

In this section we will define an algorithm that adds decorations to the edges of $d(\tw)$ and we will prove that the obtained decorated diagram is exactly the image of the basis element $b_{\tw}$ through the algebra morphism $\tilde{\theta} : \TL(\tC_n) \longrightarrow \GD(\tC_n)$, providing a combinatorial description for such a map.
\smallskip

\subsection{Alternating heaps}

Theorems~\ref{prop-hk} and \ref{prop:cycle} allow the definition of the following algorithm.

\begin{Definition}[Decoration algorithm for $(\ALT)$]\label{algorithmdecoration}
	Let $\tw \in (\ALT)$ and consider the undecorated concrete diagram $d(\tw)$.
\begin{enumerate}
	\item[(D1)] If there are no forks in $H(\tw)$, then go to step (D4).
	\item[(D2)] Otherwise, fix a fork $f$ in $H(\tw)$. 
	\begin{enumerate}
	\item[1.] If $\gamma(H(\tw),f)=\gamma_{\infty}$, then replace a $\mathcal{L}$ by a decorated loop $\lott$.
	\item[2.] If $\gamma(H(\tw),f)=\gamma$, then going from $h$ to $k$, add a decoration $\tn$ (resp. $\tb$) to the edge $(h,k)_{\gamma}$ of $d_{\tau(\tw)}$, each time the oriented path $\gamma$ crosses a left (resp. right) fork, taking into account the order in which $\gamma$ crosses them. 
	\end{enumerate}
	
	\item[(D3)] Iterate the above procedure starting from a fork that has not been crossed by $\gamma$.
	
\item[(D4)] Add a $\bullet$ (resp. $\circ$) as first decoration to the edges starting from $1$ and  $1'$, (resp. $n+2$ and from $(n+2)'$) in the diagram $d_{\tau(\tw)}$, except if there is an edge joining $1$ to $1'$ (resp. $n+2$ to $(n+2)'$).   
	\smallskip
	\item[(D5)] 
	
	Set $\mathtt{dec}_A(\tw)$ the obtained diagram.
	\end{enumerate}
\end{Definition}

\begin{figure}[h]
\centering
\includegraphics[width=0.7\linewidth]{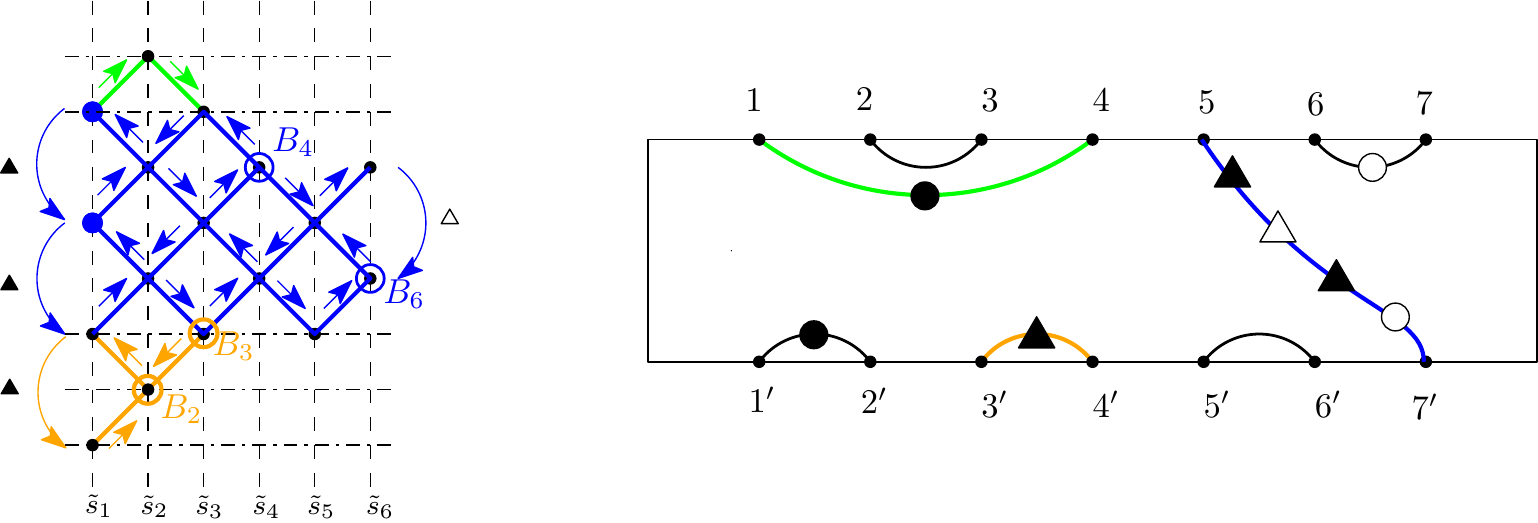}
\caption{How to decorate the edges $(5,7')$ and $(4',3')$.} \label{Example-N1-N}
\end{figure}

\begin{Remark}\label{ordine-decorazioni}\
\begin{itemize}
\item[(a)] The Decoration algorithm does not depend on the order in which the forks have been chosen. In fact, by (D3) if two forks are used in the algorithm, then they identify distinct edges to be decorated. Furthermore, if a left and a right fork give rise to equal snake paths with opposite orientation (see Remark~\ref{rem:cycle}(a)) and Figure~\ref{Fig_Regions}(c), then step (D2) guarantees that the obtained decorated edges are equal. 

\item[(b)] Combining Lemma~\ref{Lemmahk_nodeleft} (2) and the Decoration algorithm for $(\ALT)$, it follows that if we apply the algorithm to the left forks ordered from top to bottom, then the edges of $d_{\tau(\tw)}$ are decorated following the order in \eqref{order}. Namely, in the hypotheses of Lemma~\ref{Lemmahk_nodeleft}, if $h\succ_\ell \bar{h}$ the edge $(h,k)_\gamma$ is decorated before $(\bar{h},\bar{k})_{\bar{\gamma}}$. 
In particular, if the top left fork decorates the edge $(\bar{h},\bar{k})_{\bar{\gamma}}$ with a $\tn$ decoration, then all the edges $(h,k)$, with $h\succ_\ell \bar{h}$ have not $\tn$ decorations.
\end{itemize}
\end{Remark}

\begin{Theorem}\label{thm-ALT}
Let $\tw\in (\ALT)$. Then $\mathtt{dec}_A(\tw)=\tilde{\theta}(b_{\tw})$. 
\end{Theorem}

\begin{proof}
Given $\tw=\ts_{i_1}\cdots \ts_{i_k} \in (\ALT)$, we consider $\tilde{\theta}(\tw):=\tilde{\theta}(b_{\tw})=\td_{i_1}\cdots \td_{i_k} \in \GD(\tC_n)$. By Definition~\ref{def:51}, it easily follows that the diagram $\tilde{\theta}(\tw)$ without the decorations is equal to $d(\tw)$. 
\smallskip

If $H(\tw)$ does not contain any fork, then there are no decorated loops neither in $\mathtt{dec}(\tw)$ nor in $\tilde{\theta}(\tw)$. Moreover, at most one occurrence of both $\ts_1 $ and $\ts_{n+1}$ appears in $\ts_{i_1}\cdots \ts_{i_k}$. Steps (D1) and (D4) of the decoration algorithm apply and   $\mathtt{dec}_A(\tw) $ has at most a $\bullet$ and $\circ$ decorations on the edges starting on $1$, $1'$, $n+2$ and $(n+2)'$. On the other hand, in $\tilde{\theta}(\tw)$ either there is no occurrence of $\td_1$ (resp. $\td_{n+1}$) hence there is an undecorated edge from 1 to $1'$ (resp. $n+2$ to $(n+2)'$) in it, or there is a single occurrence of $\td_1$ (resp. $\td_{n+1}$), that adds a $\bullet$ (resp. $\circ$) decoration to both the edges starting from 1 and $1'$ (resp. $n+2$ to $(n+2)')$ of $\tilde{\theta}(\tw)$. It is easy to see that $\mathtt{dec}_A(\tw)=\tilde{\theta}(\tw) $ in this case.
\smallskip

Now by construction $\mathtt{dec}_A(\tw)$ has $\alpha(\tw)$ decorated loops $\lott$. On the other hand, the $\alpha(\tw)$ loops of $\tilde{\theta}(\tw)$ are also of the form $\lott$. In fact,  by 
Proposition~\ref{prop:E} any undecorated loop in $d(\tw)$ corresponds  to a complete horizontal subheap in $H(\tw)$. Hence in a reduced expression of $\tw$ there is a factor of the form $\ts_1\ts_3\cdots \ts_{n+1} \cdot \ts_2\ts_4\cdots \ts_{n} \cdot \ts_1\ts_3\cdots \ts_{n+1}$ that gives rise to a product 
\begin{equation}\label{reduc-expression-cycle}
\td_1\td_3\cdots \td_{n+1} \cdot \td_2\td_4\cdots \td_{n} \cdot \td_1\td_3\cdots \td_{n+1}
\end{equation}
of simple diagrams in $\tilde{\theta}(\tw)$. A direct computation shows that the corresponding diagram has all non-propagating edges $(1,2),(3,4),\ldots,(n+1,n+2)$, $(1',2'),(3',4'),\ldots,((n+1)',(n+2)')$ and a unique loop $\lott$. 

Consider now the case of not loop decorated edges. 

Let $f$ be a fork in $H(\tw)$ such that $\gamma(H(\tw),f)=\gamma_{x,y}=\gamma$, and by the sake of simplicity assume $f=\ts_1\ts_2\ts_1$ and that $\gamma$ crosses only $f$, as in the first case of the proof of Proposition~\ref{prop-hk}. Therefore, a reduced expression for $\tw$ is given by $\bf{\tw_1 \cdot \tw_\gamma \cdot \tw_2}$ with $\bf \tw_\gamma$ as in \eqref{reduc-expression-H-epsilon}. The edge $(h,k)_\gamma$ in $d_{\tau(\tw)}$ is decorated in (D2) by a $\tn$, since $\gamma_{x,y}$ crosses only the initial fork $f$. Now again a direct computation shows that 
\begin{eqnarray}\label{eq:calcoli}
\tilde{\theta}(\tw)& = & \tilde{\theta}({\bf \tw_1}) \cdot \td_1\td_3\cdots \td_{2i+1}\td_2\td_4\cdots \td_{2j}\td_1\td_3\cdots \td_{2l+1} \cdot \tilde{\theta}({\bf \tw}_2)\\
& = & \tilde{\theta}({\bf \tw_1}) \cdot \td_3\cdots \td_{2i+1}\cdot \td_1 \td_2 \td_1 \cdot \td_4\cdots \td_{2j} \td_3\cdots \td_{2l+1} \cdot \tilde{\theta}({\bf \tw_2})\nonumber
\end{eqnarray}
contains the edge $(h,k)_\gamma$ with a $\tn$ decoration coming from the two consecutive $\bullet$ corresponding to the factor $\td_1 \td_2 \td_1$. 
As in the proof of Proposition~\ref{prop-hk}, the multiplication by the simple diagrams corresponding to ${\bf \tw}_1$ and ${\bf \tw}_2$ does not affect neither $(h,k)_\gamma$ nor its decorations.
\medskip

Finally, if $\gamma_{x,y}$  crosses at least a right and a left fork, we distinguish two cases. 
\smallskip

If $n$ is even, as in proof of Proposition~\ref{prop-hk},  suppose that $\gamma^{top}$ reaches column $n+1$ and goes back, while $\gamma^{bot}$ stops before column $n+1$ inside the horizontal region determined by $f$. The decoration algorithm adds a single $\tb$ and $\tn$ to the edge $(h,k)_\gamma$. 
In $\tilde{\theta}(\tw)$ the same computation an in \eqref{eq:calcoli} with  $2i=n$, and $2l<n$ gives the same decorations on $(h,k)_\gamma$.  
\smallskip

If $n$ is odd, then the edge $(h,k)_\gamma$ is decorated by the Decoration algorithm with an alternating sequence of black and white decorations, or vice-versa. More precisely, each time $\gamma_{x,y}$ crosses a fork $f$, Step (D2) adds a $\tn$ or $\tb$ decoration to $(h,k)_\gamma$, depending if the fork is left or right, while Step (D4) adds a $\bullet$ as initial decoration and a $\circ$ just before node $k$ (if $(h,k)_\gamma$ is not a vertical edge). See Figure~\ref{Example-N1-N}.

Consider now $\tilde{\theta}(\tw)=\tilde{\theta}(\bf{\tw_1}) \cdot \tilde{\theta}(\bf{\tw_\gamma})\cdot \theta(\bf{\tw_2}).$ As before, the product by $\tilde{\theta}(\bf\tw _1)$ and $\tilde{\theta}({\bf \tw}_2)$ does not affect the edge $(h,k)_\gamma$ and its decorations.
In the product $d({\bf \tw}_\gamma)$ in \eqref{complicated} (see proof of Proposition \ref{prop-hk}), consider the factor $(d_\mathcal{E} \cdot d_\mathcal{O})^\ell$  with $\ell \geq 1$.
If  $\ell=1$,  then  $\gamma_{x,y}$ crosses exactly one left and one right fork and $\tilde{\theta}(\bf\tw_\gamma)$ has exactly one occurrence of $\td_1\td_2\td_1$ and  $\td_{n+1}\td_{n}\td_{n+1}$. Therefore in  $\tilde{\theta}({\bf \tw}_\gamma) $ the edge $(h,k)_\gamma$ is decorated with a $\tn$, a $\tb$ and a $\bullet$ as initial decoration and a $\circ$ just before node $k$ (if $(h,k)_\gamma$ is not a vertical edge).
If $\ell\geq 2,$  then $\gamma_{x,y}$ crosses $\ell-1$ left and right forks and $\tilde{\theta}(\bf\tw_\gamma)$ has $\ell-1$ alternating occurrences of $\td_1\td_2\td_1$ and  $\td_{n+1}\td_{n}\td_{n+1}$. Therefore in  $\tilde{\theta}({\bf \tw}_\gamma) $ the edge $(h,k)_\gamma$ is decorated with   $\ell-1$ alternating sequences ($\tn ,\tb)$ and  a $\bullet$ as initial decoration and a $\circ$ just before node $k$ (if $(h,k)_\gamma$ is not a vertical edge) and so, also in this case, $\mathtt{dec}_A(\tw)=\tilde{\theta}(\tw)$.
\end{proof}

\begin{Corollary}\label{cor:immagineA}
Let $\tw \in (\ALT)$ and $d=\mathtt{dec}_A(\tw)$. If $(1,1')$ (resp. $(n+2,(n+2)')$) is in $d$ then it is not decorated. Otherwise, the edges starting from 1 and $1'$ (resp. $n+2$ and $(n+2)'$) have both a $\bullet$ (resp. $\circ$) as first decoration. 
\end{Corollary}

\subsection{Left and right peaks}

In this section,  we consider $\tw \in (\LP), (\RP)$ or $(\LRP)$ which is not a pseudo zigzag. For brevity, we denote $(\LRPZ):=(\LRP) \setminus (\PZZ)$. We recall that $d_{\tau(\tw)}$:
		\begin{itemize}
			\item starts with $j_\ell-1$ vertical edges if $\tw\in (\LP), (\LRPZ)$, and 
			\item ends with $j_r-1$ vertical edges if $\tw \in (\RP), (\LRPZ)$. 
		\end{itemize}

\begin{Definition}[Semireduced heap]
If $H \in (\LP), (\RP)$ or $(\LRPZ)$ we will denote by $H_{sr}$, respectively, the alternating heap $H_{\{{\ts}_{j_\ell}\rightarrow \}}$, $H_{\{\leftarrow {\ts}_{j_r}\}}$ or  $H_{\{\ts_{j_\ell},\ldots,\ts_{j_r}\}}$ of Definition~\ref{def:famillesCtilde}, and we will refer to $H_{sr}$ as the {\em semireduced heap} of $H$. 
\smallskip

Note that $H_{sr}$ is exactly the heap obtained from $H$ by applying iteratively Fork elimination~\ref{forkelimination} to the forks (see Figure~\ref{Example-N5}):
\begin{itemize}
	\item $\ts_{j+1}\ts_j\ts_{j+1}$, for $j=1,\dots,j_\ell-1$  ;
	\item $\ts_{j}\ts_{j+1}\ts_{j}$, for $j=n,\dots,j_r$.
\end{itemize}
\end{Definition}

If $\tw \in (\LP), (\RP)$ or $(\LRPZ)$, and if $H=H(\tw)$, we define  $\tw_{sr}$ the unique alternating element such that  $H(\tw_{sr})=H_{sr}$. Observe that $d_{\tau(\tw)}=d_{\tau(\tw_{sr})}$, since $H_{sr}$ appears in an intermediate step of the reduction algorithm, and it is equal to $\mathtt{dec}_A(\tw_{sr})$ without decorations.

\begin{Definition}[Decoration algorithm for $(\LP), (\RP)$ or $(\LRPZ)$]\label{algdecpeak}
Let $\tw \in (\LP), (\RP)$ or $(\LRPZ)$.
\begin{enumerate}
	\item[(DP1)] Consider $\mathtt{dec}_A(\tw_{sr})$;
	\item[(DP2)] Add a $\tn$ decoration on the edge  $(1,1')$, if $\tw \in $ (LP) or $(\LRPZ)$;
	\item[(DP3)] Add a $\tb$ decoration on the edge  $(n+2, (n+2)')$, if $\tw \in $ (RP) or $(\LRPZ)$;
	\item[(DP4)] Denote by $\mathtt{dec}_P(\tw)$ the obtained decorated diagram. 
	\end{enumerate}
\end{Definition}

\begin{Theorem}
Let $\tw \in (\LP), (\RP)$ or $(\LRPZ)$. Then $\mathtt{dec}_P(\tw)=\tilde{\theta}(b_{\tw})$.
\end{Theorem}

\begin{proof}
Let us consider first a left peak $\tw$. In this case a reduced expression for $\tw$ can be written as 
\begin{equation}\label{red-expression-LP}
{\bf \tw}_1 \cdot \ts_{j_\ell} \cdots \ts_1 \cdots \ts_{j_\ell} \cdot {\bf \tw}_2, 
\end{equation}
where ${\bf \tw}_1$ and ${\bf \tw}_2$ are alternating expressions non containing any $\ts_i$ for $i\leq j_\ell$.
Now the product $\td_{j_\ell} \cdots \td_1 \cdots \td_{j_\ell}$ is equal to the simple diagram $\td_{j_\ell}$ with a $\tn$ decoration on its first vertical edge $(1,1')$. Since ${\bf \tilde{w}}_1$ and ${\bf \tilde{w}}_2$  do not contain any $\ts_{i}$ for $i\leq j_\ell$, then 
$$ \tilde{\theta}(b_{\tw})=\td({\bf \tw}_1) \cdot \td_{j_\ell} \cdots \td_1 \cdots \td_{j_\ell} \cdot \td({\bf \tw}_2),$$ 
can be obtained by simply adding a $\tn$ decoration on the edge $(1,1')$ of 
$$\td({\bf \tw}_1) \cdot \td_{j_\ell}  \cdot \td({\bf \tw}_2)=\tilde{\theta}(b_{\tw_{sr}}).$$
By Theorem~\ref{thm-ALT}, $\mathtt{dec}_A(\tw_{sr})=\tilde{\theta}(b_{{\tw}_{sr}})$, and (DP2) adds a $\tn$ decoration on its edge $(1,1')$. Hence  $\mathtt{dec}_P(\tw)=\tilde{\theta}(b_{\tw})$.
\smallskip

For (RP) the proof holds using a symmetric argument, while for $(\LRPZ)$ a similar argument can be applied considering first the left peak and then the right peak. 
\end{proof}

\begin{Corollary}\label{cor:immagineP}
Let $\tw \in  (\LP), (\RP)$ or $(\LRPZ)$ and $d=\mathtt{dec}_P(\tw)$. Then
\begin{itemize}
    \item[(a)] if $\tw\in (\LP)$, then $d$ starts with $j_{\ell}-1$ vertical edges and has a $\tn$ on the edge $(1,1')$; 
    \item[(b)] if $\tw\in (\RP)$, then $d$ ends with $j_{r}-1$ vertical edges and has a $\tb$ on the edge $(n+2,(n+2)')$.
    \item[(c)] if $\tw\in (\LRPZ)$, then $d$ starts and ends with $j_{\ell}-1$ and $j_{r}-1$ vertical edges, respectively; it has a $\tn$ on $(1,1')$, a $\tb$ on $(n+2,(n+2)')$, and ${\bf a}(d)>1$. 
\end{itemize}
\end{Corollary}
\begin{proof}
The only non trivial fact is that for $\tw\in (\LRPZ)$, ${\bf a}(d)>1$. In order to see this, we observe that the Hasse diagram of the reduction of the associated heap is connected and it is never an up or down path. 
\end{proof}

\subsection{Zigzag and pseudo-zigzag}

In this section, we analyze the remaining families of fully commutative heaps, namely the Zigzag $(\ZZ)$, and the Pseudo-zigzag $(\PZZ)$ of Definition~\ref{def:famillesCtilde}.
\smallskip

Observe that if $H=H(\tw)$ with $\tw \in (\ZZ)$ or $(\PZZ)$ and $d(\tw)$ is the corresponding diagram in $\GD(A_{n+1})$, then ${\bf a}(d(\tw))=1$ since $\tau({\tw})$ is of type $s_is_{i+1}\cdots s_j$ or $s_is_{i-1}\cdots s_j$, as can be seen by applying Fork elimination to all the forks appearing in $H$, (see Figure~\ref{Example-N4}).

\begin{Definition}[Decoration algorithm for $(\ZZ)$ and $(\PZZ)$]\label{algdeczz}
Let $\tw \in (\ZZ)$ or $(\PZZ)$.
\begin{enumerate}
	\item[(ZZ1)] Consider $d_{\tau(\tw)}$;
	\item[(ZZ2)] Add a $\tn$ (resp. $\tb$) decoration on the first (resp. last) propagating edge, for each occurrence of $\ts_1$ (resp. $\ts_{n+1}$) in $H(\tw)$;
	\item[(ZZ3)] If $\ts_i = \ts_1$ (resp. $=\ts_{n+1}$), then replace the highest $\tn$ (resp. $\tb$ ) in the first (resp. last) propagating edge  by a $\bullet$ (resp. $\circ$), and add a $\bullet$ (resp. $\circ$) to the edge $(1,2)$ (resp. $(n+1,n+2)$). 
	 \item[(ZZ4)] If $\ts_j = \ts_1$ (resp. $=\ts_{n+1}$), then replace the lower  $\tn$ (resp. $\tb$)  of the first (resp. last) propagating edge by a $\bullet$ (resp. $\circ$), and add a $\bullet$ (resp. $\circ$) to the edge $(1',2')$ (resp. $((n+1)',(n+2)')$). 
	 \item[(ZZ5)] Assign to any black and white decoration vertical positions compatible with the relative  order of the corresponding nodes labeled $\ts_1$ and $\ts_{n+1}$ in $H(\tw)$ (see Figure \ref{fig:ziges2});  
	\item[(ZZ6)] Denote by $\mathtt{dec}_Z(\tw)$ the obtained decorated diagram. 
	\end{enumerate}
\end{Definition}
 
\begin{figure}[h]
	\centering
	\includegraphics[width=0.8\linewidth]{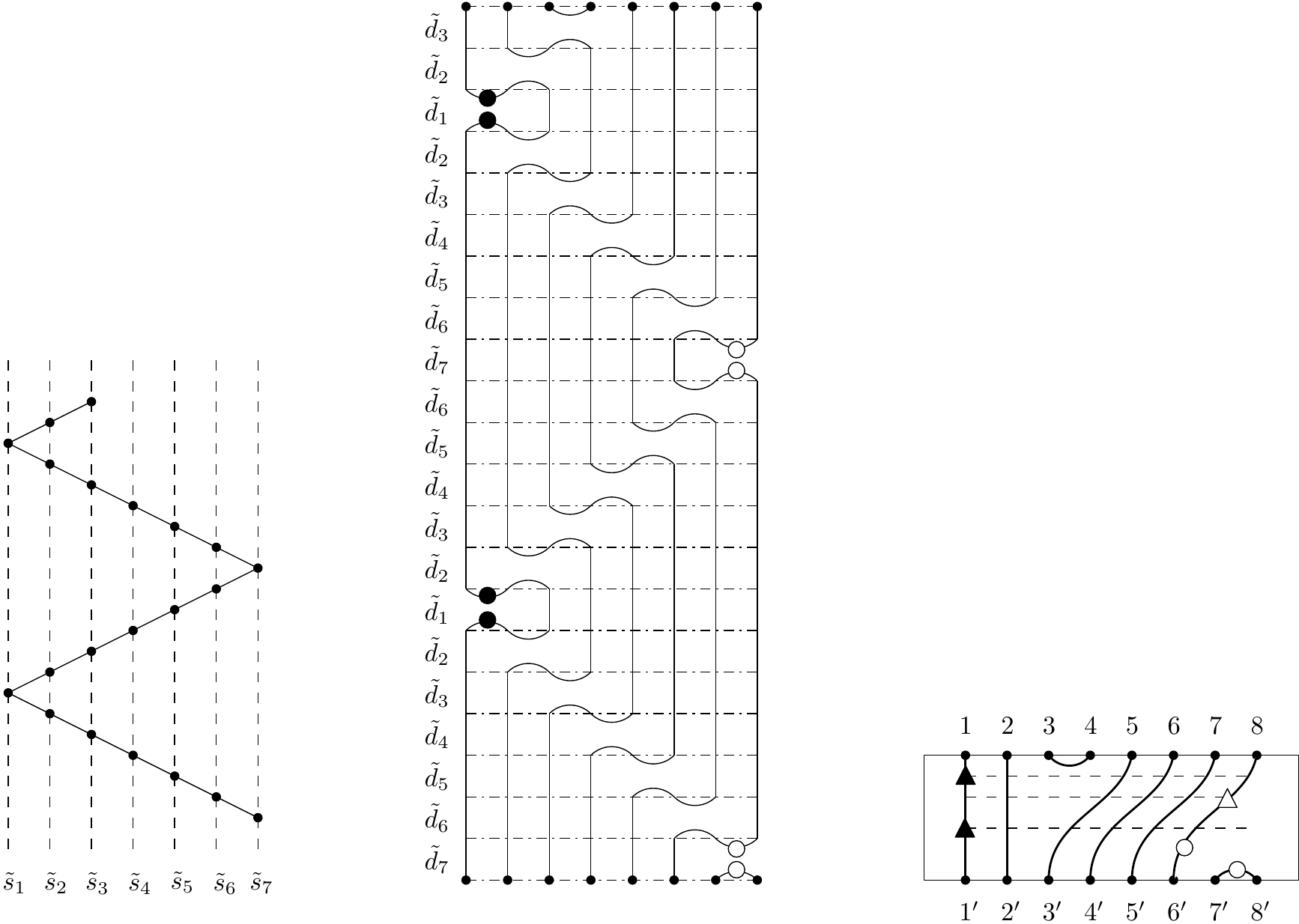}
	\caption{}
	\label{fig:ziges2}
\end{figure}

\begin{Theorem}
Let $\tw \in (\ZZ)$ or $(\PZZ)$. Then $\mathtt{dec}_Z(\tw)=\tilde{\theta}(b_{\tw})$.
\end{Theorem}

\begin{proof}
As already observed, $\tau({\tw})$ is of type $s_is_{i+1}\cdots s_j$ or $s_is_{i-1}\cdots s_j$, hence ${\bf a}(d_{\tau(\tw)})=1$ (see Figure~\ref{fig:ziges2}, right without decorations). By applying the decoration algorithm above, we obtain exactly a decorated diagram of the form of Figure~\ref{fig:ziges2}, right, where the cardinality of the $\tn$ (resp. $\tb$) decorations equals the number of occurrences of $\ts_1$ (resp. $\ts_{n+1}$). The vertical positions of such decorations alternate and the highest one is a $\tn$ (resp. $\tb$) decoration if $\ts_1$ (resp. $\ts_{n+1}$) occurs first in the unique reduced expression of $\tw$. The alternating condition implies that the $\tn$ (resp. $\tb$) decorations in the first (resp. last) propagating edge can not be joined to form a larger block, so each one forms a single block. 
\smallskip

Since the proofs are analogous in the other cases, for the sake of simplicity, we suppose that the first generator of $\tw$ is $\ts_i$, the last is $\ts_j$ with $2<i < j < n+1$  and that the first occurrence of $\ts_{n+1}$ is before $\ts_1$ (see for example Figure~\ref{fig:zig1}, left). The product $\tilde{\theta}(b_{\tw})$ associated to the zigzag can be written as 
\begin{equation}\label{red-expression-ZZ}
(\td_i \cdots \td_{n+1} \cdots \td_i) \cdot (\td_{i-1} \cdots \td_1 \cdots \td_{i-1} ) \cdot {\bf \td} \cdot {\bf \td}_f,
\end{equation}
where ${\bf \td}$ is a finite alternating product of the first and the second term of \eqref{red-expression-ZZ}, and ${\bf \td}_f$ is equal to a product of consecutive simple diagrams ending with $\td_j$, corresponding to the remaining part of the zigzag. 

Now, we notice that the first factor in \eqref{red-expression-ZZ} is the simple diagram $\td_i$ with a unique $\tb$ decoration on the edge $(n+2, (n+2)')$, while the second factor is $\td_{i-1}$ with a $\tn$ decoration on its edge $(1,1')$. Their product is $\td_i \td_{i-1}$ with a $\tb$ decoration on $(n+2, (n+2)')$ and $\tn$ decoration on $(1,1')$ in a lower vertical position with respect to $\tb$. 

Now, by multiplying again by the first factor, we would obtain the simple diagram $\td_i$ with two  $\tb$ decorations on $(n+2, (n+2)')$ and one $\tn$ decoration on $(1,1')$ with a vertical position in between those of the two $\tb$. If we keep multiplying by the remaining factors in ${\bf \td}$ analogously, we obtain a decorated diagram where the first edge has only $\tn$ and the last only $\tb$, moreover the vertical positions of $\tn$ and $\tb$ alternate. Finally, the multiplication with the last factor ${\bf \td}_f$ does not add any decoration to the obtained diagram, but it replaces its non-propagating edge $((i-1)',i')$ with $(j',(j+1)')$, confirming the fact that the undecorated diagram associated to $\tilde{\theta}(b_{\tw})$ is $d_{\tau(\tw)}$. Thus the decorations coincide with those added by the Decoration algorithm~\ref{algdeczz} and $\mathtt{dec}_Z(\tw)=\tilde{\theta}(b_{\tw})$.
\end{proof}

\begin{Corollary}\label{cor:immagineZ}
If $\tw\in (\ZZ)$ or $(\PZZ)$, then $d=\mathtt{dec}_Z(\tw)$ is a decorated diagram with $a(d)=1$, and it has at least a $\tn$ on the first propagating edge and at least a $\tb$ on the last propagating edge. 
\end{Corollary}

\section{Proof of  the injectivity of $\tilde{\theta}$ }\label{sec:inj}
In this section we prove that the map $\tilde{\theta}$ is injective. 

\begin{Theorem}\label{Teo-alternating}
	Let $\tu \neq \tw$ be two elements in $(\ALT)$. Then  $\mathtt{dec}_A(\tu)$ and $\mathtt{dec}_A(\tw)$ are distinct. 
\end{Theorem}

\begin{proof}
If $d_{\tau(\tu)} \neq d_{\tau(\tw)}$ then $\mathtt{dec}_A(\tu) \neq \mathtt{dec}_A(\tw)$.
If $d_{\tau(\tu)}=d_{\tau(\tw)}$ and $\alpha(\tu)\not=\alpha(\tw)$, we are also done, since $\mathtt{dec}_A(\tu)$ and $\mathtt{dec}_A(\tw)$ have a different number of decorated loops.
\smallskip

	So suppose $d_{\tau(\tu)}=d_{\tau(\tw)}$ and $\alpha(\tu)=\alpha(\tw)$. If the two heaps $H(\tu)$ and $H(\tw)$ have different numbers either of left forks or right forks, then  $\mathtt{dec}_A(\tu)$  and $\mathtt{dec}_A(\tw)$ have a different numbers either of $\tn$ or  $\tb$ decorations and so they are distinct.
\smallskip
	
Therefore suppose that they have the same numbers of left forks say $f^{\tu}_1,\ldots, f^{\tu}_p$, resp. $f^{\tw}_1,\ldots, f^{\tw}_p$ and right forks $f^{\tu}_{p+1},\ldots, f^{\tu}_q$ resp. $f^{\tw}_{p+1},\ldots, f^{\tw}_q$ labeled from top to bottom. 
Since $\alpha(\tu)=\alpha(\tw)$ and decorated loops are in one-to-one correspondence with cycles $\gamma_\infty$ inside the heap, we might assume that all the forks define snake paths which are not cycles. 

Since $H(\tu)$ and $H(\tw)$ are distinct, there exists $1\leq i \leq q$ such that  $\gamma(H(u),f_i^{\tu}) \neq \gamma(H(w),f_i^{\tw})$.
In fact, suppose by contradiction that they coincide for all $i$. Being $H(\tu)$ and $H(\tw)$ distinct, they can only differ in the subheaps which does not involve the nodes  contained in the support of  $\gamma(H(u),f_i^{\tu}) = \gamma(H(w),f_i^{\tw})$, for all $i=1,\dots, q$, namely those giving rise to snakes paths of the form $\gamma_{\vec{x,y}}$ (defined after Example~\ref{rem:singlepath0}). Since the  reduction does not involve such subheaps and does not disconnect them from the supports of the reduced heaps, after applying the  Reduction algorithm~\ref{reductionalgorithm} to both heaps we necessarily have  $\mathfrak{R}(H(\tu)) \neq \mathfrak{R}(H(\tw))$, which is in contradiction with $d_{\tau(\tu)}=d_{\tau(\tw)}$.
\smallskip

So, fix the smallest $1\leq i \leq q$ such that  $\gamma_{\tilde{u}}=\gamma(H(\tu),f_i^{\tu}) \neq \gamma(H(\tw),f_i^{\tw})=\gamma_{\tilde{w}}$, and set $(h^{\tu},k^{\tu})_{\gamma_{\tilde{u}}}$ and $(h^{\tw},k^{\tw})_{\gamma_{\tilde{w}}}$ the corresponding edges in $d_{\tau(\tu)}=d_{\tau(\tw)}$ (see Definition~\ref{def-hk}). 
Therefore, either $(h^{\tu},k^{\tu})_{\gamma_{\tilde{u}}}$ and $(h^{\tw},k^{\tw})_{\gamma_{\tilde{w}}}$ are different edges of $d_{\tau(u)}$, or they coincide but they are decorated in a different way in $\mathtt{dec}_A(\tu)$ and $\mathtt{dec}_A(\tw)$. 
In the latter case we are done. 

So suppose $(h^{\tu},k^{\tu})_{\gamma_{\tilde{u}}} \neq (h^{\tw},k^{\tw})_{\gamma_{\tilde{w}}}$. Then 
$h^{\tu} \neq h^{\tw}$, let us say $h^{\tu} \succ_\ell h^{\tw}$ and assume that $f_i^{\tu}$ and $f_i^{\tw}$ are left forks. In $\mathtt{dec}_A(\tu)$ there is at least a $\tn$ decoration in the edge $(h^{\tu},k^{\tu})_{\gamma_{\tilde{u}}}$, while in $\mathtt{dec}_A(\tw)$, by the minimality of $i$, by $h^{\tu} \succ_\ell h^{\tw}$ and Remark~\ref{ordine-decorazioni} (b), there can not be black decorations $\tn$ in the edge $(h^{\tu},k^{\tu})_{\gamma_{\tilde{u}}}$, hence once again $\mathtt{dec}_A(\tu) \neq \mathtt{dec}_A(\tw)$. The same argument works for right forks.
\end{proof}

\begin{Theorem}\label{Teo-peaks}
Let $\tu \neq\tw$ be two elements in $(\LP), (\RP)$, or $(\LRPZ)$. Then  $\mathtt{dec}_P(\tu)$ and $\mathtt{dec}_P(\tw)$ are distinct. 
\end{Theorem}
\begin{proof}
If $d_{\tau(\tu)}$ and $d_{\tau(\tw)}$ have a different number of vertical edges either on the left side or on the right side we are done. So suppose that $d_{\tau(\tu)}$ and $d_{\tau(\tw)}$  have the same number of vertical edges in both sides.
Since $H(\tu)$ and $H(\tw)$ are distinct, the alternating  semireduced  heap $H_{sr}(\tu)$ and $H_{sr}(\tw)$ are distinct. Therefore, by Theorem~\ref{Teo-alternating} and the Decoration algorithm~\ref{algdecpeak},  $\mathtt{dec}_P(\tu) \neq \mathtt{dec}_P(\tw)$.
\end{proof}

\begin{Theorem}\label{Teo-ZZ}
Let $\tu \neq \tw$ be two elements in $(\ZZ)$ or $(\PZZ)$. Then  $\mathtt{dec}_Z(\tu)$ and $\mathtt{dec}_Z(\tw)$ are distinct. \end{Theorem}

\begin{proof} Suppose that $\ts_{i_1} \cdots\ts_{j_1}$ and $\ts_{i_2} \cdots\ts_{j_2}$ are reduced expression for $\tu$ and $\tw$.
	If $(\ts_{i_1},\ts_{j_1})\not=(\ts_{i_2},\ts_{j_2})$, then $d_{\tau(\tu)} \neq d_{\tau(\tw)}$ and we are done.	
	If $(\ts_{i_1},\ts_{j_1})=(\ts_{i_2},\ts_{j_2})$, then $d_{\tau(\tu)} = d_{\tau(\tw)}$, since $\tu \neq \tw$  either they have a different number of black or white decorations, or these numbers are the same, but the highest decorations is a $\tn$ in one diagram and a $\tb$ in the other. 
\end{proof}

The proof and the importance of the vertical positions of the decorations can be easily understood by looking at the picture below. The two heaps in Figure~\ref{fig:zig1} are associated to the decorated diagrams in Figure~\ref{decorated-Ex2}, respectively. The vertical positions in Figure~\ref{decorated-Ex2} are needed to distinguish them, and allows the map $\mathtt{dec}_Z$ to be injective.

\begin{figure}[h]
	\centering
	\includegraphics[width=0.5\linewidth]{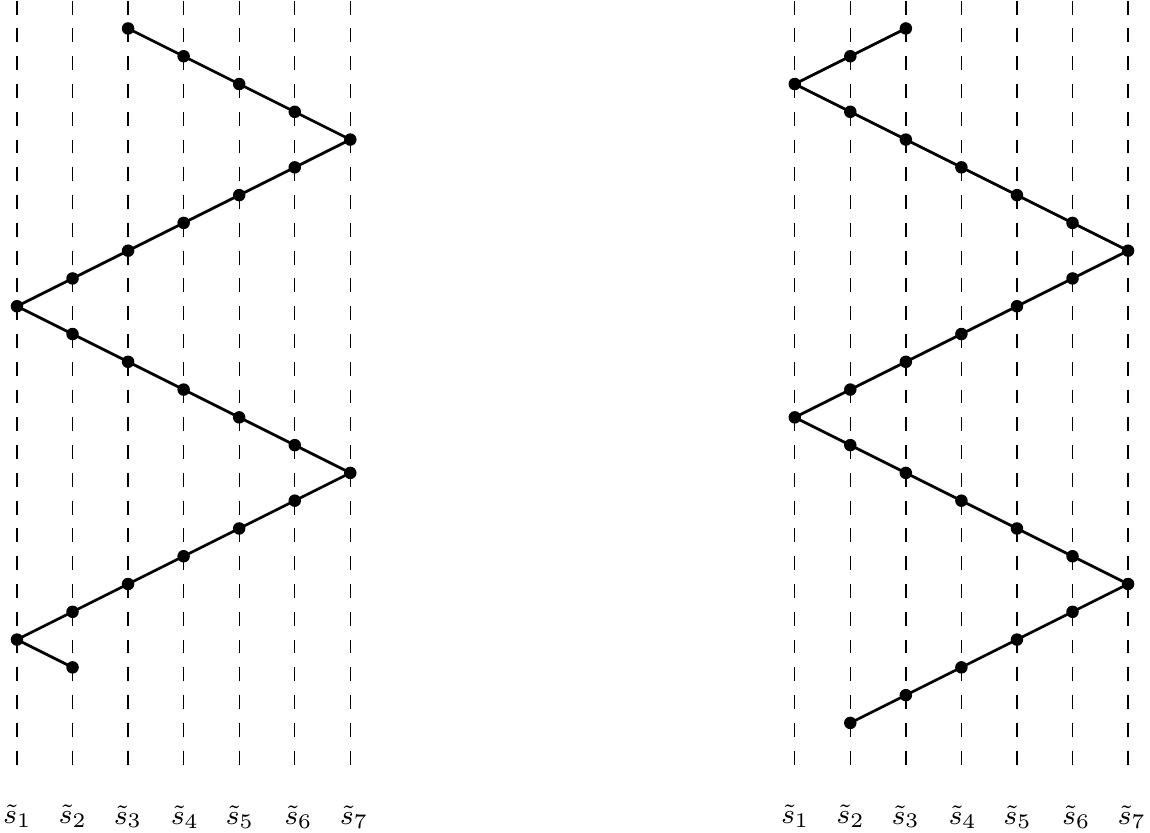}
	\caption{The heaps associated with the decorated diagrams in Figure~\ref{decorated-Ex2}.}
	\label{fig:zig1}
\end{figure}

\begin{Definition}
    Let $\mathtt{dec}: \FC(\tC_{n+1}) \longrightarrow \GD(\tC_{n+1})$ the map defined by 
    $$\mathtt{dec}(\tw):=\left \{ \begin{array}{ll}
         \mathtt{dec}_A(\tw)& \mbox{if} \quad  \tw \in (\ALT) \\
         \mathtt{dec}_P(\tw)& \mbox{if} \quad \tw \in (\LP), (\RP), (\LRPZ) \\ 
         \mathtt{dec}_Z(\tw)& \mbox{if} \quad \tw \in (\ZZ), (\PZZ).
    \end{array}\right.$$
\end{Definition}

\begin{Theorem}
    The map $\mathtt{dec}$ is injective.
\end{Theorem}
\begin{proof}
Let $\tu$, $\tw \in \FC(\tC_{n+1})$. If both $\tu, \tw$ are in $(\ALT), (\LP), (\RP), (\ZZ), (\LRPZ)$ or $(\PZZ)$, then we already proved that $\mathtt{dec}(\tu) \neq \mathtt{dec}(\tw)$. If they belong to different classes then by Corollaries~\ref{cor:immagineA}, \ref{cor:immagineP}, and \ref{cor:immagineZ}, $\mathtt{dec}(\tu)$ and $\mathtt{dec}(\tw)$ belongs to disjoint families of decorated diagrams and so they are distinct.
\end{proof}

From the previous theorem, it follows that the image through $\tilde{\theta}$  of the monomial basis $\{b_{\tw} \mid \tw \in \FC(\tC_{n+1})\}$ is an independent set in $\GD(\tC_{n+1})$, and so $\tilde{\theta}$ is an injective map.

\section{A constructive approach to  the surjectivity of $\tilde{\theta}$}\label{surjectivity}


Even though the  map $\tilde{\theta}\colon \TL(\tC_n)\longrightarrow \GD(\tC_n) $ is surjective by definition, starting from an admissible diagram $d\in\GD(\tC_n)$, it is not easy to recover  the unique element $\tw \in \FC(\tC_{n})$ such that $ \tilde{\theta}(b_{\tw})=d$. In this section, using our map  $\mathtt{dec}$, we solve this problem giving an algorithmic construction of such an element.  In order to do this we first need to introduce some definitions and technical lemmas.

\begin{Definition}
Given an admissibile diagram $d \in \GD(\tC_n)$ we say that:
\begin{itemize}
\item [(a)]  $d$ is a $\mathcal{Z}$-diagram if ${\bf a}(d)=1$, it has at least a $\tn$ on the first propagating edge  and at least a $\tb$ on the last propagating edge. 
\item [(b)] $d$ is a $\mathcal{P}$-diagram  if it is not a $\mathcal{Z}$-diagram and it has either  the edge $(1,1')$ with a $\tn$ or the edge $(n+2,(n+2)')$ with a $\tb$ or both. If $d$ has a $\tn$ (resp. $\tb$) on the edge $(1,1')$ (resp. $(n+2,(n+2)')$ then we set $j_{\ell}-1$ (resp. $j_{r}-1$) the number of its first (resp. last) vertical edges.
\item [(c)]$d$ is an $\mathcal{A}$-diagram  if it is  neither $\mathcal{Z}$-diagram nor a $\mathcal{P}$-diagram.
\end{itemize}
\end{Definition}

\begin{Lemma}\label{lem:casobase}
If $d$ is an $\mathcal{A}$-diagram with no $\tn$ and $\tb$ decorations, then there exists an alternating heap $H\in (\ALT)$ such that $\mathtt{dec}(H)=d$.
\end{Lemma}
\begin{proof}
The undecorated diagram $d'$ associated with $d$, obtained by erasing only $\bullet$ and $\circ$ decorations on the edges starting from nodes $1,1'$, $n+2$ or $(n+2)'$ is a loop-free diagram in $\GD(A_{n+1})$. By Remark~\ref{rem:TLA} there exists a unique $w \in \FC(A_{n+1})$ such that $\theta(b_w)=d'$. We know that any reduced expression $s_{i_1}\cdots s_{i_r}$ of $w$ is alternating,  contains at most one generator labeled by $s_1$ and at most one labeled by $s_{n+1}$. Hence the element $\tw:=\ts_{i_1}\cdots \ts_{i_r}$ is in $(\ALT)$, $d(\tw)$ is equal to $d'$ and since in $H(\tw)$ there are no forks, step (D4) of the Decoration algorithm~\ref{algorithmdecoration} adds $\bullet$ and $\circ$ in a way that $\mathtt{dec}(H(\tw))=d$. 
\end{proof}

The following algorithm associates to any edge $e$ of an admissible $\mathcal{A}$-diagram $d$, decorated with at least a $\tn$ or a $\tb$, a heap $H_e \in (\ALT)$, containing a fork $f$ such that:
\begin{itemize}
    \item[(a)] $\mathtt{dec}(H_e)$ has the edge $e$ decorated as in $d$, and all the remaining edges are non-propagating of type $(i,i+1)$, $(i',(i+1)')$ or vertical edges $(i,i')$. 
    \item[(b)] if $\gamma= \gamma(H_e,f)$ then $(h,k)_\gamma=e$;
        
\end{itemize}

To construct such a heap, we consider a plane grid $\mathcal{G}$ with $n+1$ columns and infinite rows, where each vertex in column $i$ is labeled $\ts_i$.

\begin{Definition} \label{def-He}
Consider an admissible $(\ALT)$-diagram $d\in \GD(\tC_n)$ with an edge $e$ decorated with at least a $\tn$ or a $\tb$. 
\begin{itemize}
    \item[(H1)] If $e$ is a decorated loop, then set $H_e=H_{\infty}$ (see Definition~\ref{def:complete}).
    \item[(H2)] Let $e=(h,k)$ be a non-propagating edge in the north face with $h\succ k$ and a $\tn$ decoration. Observe that necessarily $h$ is odd and $k$ is even, and $h\geq 3$. 
    \begin{itemize}
        \item[1.] Define an up-down path (see Section~\ref{sec:snakes_paths}) in the grid $\mathcal{G}$ starting in a node $B_x$ labeled $\ts_x$, $x=h-1$, having a left-up first step and going left for $(h-2)$-steps until the node $\ts_1$. Being $h-1$ even the step reaching the node $\ts_1$ is left-up.
         \item[2.] Consider a left fork $f_l$ having as top node the $\ts_1$ of (H2.1).
         \item[3.] Construct an up-down path starting from the bottom node of the fork $f_l$ with first step right-up and ending in a node $B_y$ labeled by $\ts_y$ with $y=k-1$ if there is only the $\tn$-decoration, and reaching  a node $\ts_{n+1}$ if there is also a $\tb$-decoration. In both cases the last step of the path is right-down. 
         \item[4.] In the latter case consider a right fork $f_r$ with bottom node the $\ts_{n+1}$ of (H2.3).
         \item[5.] Construct an up-down path starting from the top node of the right fork $f_r$,  with first step left-down and ending in a node $B_y$ labeled by $\ts_y$ with $y=k$.
         \item[6.] Set $H_e$ the alternating heap associated to the snake path from $B_x$ to $B_y$. See $H_{(5,12)}$ and $H_{(10',7')}$ in Figure~\ref{Surj-1}.
         \end{itemize}
    \item[(H3)] Let $e=(h,k)$ with $h$ in the north face and $k$ in the south face and with $p$ $\tn$ decorations and $q$ $\tb$ decorations. All white and black decorations alternate hence $q=p$ or $q=p\pm 1$. Moreover $h$ and $k$ are odd, and $h\geq 3$.
\smallskip

    Suppose first that $p=1$ and $q=0$.  
  \begin{itemize}
        \item[1.] Define an up-down path $\gamma_{start}$ starting from a node $B_x$ labeled $\ts_x$, $x=h-1$, having a left-up first step and going left for $(h-2)$-steps until the node $\ts_1$ inside a half-horizontal region. Being $h-1$ even the step reaching the node $\ts_1$ is left-up.  
        \item[2.] Define an up-down path $\gamma_{end}$ starting from a node $B_y$ labeled $\ts_y$, $y=k-1$, lying in the same horizontal line of $B_x$, having a left-down first step and going left for $(k-2)$-steps until it reaches the node $\ts_{1}$ with a left-down step. This is actually the node labeled $\ts_1$ right below the node $\ts_1$ of step (H3.1). These two nodes give rise to a left fork. 
        \item[3.] Set $H_e$ the alternating heap associated to the snake path from $B_x$ to $B_y$.
    \end{itemize}
   \smallskip
   
    Suppose now that $p,q\geq 1$, the highest decoration is $\tn$ and that the lowest is $\tb$, hence $q=p$. 
    \begin{itemize}
        \item[$1'$.] Define an up-down path $\gamma_{start}$ starting from a node $B_x$ labeled $\ts_x$, $x=h-1$, having a left-up first step and going left for $(h-2)$-steps until the node $\ts_1$ inside the same half-horizontal region. Being $h-1$ even the step reaching the node $\ts_1$ is left-up.
        \item[$2'$.] Define an up-down path $\gamma_{end}$ starting from a node $B_y$ labeld $\ts_y$, $y=k$, lying $2p-1$ horizontal lines below the horizontal line through $B_x$, having a right-down first step and going right for $(n+1-k)$-steps until the node $\ts_{n+1}$ inside the same half-horizontal region. Being $n+1$ even the step reaching the node $\ts_{n+1}$ is right-up.
        \item[$3'$.] Join $\gamma_{start}$ with $\gamma_{end}$ and with a sequence of $2p-1$ complete up-down paths from $\ts_1$ to $\ts_{n+1}$ and viceversa by obtaining a snake-path crossing $p$ left-forks and $q$ right forks. The complete up-down paths going from $\ts_1$ to $\ts_{n+1}$ start and end with a right-up step while those going backward from $\ts_{n+1}$ to $\ts_{1}$ start and end with a left-up step. 
        \item[$4'$.] Set $H_e$ the alternating heap associated to the snake path from $B_x$ to $B_y$.
         \end{itemize}         
\end{itemize} 
 
It is easy to verify that properties (a) and (b) listed above holds for the heap $H_e$.

There are other cases to be considered: $e$ is non-propagating in the north face and it has only a $\tb$ decoration or it is non-propagating in the south-face with a $\tn$, a $\tb$ or both; $e$ is propagating with only a $\tb$ decoration or it starts or ends with a different decoration with respect to the treated case. In all these situations, similar algorithms can be defined to obtain heaps $H_e$ with the properties (a) and (b) listed above. Some of these examples are represented in Figures~\ref{Partial-Heaps-H} and \ref{Surj-2}.
\end{Definition}

\begin{figure}[h]
	\centering
	\includegraphics[width=0.8\linewidth]{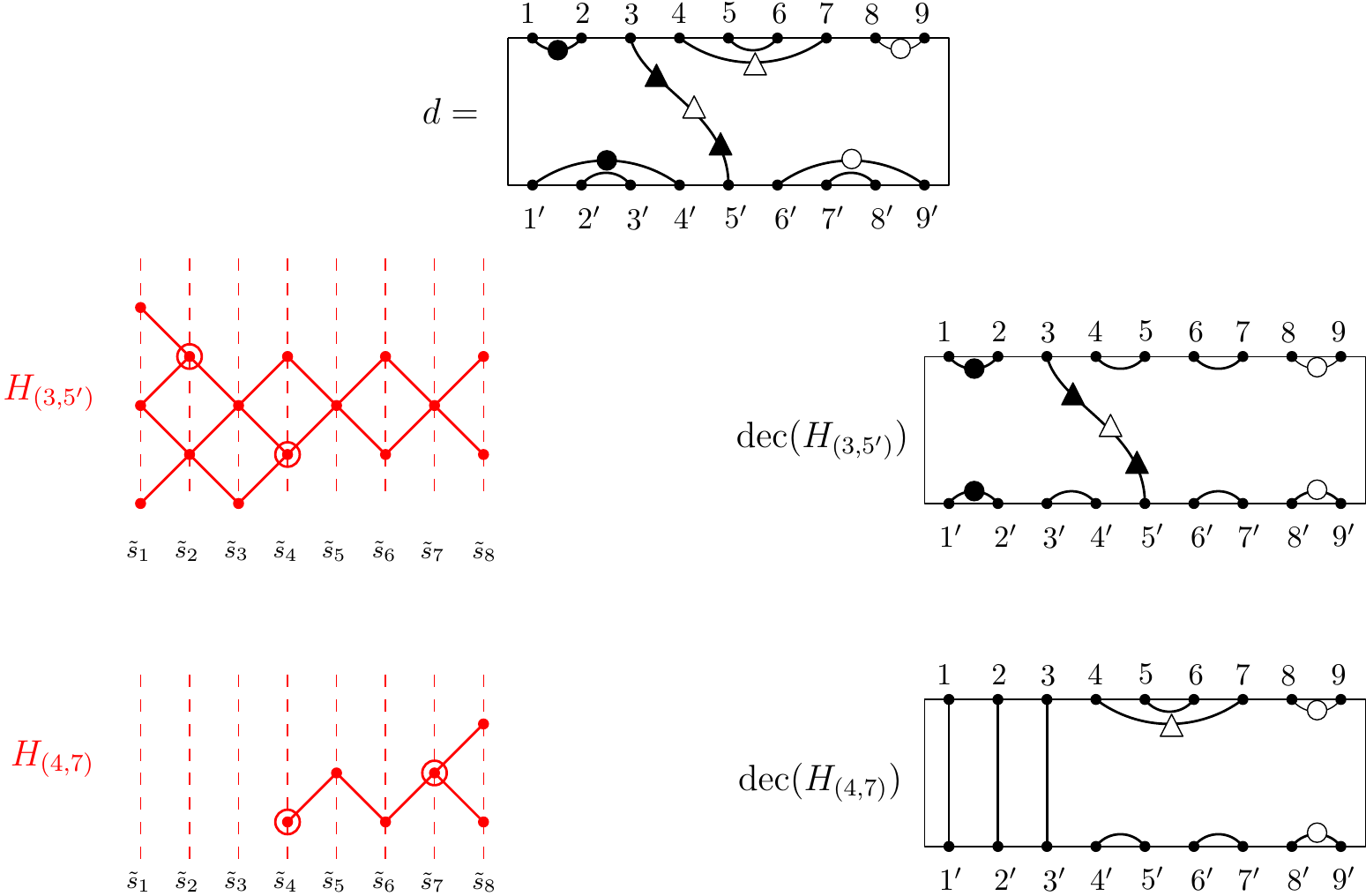}
	\caption{The alternating heaps $H_{(3,5')}$ and $H_{(4,7)}$.}
	\label{Partial-Heaps-H}
\end{figure}

Consider a $\mathcal{A}$-diagram $d$ with an edge $e$ decorated with at least a $\tn$ or a $\tb$. We recall that the only possible loops in $d$ are of the form $\lott$.
\begin{Definition}
We denote by $d_{\check{e}}$ the diagram obtained from $d$ by erasing $e$ if $e=\lott$ or by removing all the decorations on $e$ otherwise.
\end{Definition}
Suppose there exists an alternating heap $\check{H}\in (\ALT)$ such that $\mathtt{dec}(\check{H})=d_{\check{e}}$. By Proposition~\ref{prop:E}, we can find $\gamma \in \Snakes(\check{H})$ such that $E(\gamma)=e$.  

 Since the edge $e$ of $d$ has at least either a $\tn$ or a $\tb$ then it can  be deformed so as to take the black decorations  to the left wall of the diagram and the white decorations to the right wall simultaneously without crossing any other edge. All these  observations allow us to give the following definition.

\begin{Definition} Let $d$ be a $\mathcal{A}$-diagram with an edge $e$ decorated with at least a $\tn$ or a $\tb$. 
\begin{itemize}
\item[(a)]
A non-propagating edge $(h,k)$ (resp. $(h',k')$) of $d_{\check{e}}$ is \emph{above} (resp. \emph{below}) $e$ if at least one of the vertical lines $hh'$, $kk'$ crosses the deformation of $e$ in $d$. If the  edge $e$ has only black (resp. white) decorations then a non-propagating edge can  be neither above nor below $e$.
\item[(b)] If $e\neq \lott$ is  non-propagating on the north (resp. south) face, we say that a propagating edge is \emph{below} (resp. \emph{above}) $e$ if one of its  nodes is on the left of $e$ and the other one is on the  right of $e$.
\item[(c)]  We will denote by $SP_N(e)$ (resp. $SP_S(e)$)  the set of nodes of all $\gamma \in \Snakes(\Hc)$ such that $E(\gamma)$ is an edge above (resp. below) $e$.
 \item[(d)] The set $SP(e)=SP_N(e) \cap SP_S(e)$ is the set of the {\em  splitting nodes} of $e$. 
 \end{itemize}
\end{Definition}

\begin{figure}[h]
	\centering
	\includegraphics[width=0.9\linewidth]{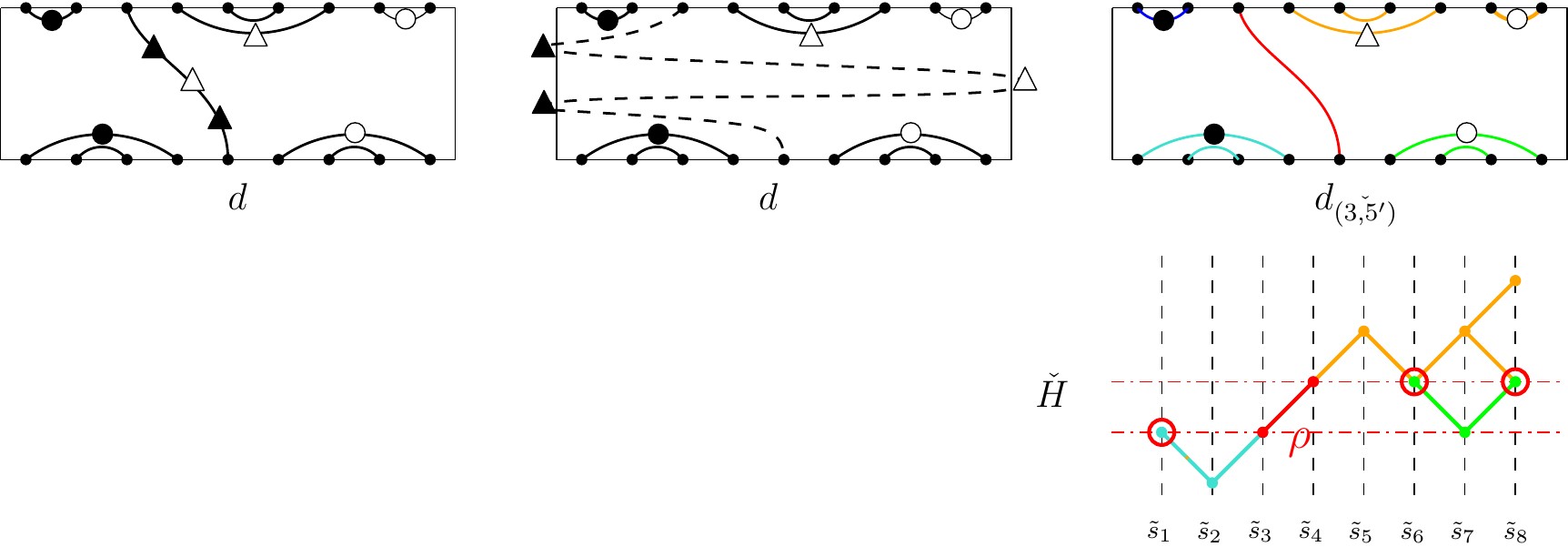}
	\caption{The splitting nodes in $SP((3,5'))$ are circled. The diagram $d_{\check{(3,5')}}$ is $\mathtt{dec}(\check{H})$.}
	\label{Splitting-nodes}
\end{figure}

\begin{Lemma} Let $d$ be an $\mathcal{A}$-diagram with an edge $e$ decorated with at least  a $\tn$ or a $\tb$. Let $\check{H} \in (ALT)$ such that $\mathtt{dec}(\check{H})=d_{\check{e}}$. Finally, let $\rho \in \Snakes(\check{H})$ such that $E(\rho)=e$. Then
\begin{itemize} 
\item[(a)]  
\begin{enumerate}
    \item[1.]   If  $e$ is the unique decorated loop in  $d$, then $SP(e)=\{\ts_1, \ts_3,\ldots, \ts_{n+1}\}$ and these nodes can be depicted in the same horizontal line.  
    \item[2.] If $e$ is neither a loop nor a vertical edge, then the splitting nodes on the left (resp. on the right) of $\rho$ can be depicted in the horizontal line passing through the leftmost (resp. rightmost) node of $\rho$.
    \item[3.] If $e$ is a vertical edge, then the splitting nodes can be depicted in the same horizontal line. 
\end{enumerate}
\item[(b)] The north and the south border of the heap $H_e$, in Definition~\ref{def-He}, contain nodes labeled as the splitting nodes and nodes labeled as those of the snake path $\rho$. 
\end{itemize}
\end{Lemma}

\begin{proof} 
\begin{itemize}
\item[(a.1)] Suppose that $e$ is the unique decorated loop in $d$. Hence in $d$ and so in $d_{\check{e}}$ there are no propagating edges, $n$ is even, and in both $SP_N(e)$ and $SP_S(e)$ there is at least an occurrence of a node labeled $\ts_i$ for  $i=1,3,\ldots, n-1, n+1$. This is immediate since all the edges in the north (resp. south) face are above (resp. below) the edge $e$. 

If one of these $\ts_i$, $i$ odd, is an isolated node, then $\ts_i\in \gamma_{i} \cap \gamma_{i'} \in \Snakes(\check{H})$, hence $\ts_i\in SP(e)$. 

Thus, consider the lowest and the highest occurrences of a non isolated node $\ts_i$ with $i$ odd in $SP_N(e)$ and  $SP_S(e)$, respectively. This condition yields that no other nodes $\ts_i$ in $\check{H}$ are  between them. If they coincide, then $\ts_i \in SP(e)$ and we are done. If not, since $\check{H}$ is alternating and $SP_N(e)$ and $SP_S(e)$ contain at least an occurrence of any odd generators, one can prove that in $\check{H}$ there is either a complete horizontal subheap that would produce another loop in $d_{\check{e}}$ against the uniqueness, or a snake path $\gamma$ with $E(\gamma)$ propagating, which is not possible. 

Now, we show that no node $\ts_i$ with $i$ even is in $SP(e)$. In fact, if this happens, since $SP(e)$ contains all odd nodes, then there would exist two snake paths corresponding to non-propagating edges above and  below $e$, respectively, with a common edge $\ts_i\ts_{i+1}$ or $\ts_{i-1}\ts_i$, which is not possible. 

Moreover, all the splitting nodes can be depicted  in the same horizontal line in each connected component. If not, there would exist in $\check{H}$ an increasing or decreasing  chain of type $\ts_i -\ts_{i+1}-\ts_{i+2}$, with $\ts_i, \ts_{i+2} \in SP(e)$ in $\check{H}$ giving rise to a propagating edge in $d_{\check{e}}$, which is not possible.

\item[(a.2)] Let $e=(h,k) \in d$, $k\neq h'$ decorated with $\tn$ (resp. $\tb$). 
If there is  a connected component $K$ in $\check{H}$ on the left (resp. right) of the connected component $K_{\rho}$ containing $\rho$, then  the leftmost and rightmost nodes of $K$ are labeled by $\ts_{x}$ and $\ts_{y}$ with $x$ and $y$ of the same parity: otherwise there would be a propagating edge with starting and ending  node  on the left (resp. right) of $e$ and so it would not be possible to deform $e$ to the left (resp. right) side of $d$.
 
Suppose that there exist two consecutive splitting nodes on the left of $\rho$ labeled by $\ts_p$ and $\ts_{q}.$ 
 By definition they are common nodes of paths corresponding to non-propagating edges on the left (resp. right) of $e$. If $\ts_p$ and $\ts_q$ are in two  distinct connected components, then necessarily $q=p+2$, otherwise if  $q=p+2l$  with $l>1$ then at least  a vertical edge occurs in $d_{\check{e}}$ on the left (resp. right) of $e$ which is impossible. Furthermore,   being  $\ts_p$ and $\ts_{q} $ in distinct connected components they  can be depicted in the same horizontal line.

Now consider a splitting node $\ts_p$ in a connected component $K$ on the left (resp. on the right) of $K_\rho$ or in $K_\rho \setminus \rho$. Then we show that all the nodes in $K$ or $K_\rho \setminus \rho$ of the form $\ts_{p\pm 2l}$ and in the same horizontal line of $\ts_p$ in our representation of $\check{H}$ are also splitting nodes. We first prove the result for two of such consecutive nodes  $\ts_p $ and  $\ts_{p+2}$ in the same connected component $K$ and in the same horizontal line. Since $\check{H}$ is alternating,  $K$ is connected and $\ts_p$ is a splitting node, we may suppose that there exists a snake path in $\check{H}$ passing through  $\ts_p $, $\ts_{p+1}$ and  $\ts_{p+2}$ corresponding to a non-propagating edge on the south face  of the diagram $d_{\check{e}}$.  Then either there is a snake path in $\check{H} $ passing through  $\ts_p $, $\ts_{p+1}$ and  $\ts_{p+2}$ corresponding to a non-propagating edge in the north face  of $d_{\check{e}} $ and so $\ts_{p+2}$ is a splitting node  or the snake path corresponding to the non-propagating edge in the north  face of $d_{\check{e}}$ stops in $\ts_p$, therefore the non-propagating  edge of $d_{\check{e}}$ starting in $\ts_{p+2}$ in the north face corresponds to a snake path in $\check{H}$ through $\ts_{p+2}$ and so $\ts_{p+2}$ is a splitting node.

Assume now that the splitting node $\ts_p $ is in $K_\rho$ on the left (resp. right) of $\rho$. Observe that only one of the two snake paths defining $\ts_p $ can correspond to a propagating edge so first suppose that both of such edges are non-propagating. Let us show that the node $\ts_{p+2}$  (resp.  $\ts_{p-2}$) in the horizontal line through  $\ts_p $ is in $\check{H}$. Since $\check{H}$ is alternating one of the two  nodes adjacent to $\ts_p$ labeled by $\ts_{p+1}$  (resp.  $\ts_{p-1}$)  is in $\check{H}$.  Being $\ts_p$  a splitting node, one of the two snake paths corresponding to the two non-propagating edges that identify $\ts_p$ has to pass through one of these two nodes $\ts_{p+1}$  (resp.  $\ts_{p-1}$). If the node $\ts_{p+2}$ (resp. $\ts_{p-2}$) in the horizontal line through  $\ts_p $ is missing then  either we have a fork against the fact that $\check{H}$ is alternating or the two non-propagating edges that identify $\ts_p$ are on the same face of $d_{\check{e}}$ against the fact that $\ts_p$ is a splitting node. Furthermore, either $\ts_{p+2}$ (resp. $\ts_{p-2}$) is in $\rho$ or it is a splitting node collinear with $\ts_p$ and in the same connected component and so we can apply the argumentation given for this case above. On the other hand, it is easy to see that if $\ts_p$ is a splitting node then $\ts_{p+1}$ can not.

In the other case, when one of the two edges through the splitting node $\ts_p$ is propagating, then $e$ is necessarily non-propagating, without loss of generality suppose it is on the south face. By definition, the propagating edge has one node after (resp. before) $e$ so the corresponding snake path must contain all the nodes $\ts_{p+2l}$  (resp. $\ts_{p-2}$)  on the left (resp. on the right) of $\rho$ and in $\rho$ and in the same horizontal line, so all these nodes are in $SP_N(e)$.  On the other hand, all such nodes are in $SP_S(e)$ because all the edges on the south face of $d_{\check{e}}$ on the left of $e$ are non- propagating and below $e$.

Being $\ts_p$ a splitting node in $K_\rho$ with $\check{H}$ alternating, the leftmost node of $\rho $ is labeled by $\ts_{p+2l}$ and it can be depicted collinear with $\ts_p$, so they are all in the same horizontal line. Moreover the set of splitting nodes is :
\begin{eqnarray}
\mbox{$\{\ts_{1}, \ts_{3},\ldots, \ts_{*}\}$, where $\ts_*\in \{\ts_{h-1}, \ts_{h-2}\}$} \label{eq:splitting}  \\
\mbox{(resp. $\{\ts_{n+1}, \ts_{n-1}, \ldots, \ts_{+}\}$, where $\ts_{+}\in \{\ts_{k}, \ts_{k+1}\}$).} \label{eq:splitting1}
\end{eqnarray}

\item[(a.3)] If $e=(k,k')$ is a vertical edge, then $\rho=\gamma_{\check{k}}$ is the degenerate empty path and in $\check{H}$ there are no nodes labeled $k-1,k$. So in $\check{H}$ the elements labeled from $1$ to $k-2$ and from $k+1$ to $n+1$ are in different connected components. As we proved above, the splitting nodes inside the same connected can be depicted in the same horizontal line. 

\item[(b)] If $e$ is a decorated loop, then $H_{e}$ is a complete horizontal region and the result is obvious. Suppose that $e=(h,k)$ is decorated with at least a $\tn$ (resp. $\tb$). The set of splitting nodes is described in \eqref{eq:splitting} (resp. \eqref{eq:splitting1}). By (H2) and (H3) of Definition~\ref{def-He}, $H_{e}$ contains pairs of nodes labeled by $\ts_1, \ts_3,\ldots, \ts_{*}$, where $\ts_*\in \{\ts_{h-1}, \ts_{h-2}\}$ (resp. $\ts_{n+1}, \ts_{n-1},\ldots, \ts_{+}$, where $\ts_{+}\in \{\ts_{k}, \ts_{k+1}\}$), in its north and south border, which prove the first part of the statement.   

The heap $H_{e}$ contains at least a fork $f$ and $\gamma(H_{e},f)$ starts and ends in nodes  labeled $\ts_x$ or $\ts_{y}$, according to Definition~\ref{def-hk}. Depending on the positions of $h$ and $k$ and on the highest decoration in $e$, the path $\gamma(H_{e},f)$ first reaches the left or the right side of $H_e$, crosses the highest fork of $H_{e}$ and continues on the other direction. In doing so, $\gamma(H_{e},f)$ necessary passes through both nodes $\ts_x$ and $\ts_{y}$, hence crosses entirely $\rho$ inside the top horizontal region of $H_e$. This provides a first copy of $\rho$ in the north border of $H_{e}$, (see the red paths inside $H_{(5,12)}$ and $H_{(10',7')}$ in Figure~\ref{Surj-1}). 
Furthermore, the path $\gamma(H_{e},f)$ continues going back and forth between  the first to the last column of the $H_e$ several times depending on the number of decorations, and ends in a node in the south border. Once again, $\gamma(H_{e},f)$ passes from a node labeled $\ts_x$ to one $\ts_{y}$ in the lowest horizontal region of $H_e$, and this provides a copy of $\rho$ in the south boarder of $H_e$. A non trivial example is depicted in Figure~\ref{Surj-2}. The two red paths in $H_{(1,9')}$ represents the two occurrences of $\rho$ in the north and south border of $H_e$. This proves the second part of the statement.
\end{itemize}
\end{proof}

\begin{Remark} In the hypotheses of the previous lemma, if there is not a unique decorated loop in $d$, then $\dc$ contains a decorated loop too, and so from Proposition~\ref{prop:E}, it follows that $\Hc$ contains at least a complete horizontal subheap. In particular, this yields that in $\Hc$ there are two distinct sets of nodes labeled $\{\ts_1, \ts_3,\ldots, \ts_{n+1}\}$ depicted in the same horizontal line. For the purpose of the next definition we will call call these two sets of nodes splitting nodes as well.
\end{Remark}

The previous results allow the following definition.

\begin{Definition}\label{def:H'H}
 Let $d$ be a $\mathcal{A}$-diagram with an edge $e=(h,k)$ decorated with at least either a $\tn$ or a $\tb$ and let $\Hc \in (\ALT)$ such that $\mathtt{dec}(\Hc)=\dc$. 
 We define $H(d,e)$ to be the heap obtained by $H_{e}$ and $\Hc$, by attaching:
\begin{itemize}
\item to the nodes on the north border of $H_{e}$ labeled as the splitting nodes, the subheap of $\Hc$ made of the nodes above the splitting nodes;
\item to the nodes on the south border of $H_{e}$ labeled as the splitting nodes, the subheap of $\Hc$ made of the nodes below the splitting nodes;
\item to the nodes on the north border of $H_{e}$ labeled as those of $\rho$, the subheap of $\Hc$ made of the nodes above or connected to $\rho$ if $\rho$ is non-degenerate;
\item to the nodes on the south border of $H_{e}$ labeled as those of $\rho$, the subheap of $\Hc$ made of the nodes below or connected to $\rho$ if $\rho$ is non-degenerate;
\end{itemize}
and by copying the connected components of $\Hc$ not containing splitting nodes.
\end{Definition}

In Figure~\ref{Construction-H-1}, we give an example of the construction of $H(d,e)$ with $e=(3,5')$, based on the diagrams in Figure~\ref{Splitting-nodes}. More complete examples are shown in Figures~\ref{Surj-1} and \ref{Surj-2}.
\begin{figure}[h]
	\centering
	\includegraphics[width=0.9\linewidth]{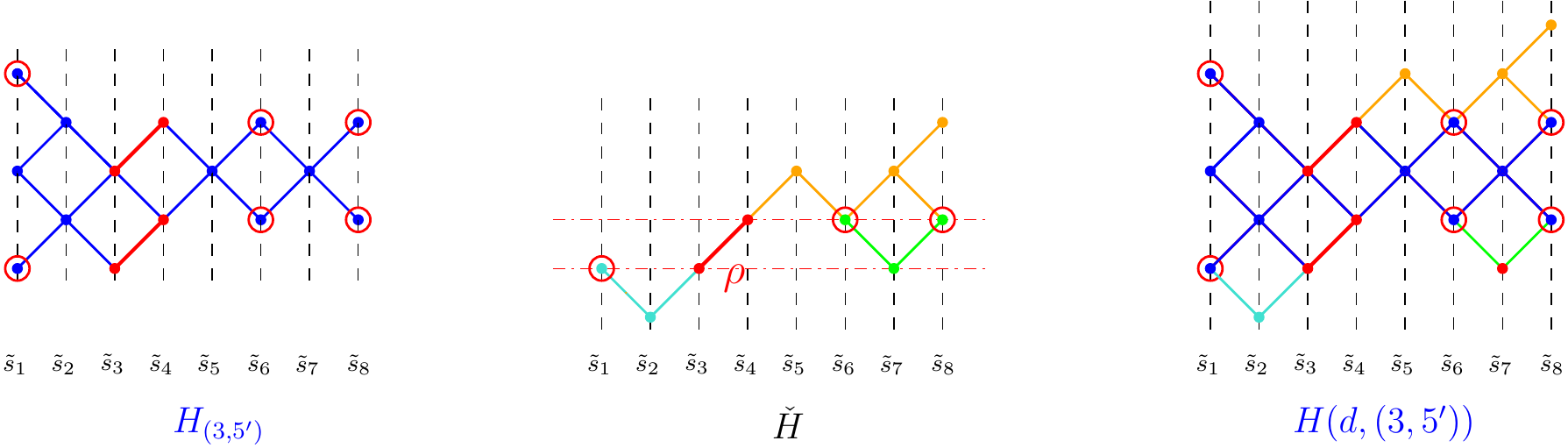}
	\caption{How obtain $H(d,e)$ from $H_e$ and $\Hc$.}
	\label{Construction-H-1}
\end{figure}

\begin{Theorem}\label{teosuralt}
The map $\mathtt{dec}_A:(\ALT) \longrightarrow \mathcal{A}$-diagrams is surjective.
\end{Theorem}

\begin{proof} We will proceed by induction on the number $t$ of decorated edges with at least a $\tn$ or $\tb$ of a $\mathcal{A}$-diagram.
If $t=0$ the result is proved in Lemma~\ref{lem:casobase}.
So, suppose that $d$ is a $\mathcal{A}$-diagram with $t>0$ decorated edges  with at least a $\tn$ or $\tb$, and $e$ be one of such edges. The diagram $d_{\check{e}}$ is a $\mathcal{A}$-diagram with $t-1$ decorated edges. By induction, there exists a heap $\check{H} \in (\ALT)$ such that ${\mathtt{dec}}(\check{H})=d_{\check{e}}$.  Consider the heap $H(d,e)$ constructed by using  Definition \ref{def:H'H}. The heap $H(d,e)$ is in $(\ALT)$ since it is obtained by attaching to the alternating heap $H_{e}$ alternating subheaps in the splitting nodes. From the inductive hypothesis and properties (a) and (b) listed before Definition~\ref{def-He} it follows that $\mathtt{dec}(H(d,e))=d$.
\end{proof}

\begin{figure}[h]
	\centering
	\includegraphics[width=0.9\linewidth]{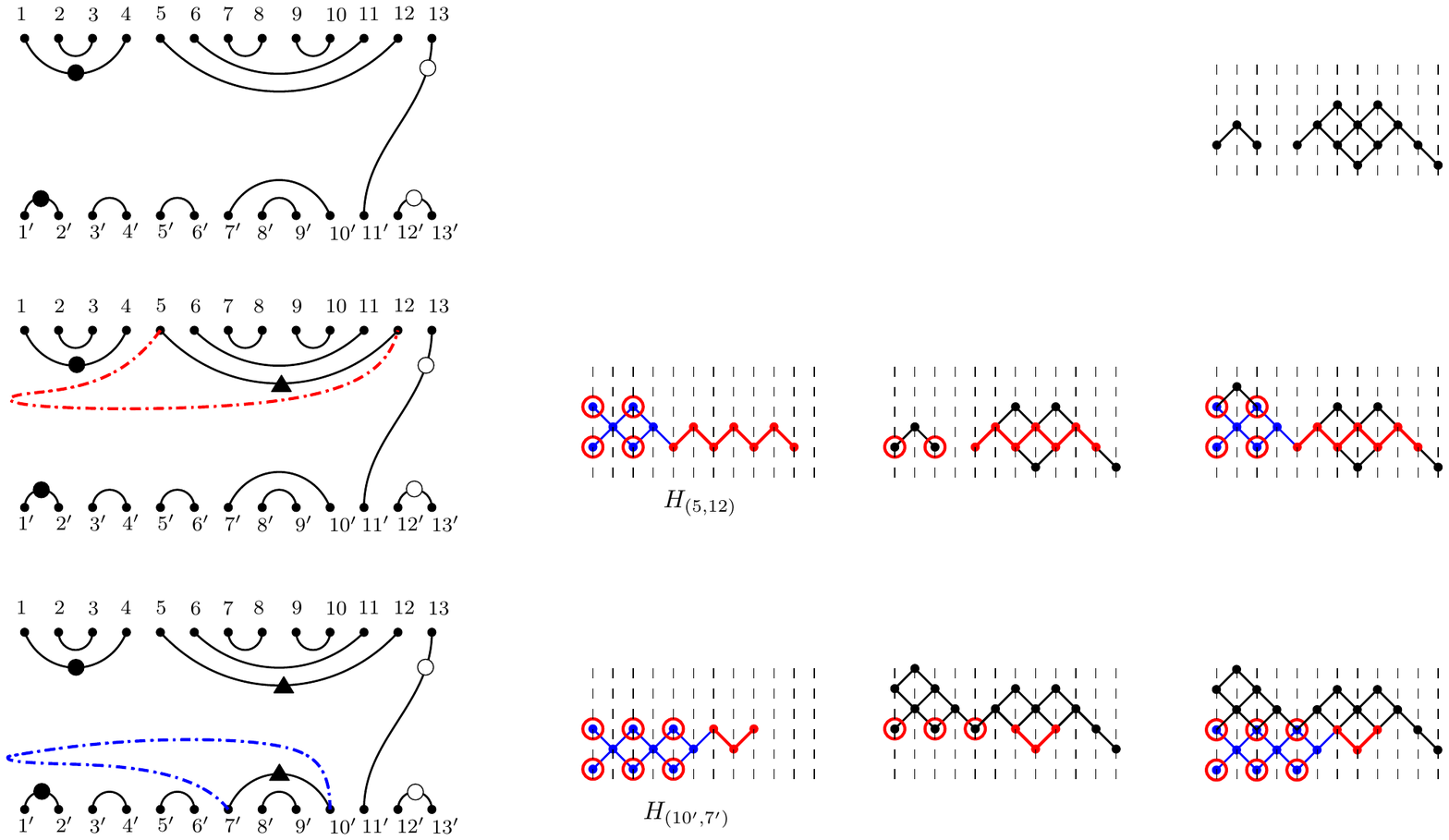}
	\caption{}
	\label{Surj-1}
\end{figure}

\begin{Theorem}\label{teosurpeak}
The map $\mathtt{dec}_P: (\LP)\cup (\RP) \cup (\LRPZ) \longrightarrow \mathcal{P}$-diagrams is surjective.
\end{Theorem}
\begin{proof}
Consider a decorated $\mathcal{P}$-diagram $d$ starting  with $j_\ell-1$ vertical edges and no $\tb$ decoration on the edge $(n+2,(n+2)')$ if it exists.
The diagram $d_{\check{(1,1')}}$, obtained by $d$ by erasing the $\tn$ on the edge $(1,1')$, is a decorated $\mathcal{A}$-diagram, so by Theorem \ref{teosuralt} there exists an alternating heap $\Hc$ such that $\mathtt{dec}(\Hc)=d_{\check{(1,1')}}$. Note that $\Hc$ starts with a single node $s_{j_\ell}$. 
There are two possibilities: if there is at most a node $\ts_{\ell+1}$ in $\Hc$, then duplicate the node $\ts_{\ell}$ and place it in the vertical opposite position with respect to $\ts_{\ell+1}$; if there are two nodes $\ts_{\ell+1}$, then duplicate $\ts_{\ell}$ and attach the original $\ts_{\ell}$ to the top node $\ts_{\ell+1}$ and the new node $\ts_{\ell}$ to the bottom node $\ts_{\ell+1}$. In both cases define $H$ by attaching a left peak starting from $\ts_1$ ending with the two $\ts_{j_\ell}$ to the two duplicated nodes of $\Hc$.

Clearly the so obtained heap $H$ is in $(\LP)$ and $\mathtt{dec}_P(H)=d$. A symmetric argument works for $\mathcal{P}$-diagrams finishing with $j_r-1$ vertical edges and $(n+2,(n+2)')$ decorated with a unique $\tb$, in this case $H$ will be in $(\RP)$. Finally, combining the two previous arguments we settle the remaining ``left-right'' $\mathcal{P}$-diagrams.
\end{proof}

\begin{Theorem}\label{teosurzz}
The map $\mathtt{dec}_Z: (\ZZ)\cup (\PZZ) \longrightarrow \mathcal{Z}$-diagrams is surjective.
\end{Theorem}

\begin{proof}
Let $d$ be a $\mathcal{Z}$-diagram with a non-propagating edge $(i,i+1)$ on the north face, $(j',(j+1)')$ on the south face and with $p$ $\tn$ decorations on the edge $(1,1')$ and $q$ $\tb$ decorations on $(n+2,(n+2)')$. We know that $q=p$ or $q=p\pm 1$ and that their vertical positions alternate. If the highest decoration is: 
\begin{itemize}
    \item $\tn$ then set $H_{start}:=H(\ts_i \ts_{i-1}\cdots \ts_1)$;
    \item $\tb$ then set $H_{start}:=H(\ts_i \ts_{i+1}\cdots \ts_{n+1})$.
\end{itemize}
If the lowest decoration is:
\begin{itemize}
    \item $\tn$ then set $H_{end}:=H(\ts_1 \ts_{2}\cdots \ts_j)$;
    \item $\tb$ then set $H_{end}:=H(\ts_{n+1} \ts_{n}\cdots \ts_{j})$.
\end{itemize}
Now connect $H_{start}$ with $H_{end}$ by the unique down chain made of descending segments going back and forth from column $1$ to column $n+1$ or viceversa in a way that the final heap $H$ is a zigzag or a pseudo-zigzag  having $p$ nodes labeled by $\ts_1$ and $q$ nodes labeled  by $\ts_{n+1}$. 
It is easy to see that $\mathtt{dec}_Z$ is $d$. 
\end{proof}

\begin{figure}[h]
	\centering
	\includegraphics[width=0.9\linewidth]{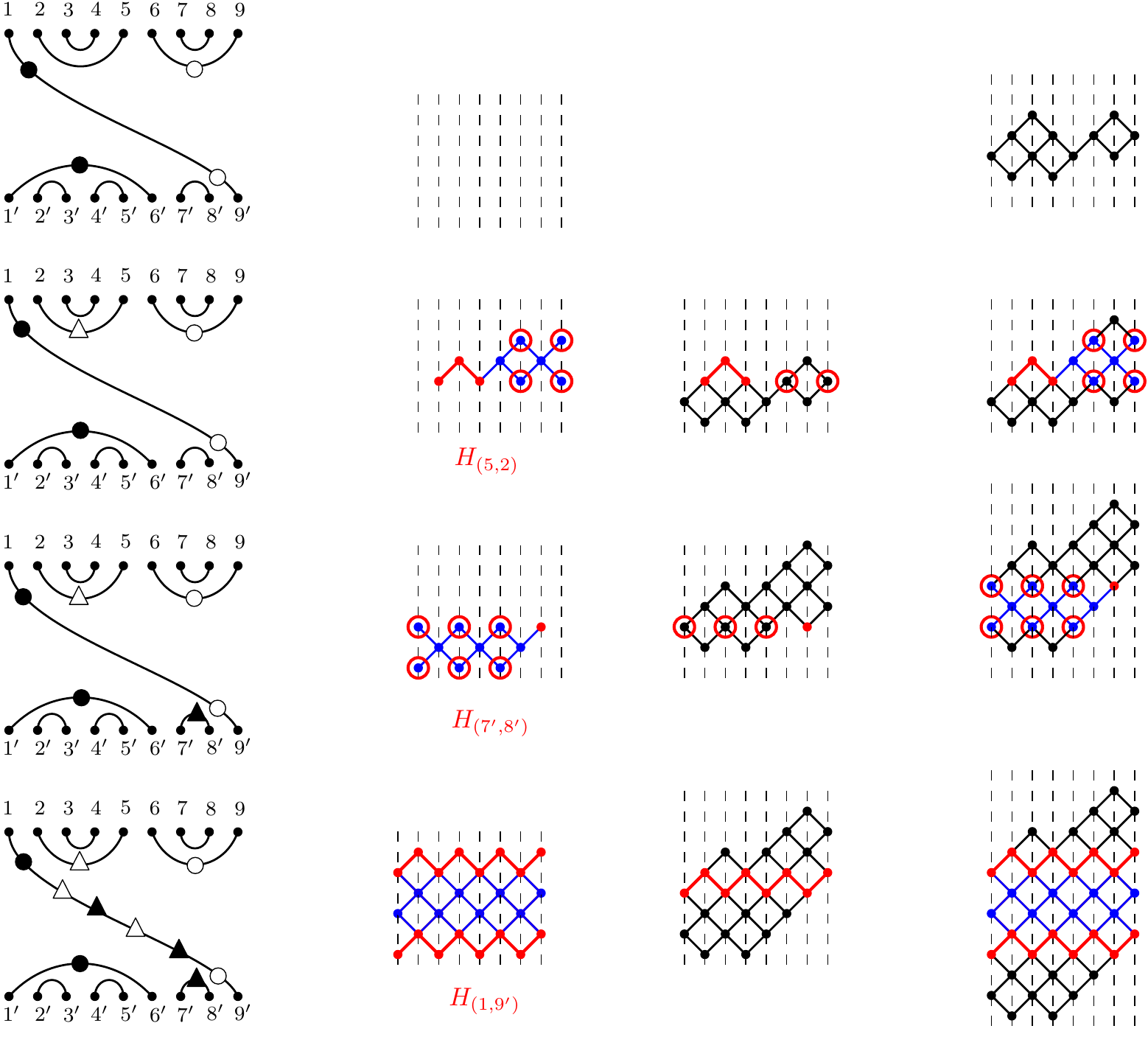}
	\caption{}
	\label{Surj-2}
\end{figure}


\end{document}